\documentclass[12pt]{amsart}
\usepackage{amssymb}
\usepackage{graphicx}
\usepackage{enumerate}
\usepackage[all]{xy}
\usepackage{mathtools} 
\usepackage[hidelinks]{hyperref}
\usepackage[margin=1in]{geometry}

\newtheorem{theorem}{Theorem}[section]
\newtheorem{corollary}[theorem]{Corollary}
\newtheorem{lemma}[theorem]{Lemma}
\newtheorem{proposition}[theorem]{Proposition}

\theoremstyle{definition}
\newtheorem{definition}[theorem]{Definition}
\newtheorem{warning}[theorem]{Warning}
\newtheorem{remark}[theorem]{Remark}
\newtheorem{example}[theorem]{Example}
\newtheorem{convention}[theorem]{Convention}

\numberwithin{equation}{section}

\DeclareMathOperator{\Alex}{Alex}
\DeclareMathOperator{\Ar}{Ar}
\DeclareMathOperator{\const}{Const}
\DeclareMathOperator{\id}{id}
\DeclareMathOperator{\im}{im}
\DeclareMathOperator{\inv}{inv}
\DeclareMathOperator{\Int}{Int}
\DeclareMathOperator{\op}{op}
\DeclareMathOperator{\Path}{Path}
\DeclareMathOperator{\Quiv}{Quiv}
\DeclareMathOperator{\Spin}{Spin}
\DeclareMathOperator{\un}{un}

\newcommand{\A}{\mathcal{A}}
\newcommand{\B}{\mathcal{B}}
\newcommand{\CL}[1]{\mathcal{CL}_{#1}}
\newcommand{\CLmonomial}{\prod_{i\in\CL{\x,\y}} U_i^{r_i} }
\newcommand{\de}{\partial}
\newcommand{\dumline}{j}
\newcommand{\F}{\mathbb{F}}
\newcommand{\gloneone}{\mathcal{U}_q(\mathfrak{gl}(1|1))}
\newcommand{\I}{[0,1]}
\newcommand{\ib}{\mathbf{i}}
\newcommand{\Ib}{\mathbf{I}}
\newcommand{\IdemRing}{\mathbb{I}}
\newcommand{\jb}{\mathbf{j}}
\newcommand{\Jb}{\mathbf{J}}
\newcommand{\kb}{\mathbf{k}}
\newcommand{\lda}{\lambda}
\newcommand{\m}{\mathfrak{m}}
\newcommand{\mb}{\mathbb}
\newcommand{\mc}{\mathcal}
\newcommand{\pvec}[1]{\vec{#1}\mkern2mu\vphantom{#1}}
\newcommand{\ring}{\Bbbk}
\newcommand{\ru}{{\underline{r}}}
\newcommand{\sa}[2]{\mathcal{A}(#1, #2)}
\newcommand{\sac}[3]{\mathcal{A}(#1, #2, #3)}
\newcommand{\Sb}{\mathbf{S}}
\newcommand{\Sc}{\mathcal{S}}
\newcommand{\set}[1]{\left\{#1\right\}}
\newcommand{\sm}{\setminus}
\newcommand{\Spinc}{\Spin^c}
\newcommand{\Tb}{\mathbf{T}}
\newcommand{\td}{\widetilde}
\newcommand{\To}{\longrightarrow}
\newcommand{\tsa}[2]{\widetilde{\mathcal{A}}(#1, #2)}
\newcommand{\tsac}[3]{\widetilde{\mathcal{A}}(#1, #2, #3)}
\newcommand{\vsimeq}{\rotatebox[origin=c]{-90}{\footnotesize $\backsimeq$}}
\newcommand{\vv}[3]{\left(\begin{matrix} #1\\#2\end{matrix}\right)_{#3}}
\newcommand{\x}{\mathbf{x}}
\newcommand{\Xb}{\mathbf{X}}
\newcommand{\y}{\mathbf{y}}
\newcommand{\z}{\mathbf{z}}
\newcommand{\Z}{\mathbb{Z}}
\newcommand{\Zc}{\mathcal{Z}}

\renewcommand{\emptyset}{\varnothing}

\renewcommand{\subseteq}{\subset}

\renewcommand{\theta}{\vartheta}

\DeclarePairedDelimiter{\floor}{\lfloor}{\rfloor}
\DeclarePairedDelimiter{\ceil}{\lceil}{\rceil}

\newcommand{\secQuiverAlgs}{Section 2.1}

\newcommand{\propQuiverAlgUniversalProp}{Proposition 2.6}

\newcommand{\defOSzStyleDef}{Definition 2.11}

\newcommand{\remNinetyDegRot}{Remark 2.13}

\newcommand{\secGraphicalInterp}{Section 2.3}

\newcommand{\secEqvofdescriptions}{Section 2.4}

\newcommand{\defRecursive}{Definition 2.28}

\newcommand{\secOSzgradings}{Section 3.3}

\newcommand{\corOSzQuiverEquivDG}{Corollary 3.14}

\newcommand{\secOSzTruncatedAlgs}{Section 3.4}

\newcommand{\defTruncatedOSzAlgs}{Definition 3.16}

\newcommand{\propQuartumNonDatur}{Proposition 4.9}

\newcommand{\corExplicitQuiverGensForMonomials}{Corollary 4.12}

\newcommand{\secOSzSplittingTheorem}{Section 4.3}

\newcommand{\defCrossedLinesAlg}{Definition 4.15}

\newcommand{\corIBItoTensorProduct}{Corollary 4.16}

\newcommand{\rmkSplittingBwun}{Remark 4.17}

\newcommand{\secOSzSymmetries}{Section 4.5}

\newcommand{\lemOSzGenIntHomology}{Lemma 5.3}

\newcommand{\thmOSzHomology}{Theorem 5.4}

\newcommand{\secFormality}{Section 5.2}

\newcommand{\thmUntruncatedFormality}{5.10}
\newcommand{\thmRightTruncationFormality}{5.13}
\newcommand{\thmLeftTruncationFormality}{5.14}
\newcommand{\thmDoubleTruncationFormality}{5.17}

\newcommand{\AppendixA}{Appendix A}

\newcommand{\secAlgsAndCats}{Section A.3 in \AppendixA}

\newcommand{\lemOrthogonalIdempotents}{Lemma A.17 in \AppendixA}

\newcommand{\defDGCatQI}{Definition A.20 in \AppendixA}

\begin{document}

\author[Andrew Manion]{Andrew Manion}
\thanks {AM was supported by an NSF MSPRF fellowship, grant number DMS-1502686.}
\address{Department of Mathematics, USC, 3620 S. Vermont Ave., Los Angeles, CA 90089}
\email{amanion@usc.edu}

\author[Marco Marengon]{Marco Marengon}
\address {Department of Mathematics, UCLA, 520 Portola Plaza, Los Angeles, CA 90095}
\email {marengon@math.ucla.edu}

\author[Michael Willis]{Michael Willis}
\thanks {MW was supported by the NSF grant DMS-1563615.}
\address {Department of Mathematics, UCLA, 520 Portola Plaza, Los Angeles, CA 90095}
\email {mike.willis@math.ucla.edu}

\title[Strands algebras and Ozsv{\'a}th--Szab{\'o}'s Kauffman-states functor]{Strands algebras and Ozsv{\'a}th--Szab{\'o}'s Kauffman-states functor}

\date{}

\begin{abstract} 
We define new differential graded algebras $\sac nk\Sc$ in the framework of Lipshitz--Ozsv{\'a}th--Thurston's and Zarev's strands algebras from bordered Floer homology. The algebras $\A(n,k,\Sc)$ are meant to be strands models for Ozsv{\'a}th--Szab{\'o}'s algebras $\B(n,k,\Sc)$; indeed, we exhibit a quasi-isomorphism from $\B(n,k,\Sc)$ to $\sac nk\Sc$. We also show how Ozsv{\'a}th--Szab{\'o}'s gradings on $\B(n,k,\Sc)$ arise naturally from the general framework of group-valued gradings on strands algebras.
\end{abstract}
\maketitle

\section{Introduction}
Heegaard Floer homology is a package of invariants for 3-manifolds and 4-manifolds introduced by Ozsv\'ath and Szab\'o \cite{HFOrig, PropsApps} that has proven to be particularly powerful in the last two decades. A variation \cite{OSzHFK,RasmussenThesis} of their construction, called knot Floer homology and abbreviated $HFK$, assigns a graded abelian group to a knot or link, and the Euler characteristic of this group recovers the Alexander polynomial. Knot Floer homology has many applications in knot theory; for example, it exactly characterizes elusive knot information like Seifert genus and fiberedness, for which the knot polynomials provide only incomplete bounds, and it leads to the definition of many interesting knot concordance invariants.

In the past ten years, there has been considerable interest in assigning Heegaard Floer invariants to surfaces and $3$-dimensional cobordisms between them. Lipshitz--Ozsv{\'a}th--Thurston's bordered Floer homology \cite{LOT} initiated this project; Zarev \cite{BSFH} introduced a generalization known as bordered sutured Floer homology. If one views Heegaard Floer homology from the perspective of topological quantum field theories (TQFTs), then bordered Floer homology begins the investigation of Heegaard Floer homology as an ``extended'' TQFT. Extensions of TQFTs have been of particular interest since Lurie's proof \cite{Lurie} of the Baez--Dolan cobordism hypothesis classifying fully extended TQFTs.

Bordered sutured Floer homology assigns an invariant to a surface $F$ by first choosing a combinatorial representation of $F$, called an ``arc diagram'' by Zarev. Arc diagrams are a special case of what are known as ``chord diagrams'' in e.g.~\cite{CD} (see Definition \ref{def:chord diagram}). Chord diagrams may have linear and/or circular ``backbones'' (see Figure \ref{fig:MotivatingHD}); arc diagrams are the same as chord diagrams with no circular backbones. To an arc diagram $\Zc$ representing a surface $F(\Zc)$, bordered sutured Floer homology associates a differential graded (dg) algebra $\A(\Zc)$, called the bordered strands algebra of $\Zc$ because it can be visualized by pictures of strands intersecting in $[0,1] \times \Zc$. Auroux~\cite{Auroux} has shown that $\A(\Zc)$ is closely related to Fukaya categories of symmetric powers of $F(\Zc)$, in line with the original definition of Heegaard Floer homology.

More recently, Ozsv{\'a}th--Szab{\'o} \cite{OSzNew,OSzNewer,OSzHolo,OSzPong} have used the ideas of bordered Floer homology to define a new algorithmic method for computing $HFK$ by decomposing a knot into tangles. Their theory has striking computational properties \cite{HFKCalc}, categorifies aspects of the representation theory of $\gloneone$ \cite{ManionDecat}, and has surprising connections with other such categorifications \cite{ManionKS}. We will refer to their theory as the \emph{Kauffman-states functor}, since Kauffman states for a knot or tangle projection (equivalently, spanning trees of the Tait graph) play a prominent role.

\begin{figure}
\includegraphics{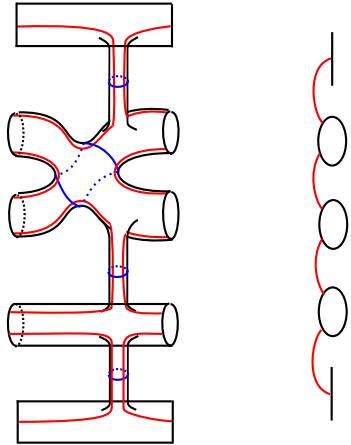}
\caption{The Heegaard diagram motivating the Kauffman-states functor, and the chord diagram implied by it. Such a chord diagram has $2$ linear backbones and $3$ circular backbones (all drawn in black).}
\label{fig:MotivatingHD}
\end{figure}

To a tangle diagram, the Kauffman-states functor assigns a bimodule whose definition is motivated by holomorphic curve counting as in bordered Floer homology. However, the dg algebras $\B(n,\Sc)$ over which the bimodule is defined are not among Zarev's bordered strands algebras. Indeed, for a single crossing, Ozsv{\'a}th--Szab{\'o} count curves in a particular Heegaard diagram from which a chord diagram $\Zc(n)$ can be inferred (see Figure~\ref{fig:MotivatingHD}), but some of the backbones of $\Zc(n)$ are circular rather than linear, so $\Zc(n)$ is not an arc diagram and Zarev's construction does not apply. 

We begin by defining a reasonable candidate $\A(n,\Sc)$ for the bordered strands algebra of the chord diagram $\Zc(n)$ in question, with a diagrammatic interpretation in terms of intersecting strands as usual (the data $\Sc$ encodes orientations on tangle endpoints and will be described below in Section~\ref{sec:MMW1Review}). The algebra $\A(n,\Sc)$ is larger than $\B(n,\Sc)$, with a more elaborate differential. See e.g.~Figure~\ref{fig:kstrand examples} for an illustration.
The dg algebras $\A(n,\Sc)$ and $\B(n,\Sc)$ are both direct sums of dg algebras $\A(n,k,\Sc)$ and $\B(n,k,\Sc)$ for $0 \leq k \leq n$.
Like $\B(n,k,\Sc)$, the strands algebra $\A(n,k,\Sc)$ comes with a Maslov grading and various Alexander multi-gradings.

The bordered strands algebra $\A(n,k,\Sc)$ and Ozsv\'ath--Szab\'o's algebra $\B(n,k,\Sc)$ are in fact closely related to each other. Using the generators-and-relations description of $\B(n,k,\Sc)$ from \cite{MMW1}, we define a dg algebra homomorphism $\Phi: \B(n,k,\Sc) \to \A(n,k,\Sc)$ and prove the following result.
\begin{theorem}\label{thm:IntroQuasiIso}
The map $\Phi: \B(n,k,\Sc) \to \A(n,k,\Sc)$ is a quasi-isomorphism.
\end{theorem}

Since we computed the homology of $\B(n,k,\Sc)$ in \cite{MMW1}, we deduce the homology of $\A(n,k,\Sc)$ from Theorem \ref{thm:IntroQuasiIso}.
\begin{corollary}
Applying $\Phi$ to the basis for $H_*(\B(n,k,\Sc))$ given in Theorem \ref{thm:ReviewOSzHomology} yields a basis for $H_*(\A(n,k,\Sc))$.
\end{corollary}
We can also transfer the formality properties for $\B(n,k,\Sc)$ proved in \cite{MMW1} to the quasi-isomorphic algebras $\A(n,k,\Sc)$.
\begin{corollary}\label{cor:IntroFormality}
The dg algebra $\A(n,k,\Sc)$ is formal if and only if $\Sc = \varnothing$ or $k \in \{0,n,n+1\}$.
\end{corollary}

We describe the gradings on $\A(n,k,\Sc)$ combinatorially in Definition~\ref{def:gradings}; their definition depends on $\Sc$. However, bordered strands algebras $\A(\Zc)$ typically have gradings by nonabelian groups $G'(\Zc)$ and $G(\Zc)$ which do not see the dependence on $\Sc$. We define these gradings in our setting too (both groups end up being abelian) and show how they are related to the combinatorial gradings.
\begin{theorem}\label{thm:IntroGradings}
Given $\Sc$, we have an isomorphism
\[
\Theta_{\Sc}: G'(\Zc(n)) \xrightarrow{\cong} \Z \oplus \Z^{2n}
\]
such that for a homogeneous element $a$ of $\A(n,k,\Sc)$, the first component of $\Theta_{\Sc}(\deg'(a))$ is the Maslov degree of $a$ and the rest of the components form the unrefined Alexander multi-degree of $a$. Similarly, we have an isomorphism
\[
\Theta_{\Sc}: G(\Zc(n)) \xrightarrow{\cong} \Z \oplus \left(\textstyle{\frac12}\Z\right)^n
\]
whose first component recovers the Maslov grading and whose second component recovers the refined Alexander multi-grading.
\end{theorem}
Theorem~\ref{thm:IntroGradings} helps to explain the appearance of the data $\Sc$ in the algebras $\A(n,k,\Sc)$ and $\B(n,k,\Sc)$, since the orientation data for tangle endpoints is not visible from the chord diagram $\Zc(n)$. While the gradings by $G'(\Zc(n))$ and $G(\Zc(n))$ are independent of this orientation data, their interpretation as standard Maslov and Alexander gradings is noncanonical and its choice forces a choice of $\Sc$. 

We note that this noncanonicity stems from the condition ``$j \equiv \varepsilon(\alpha)$ mod $1$'' in Lipshitz--Ozsv{\'a}th--Thurston's definition of their nonabelian gradings (see \cite[Definition 3.33]{LOT}). Thus, we have a new motivation for this somewhat mysterious-seeming condition, since in the end we do want $\A(n,k,\Sc)$ and $\B(n,k,\Sc)$ to depend on orientations. Note, however, that $\A(n,k,\Sc)$ and $\B(n,k,\Sc)$ depend on $\Sc$ for more than just their gradings; $\Sc$ also determines whether certain additional generators are allowed in the algebras.

Along with $\B(n,k,\Sc)$, it is natural to consider certain idempotent-truncated algebras $\B_r(n,k,\Sc)$, $\B_l(n,k,\Sc)$, and $\B'(n,k,\Sc)$. Each has an associated chord diagram $\Zc_r(n)$, $\Zc_l(n)$, or $\Zc'(n)$, and we define the corresponding strands algebras $\A_r(n,k,\Sc)$, $\A_l(n,k,\Sc)$, and $\A'(n,k,\Sc)$ (they are idempotent truncations of $\A(n,k,\Sc)$). We prove that $\Phi$ gives a quasi-isomorphism between the truncated algebras as well, deducing results about homology and formality for the truncated algebras from the analogous results in \cite{MMW1}. 

Finally, we define symmetries on the strands algebras $\A(n,k,\Sc)$ analogous to Ozsv{\'a}th--Szab{\'o}'s symmetries $\mc R$ and $o$ on the algebras $\B(n,k,\Sc)$, and we show that $\Phi$ preserves these symmetries. The symmetries on $\A(n,k,\Sc)$ have an appealing visual interpretation as symmetries of the surface $[0,1] \times \Zc(n)$ on which the strands pictures are drawn.

\subsection*{Context and motivation}

This paper is a sequel to \cite{MMW1}, which lays much of the necessary groundwork for our main results here. A third paper \cite{MMW3} in the series is planned, in which we define bimodules over $\A(n,k,\Sc)$ for crossings and prove that they are compatible with Ozsv{\'a}th--Szab{\'o}'s bimodules in an appropriate sense.

We view our constructions as evidence for the existence of a generalized theory of bordered sutured Floer homology, allowing chord diagrams with circular backbones and correspondingly generalized Heegaard diagrams. Defining Heegaard Floer homology analytically in this level of generality has not been attempted, and appears to be quite difficult. However, various recent constructions should be special cases of such a generalized theory, including Lipshitz--Ozsv{\'a}th--Thurston's work in progress on a bordered $HF^-$ theory for $3$-manifolds with torus boundary \cite{LOTMinus} as well as Zibrowius' constructions in \cite{Claudius}. Our work should enable Ozsv{\'a}th--Szab{\'o}'s Kauffman-states functor to be directly compared with such a generalized theory once it exists, unifying the Kauffman-states functor with the rest of bordered Floer homology.

In \cite{ManionRouquier}, Rapha{\"e}l Rouquier and the first named author will define generalized strands algebras $\A(\Zc)$, including $\A(n,k,\Sc)$ as a special case. These algebras are candidates for the algebras appearing in a generalized bordered sutured theory, possibly after deformation as in \cite{LOTMinus}. The constructions of \cite{ManionRouquier} will also give $\A(\Zc)$ the structure of a $2$-representation of Khovanov's categorified $\mathcal{U}_q^+(\mathfrak{gl}(1|1))$. Thus, together with \cite{ManionRouquier}, this paper fills in (the positive half of) a missing piece from the discussion of \cite{ManionDecat}. While \cite{ManionDecat} shows that the bimodules from the Kauffman-states functor categorify $\gloneone$-intertwining maps between representations, no candidate was offered for the categorification of the $\gloneone$ actions on the representations. This paper allows us to replace Ozsv{\'a}th--Szab{\'o}'s algebra $\B(n,k,\Sc)$ (when desired) with a strands algebra $\A(n,k,\Sc)$ on which a categorified quantum-group action is given in \cite{ManionRouquier}. Alternatively, one could directly define a $2$-action on $\B(n,k,\Sc)$, and show that it is compatible with the $2$-action on $\A(n,k,\Sc)$ via $\Phi$; the first named author plans to do this once the general framework of \cite{ManionRouquier} is available.

This paper, along with \cite{MMW1}, only discusses the algebras coming from Ozsv{\'a}th--Szab{\'o}'s first paper \cite{OSzNew} on the Kauffman-states functor. A variant of these algebras was introduced in \cite{OSzNewer}, and further variants will be defined in \cite{OSzHolo,OSzPong}. It would be very interesting to find analogues of the results of this paper for any of these algebras, especially the ``Pong algebra'' from \cite{OSzPong}. As with Lipshitz--Ozsv{\'a}th--Thurston's constructions in \cite{LOTMinus}, the Pong algebra may give further insight into the algebraic structure required for a generalized bordered sutured theory as mentioned above.

For the reasons discussed in \cite{MMW1}, we will follow the standard conventions in bordered Floer homology and work over $\F_2$. While the bordered strands algebras have not been defined over $\Z$ in general, to the authors' knowledge, it is plausible that the constructions in this paper could be done over $\Z$. However, it is likely that an analytic generalization of bordered sutured Floer homology would be considerably more difficult over $\Z$ than over $\F_2$.

\subsection*{Organization}

We start with a brief review of some essential definitions and results from \cite{MMW1} in Section~\ref{sec:MMW1Review}. For motivation, we discuss chord diagrams and sutured surfaces in Section~\ref{sec:Chord diagrams}, giving generalized versions of Zarev's definitions.

In Section~\ref{sec:StrandsAlgDef}, we define the strands algebras $\A(n,k,\Sc)$ and give illustrations. Section~\ref{sec:StrandStructure} proves some properties that will be useful both here and in \cite{MMW3}; in particular, we give an explicit calculus for products and differentials of certain basis elements of $\A(n,k,\Sc)$. In Section~\ref{sec:gradings} we discuss gradings and prove Theorem~\ref{thm:IntroGradings}; in Section~\ref{sec:StrandSymmetries}, we define symmetries on $\A(n,k,\Sc)$.

In Section~\ref{sec:Homology} we compute the homology of $\A(n,k,\varnothing)$. In Section~\ref{sec:qi} we define the map $\Phi$ from $\B(n,k,\Sc)$ to $\A(n,k,\Sc)$ and prove Theorem~\ref{thm:IntroQuasiIso} by induction on $|\Sc|$; the base case of the induction ($|\Sc| = 0$) follows from the computation of $H_*(\A(n,k,\emptyset))$ in Section~\ref{sec:Homology}.
Finally, in Section~\ref{sec:PhiSymmetries} we show that $\Phi$ preserves the algebra symmetries.

\subsection*{Acknowledgments}

The authors would like to thank Francis Bonahon, Ko Honda, Aaron Lauda, Robert Lipshitz, Ciprian Manolescu, Peter Ozsv{\'a}th, Rapha{\"e}l Rouquier, and Zolt{\'a}n Szab{\'o} for many useful conversations. The first named author would especially like to thank Zolt{\'a}n Szab{\'o} for teaching him about the Kauffman-states functor.
\section{Background on Ozsv{\'a}th--Szab{\'o}'s algebras}\label{sec:MMW1Review}

We begin with a brief review of some important terminology and results from \cite{OSzNew, MMW1}.
As in \cite[\AppendixA]{MMW1}, given a commutative ring $\ring$, we define a $\ring$-algebra to be a ring $\mc A$ equipped with a ring homomorphism $\ring \to \mc A$.
Given a quiver $\Gamma$ (i.e.~a finite directed graph, allowed to have loops and multi-edges), one has a path algebra $\Path(\Gamma)$ formally spanned over $\ring$ by paths in $\Gamma$, with multiplication given by concatenation. If $V$ is the vertex set of $\Gamma$, one can view $\Path(\Gamma)$ as an algebra over $\IdemRing = \ring^V$, the ring of functions from $V$ into $\ring$.
The homomorphism $\IdemRing \to \Path(\Gamma)$ sends the indicator function of a vertex $\x$ to the empty path $\Ib_\x$ based at $\x$. The composition $\ring \to \IdemRing \to \Path(\Gamma)$ has image in the center of $\Path(\Gamma)$.

Equivalently, one may work in terms of a $\ring$-linear category $\ring\Gamma$ whose set of objects is $V$; see \cite[\secQuiverAlgs{} and \AppendixA]{MMW1} for a detailed review of this algebraic framework. Hom-spaces in this category are given by $\Ib_\x \Path(\Gamma) \Ib_\y$ for $\x,\y \in V$, and we have a decomposition
\[
\Path(\Gamma) \cong \bigoplus_{\x, \y \in V} \Ib_\x \Path(\Gamma) \Ib_\y.
\]

If $\mc R$ is a subset of $\Path(\Gamma)$, we can also consider the quotient of $\Path(\Gamma)$ by the two-sided ideal generated by $\mc R$. We will call this quotient $\Quiv(\Gamma, \mc R)$; it is still an algebra over $\IdemRing$, and we can still view it as a $\ring$-linear category. Gradings and differentials on $\Path(\Gamma)$ and $\Quiv(\Gamma, \mc R)$ can be specified by defining them on the edges of $\Gamma$, as long as the relations are homogeneous cycles, so that we can consider dg algebras defined by quiver generators and relations.

\begin{convention}
In this paper, as in \cite{MMW1}, the interval $[a,b]$ will denote the set of integers $i$ with $a \leq i \leq b$.
\end{convention}

Given a subset $\Sc\subset[1,n]$, we now recall the definition of Ozsv{\'a}th--Szab{\'o}'s algebra $\B(n,k,\Sc)$ in the language of \cite{MMW1}.

\begin{definition}\label{def:ReviewVnk}
For $n \geq 0$ and $0 \leq k \leq n$, let $V(n,k)$ denote the set of $k$-element subsets $\x \subset [0,n]$. Elements of $V(n,k)$ will sometimes be called \emph{I-states}, following \cite[Section 3.1]{OSzNew}. Taking $\ring = \F_2$, let $\Ib(n,k) = \F_2^{V(n,k)}$. 

\end{definition}

Elements of $V(n,k)$ are visualized as in Figure \ref{fig:IStates}. Elements of $[0,n]$ are thought of as regions between $n$ parallel horizontal lines, including the two unbounded regions above and below the lines. An I-state $\x$ is drawn by placing a dot in each region corresponding to an element of $\x$. Ozsv\'ath and Szab\'o use a $90^\circ$-rotated visualization; see Remark \ref{rem:Orientation}.

\begin{figure}
\includegraphics[scale=0.5]{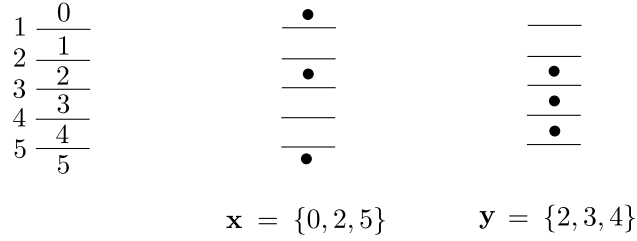}
\caption{Elements $\x$ and $\y$ of $V(5,3)$ viewed as dots occupying regions.  The left-most figure indicates the numbered labeling of the regions and the lines between them.}
\label{fig:IStates}
\end{figure}

\begin{definition}\label{def:ReviewGammanks}
The directed graph $\Gamma(n,k,\Sc)$ has set of vertices $V(n,k)$. It has the following edges:
\begin{itemize}
\item if $1 \leq i \leq n$ and $\x \cap \{i-1,i\} = \{i-1\}$, then $\Gamma(n,k,\Sc)$ has an edge from $\x$ to $(\x \sm\set{i-1})\cup\set{i}$ labeled $R_i$;
\item if $1 \leq i \leq n$ and $\x \cap \{i-1,i\} = \{i\}$, then $\Gamma(n,k,\Sc)$ has an edge from $\x$ to $(\x \sm\set{i})\cup\set{i-1}$ labeled $L_i$;
\item if $1 \leq i \leq n$ and $\x \in V(n,k)$, then $\Gamma(n,k,\Sc)$ has an edge from $\x$ to itself labeled $U_i$;
\item if $i \in \Sc$ and $\x \in V(n,k)$, then $\Gamma(n,k,\Sc)$ has an edge from $\x$ to itself labeled $C_i$.
\end{itemize}
\end{definition}

For each path $\gamma$ in $\Gamma(n,k,\Sc)$, we associate a noncommutative monomial $\mu(\gamma)$ in the letters $\{R_i, L_i, U_i, C_i\}$ by taking the labels of the edges of $\gamma$ in order. We extend $\mu$ additively to the path algebra of $\Gamma(n,k,\Sc)$.
\begin{definition}\label{def:ReviewGammaQuiverRels}
For $\x, \y \in V(n,k)$, let $\td{\mc R}_{\x, \y, \Sc}$ be the set of elements $a \in \Ib_\x \Path(\Gamma(n,k,\Sc)) \Ib_\y$ such that $\mu(a)$ is one of the following:
\begin{itemize}
\item $R_i U_j - U_j R_i$,\, $L_i U_j - U_j L_i$, or $U_i U_j - U_j U_i$ (the ``$U$ central relations'')
\item $R_i L_i - U_i$ or $L_i R_i - U_i$ (the ``loop relations'')
\item $R_i R_j - R_j R_i$,\, $L_i L_j - L_j L_i$, or $R_i L_j - L_j R_i$ for $|i-j|>1$ (the ``distant commutation relations'')
\item $R_i R_{i+1}$ or $L_{i+1} L_i$ (the ``two-line pass relations'')
\item $U_i$ when $a$ is represented by a loop $\gamma$ at a vertex $\x$ of $V(n,k)$ such that $\x \cap \{i-1,i\} = \varnothing$ (the ``$U$ vanishing relations")
\item $C_i^2$ (the ``$C$ vanishing relations")
\item $C_i A - A C_i$ for any label $A \in \{R_j, L_j, U_j, C_j\}$ (the ``$C$ central relations'').
\end{itemize}
\end{definition}

Let $\td{\mc R}_{\Sc} = \bigcup_{\x, \y \in V(n,k)} \td{\mc R}_{\x, \y, \Sc}$.
Let $\Quiv(\Gamma(n,k,\Sc), \td{\mc R}_{\Sc})$ denote the quotient of the path algebra of $\Gamma(n,k,\Sc)$ by the two-sided ideal generated by elements of $\td{\mc R}_{\Sc}$. Define a differential on $\Quiv(\Gamma(n,k,\Sc), \td{\mc R}_{\Sc})$ by declaring that $\partial(C_i) = U_i$.

We have a homological grading on $\Quiv(\Gamma(n,k,\Sc), \td{\mc R}_{\Sc})$ called the Maslov grading, as well as three related types of intrinsic gradings called Alexander gradings. We recall their definitions now.

\begin{definition}[{\cite[\secOSzgradings]{MMW1}}]\label{def:ReviewOSzGradings}
The gradings on $\Quiv(\Gamma(n,k,\Sc), \td{\mc R}_{\Sc})$ are defined as follows:
\begin{itemize}
\item Let $\{\tau_1, \ldots, \tau_n, \beta_1, \ldots, \beta_n\}$ denote the standard basis of $\Z^{2n}$. For an edge $\gamma$ of $\Gamma(n,k,\Sc)$, define the \emph{unrefined Alexander multi-degree} $w^{\un}(\gamma) \in (\Z)^{2n}$ to be
\begin{itemize}
\item[\textasteriskcentered] $w^{\un}(\gamma) = \tau_i$ if $\gamma$ has label $R_i$
\item[\textasteriskcentered] $w^{\un}(\gamma) = \beta_i$ if $\gamma$ has label $L_i$
\item[\textasteriskcentered] $w^{\un}(\gamma) = \tau_i + \beta_i$ if $\gamma$ has label $U_i$ or $C_i$.
\end{itemize}
Extend $w^{\un}$ additively to any path $\gamma \in \Path(\Gamma(n,k,\Sc))$.

\item Let $\{e_1,\ldots, e_n\}$ denote the standard basis of $\Z^n$. Define the \emph{refined Alexander multi-grading} on $\Quiv(\Gamma(n,k,\Sc), \td{\mc R}_{\Sc})$, a grading by $\left(\frac{1}{2}\Z\right)^n$, by applying the homomorphism $\Z^{2n} \to \left(\frac{1}{2}\Z\right)^n$ sending $\tau_i$ and $\beta_i$ to $\frac{e_i}{2}$ to the unrefined Alexander multi-degrees. For $a \in \Quiv(\Gamma(n,k,\Sc), \td{\mc R}_{\Sc})$ homogeneous, let $w(a)$ denote the refined Alexander multi-degree of $a$.
Explicitly, for an edge $\gamma$ of $\Gamma(n,k,\Sc)$, we have
\begin{itemize}
\item[\textasteriskcentered] $w(\gamma) = \frac{1}{2}e_i$ if $\gamma$ has label $R_i$ or $L_i$
\item[\textasteriskcentered] $w(\gamma) = e_i$ if $\gamma$ has label $U_i$ or $C_i$.
\end{itemize}
Let $w_i(a)$ denote the coefficient of $w(a)$ on the basis element $e_i$.

\item Define the \emph{single Alexander grading} on $\Quiv(\Gamma(n,k,\Sc), \td{\mc R}_{\Sc})$, a grading by $\frac{1}{2}\Z$, by applying the homomorphism $\left(\frac{1}{2}\Z\right)^n \to \frac{1}{2}\Z$ sending
\[
e_i \mapsto \begin{cases} 1 & \textrm{ if } i \notin \Sc \\ -1 & \textrm{ if } i \in \Sc \end{cases}
\]
to the refined Alexander multi-degrees. Let $\Alex(a)$ denote the single Alexander degree of $a$. We have
\[
\Alex(a) = \sum_{i \notin \Sc} w_i(a) - \sum_{i \in \Sc} w_i(a).
\]
Explicitly, for a single edge $\gamma$, we have
\begin{itemize}
\item[\textasteriskcentered] $\Alex(\gamma) = 1/2$ if $\gamma$ has label $R_i$ or $L_i$ and $i \notin \Sc$
\item[\textasteriskcentered] $\Alex(\gamma) = -1/2$ if $\gamma$ has label $R_i$ or $L_i$ and $i \in \Sc$
\item[\textasteriskcentered] $\Alex(\gamma) = 1$ if $\gamma$ has label $U_i$ and $i \notin \Sc$
\item[\textasteriskcentered] $\Alex(\gamma) = -1$ if $\gamma$ has label $U_i$ or $C_i$ and $i \in \Sc$.
\end{itemize}

\item Define the \emph{Maslov grading} on $\Quiv(\Gamma(n,k,\Sc), \td{\mc R}_{\Sc})$, a grading by $\Z$, by declaring
\[
\m(\gamma) = \#_C(\gamma) - 2 \sum_{i \in \Sc} w_i(\gamma)
\]
for a path $\gamma$ in $\Gamma(n,k,\Sc)$, where $\#_C(\gamma)$ is the number of edges in $\gamma$ labeled $C_i$ for some $i$. Explicitly, for a single edge $\gamma$, we have
\begin{itemize}
\item[\textasteriskcentered] $\m(\gamma) = 0$ if $\gamma$ has label $R_i$, $L_i$, or $U_i$ and $i \notin \Sc$
\item[\textasteriskcentered] $\m(\gamma) = -1$ if $\gamma$ has label $R_i$, $L_i$, or $C_i$ and $i \in \Sc$
\item[\textasteriskcentered] $\m(\gamma) = -2$ if $\gamma$ has label $U_i$ and $i \in \Sc$.
\end{itemize}
\end{itemize}
\end{definition}

\begin{remark}
Our use of the words ``refined'' and ``unrefined'' follows the standard usage in bordered Floer homology, in contrast with \cite{ManionKS} (see Section~\ref{sec:gradings} below).
\end{remark}

\begin{definition}
The dg algebra $\B(n,k,\Sc)$ is defined to be $\Quiv(\Gamma(n,k,\Sc), \td{\mc R}_{\Sc})$, with any of the above three Alexander gradings as an intrinsic grading (preserved by $\partial$) and the Maslov grading as a homological grading (decreased by $1$ by $\partial$).
\end{definition}

The above definition is justified by the following theorem.
\begin{theorem}[{\cite[\corOSzQuiverEquivDG]{MMW1}}]\label{thm:ReviewQuiverEquivalence}
The dg algebra $\B(n,k,\Sc)$ defined in \cite{OSzNew} is isomorphic to $\Quiv(\Gamma(n,k,\Sc), \td{\mc R}_{\Sc})$.
\end{theorem}

One can also consider idempotent truncations of the algebras $\B(n,k,\Sc)$, which we review below.
\begin{definition}
For $0 \leq k \leq n$, define $\B_r(n,k,\Sc)$ to be
\[
\bigg( \sum_{\x: 0 \notin \x} \Ib_{\x} \bigg) \B(n,k,\Sc) \bigg( \sum_{\x: 0 \notin \x} \Ib_{\x} \bigg).
\]
Similarly, define $\B_l(n,k,\Sc)$ to be
\[
\bigg( \sum_{\x: n \notin \x} \Ib_{\x} \bigg) \B(n,k,\Sc) \bigg( \sum_{\x: n \notin \x} \Ib_{\x} \bigg).
\]
For $0 \leq k \leq n-1$, define $\B'(n,k,\Sc)$ to be
\[
\bigg( \sum_{\x: 0,n \notin \x} \Ib_{\x} \bigg) \B(n,k,\Sc) \bigg( \sum_{\x: 0,n \notin \x} \Ib_{\x} \bigg).
\]
One can also describe these algebras in terms of full subcategories of the dg category corresponding to $\B(n,k,\Sc)$; see \cite[\defTruncatedOSzAlgs]{MMW1}.
\end{definition}

\begin{remark}
As defined, $\Quiv(\Gamma(n,k,\Sc), \td{\mc R}_{\Sc})$ is a dg algebra over $\Ib(n,k)$. However, we can view it as an algebra over $\F_2[U_1,\ldots,U_n]^{V(n,k)}$ via the ring homomorphism 
\[
\F_2[U_1,\ldots,U_n]^{V(n,k)} \to \Quiv(\Gamma(n,k,\Sc), \td{\mc R}_{\Sc})
\]
sending $p \Ib_{\x}$, where $p$ is a monomial in the $U_i$ variables, to a path at $\x$ consisting of a $U_i$ loop for each factor of $p$ (in any order). The $U$ central relations in $\td{\mc R}_{\Sc}$ ensure that this homomorphism is well-defined and that the natural map 
\[
\F_2[U_1,\ldots,U_n] \to \F_2[U_1,\ldots,U_n]^{V(n,k)} \to \Quiv(\Gamma(n,k,\Sc), \td{\mc R}_{\Sc})
\]
has image in in the center of $\Quiv(\Gamma(n,k,\Sc), \td{\mc R}_{\Sc})$, so that we may view $\Quiv(\Gamma(n,k,\Sc), \td{\mc R}_{\Sc})$ as an $\F_2[U_1,\ldots,U_n]$-linear category. With this algebra structure understood, Theorem~\ref{thm:ReviewQuiverEquivalence} gives us an isomorphism of $\F_2[U_1,\ldots,U_n]^{V(n,k)}$-algebras.
\end{remark}

\begin{remark}
\label{rem:Orientation}
In Ozsv{\'a}th--Szab{\'o}'s conventions, the algebra $\B(n,k,\Sc)$ arises when one has an oriented tangle diagram with $n$ bottom (or top) endpoints, such that endpoint $i$ is oriented upward if and only if $i \in \Sc$. In our conventions, these diagrams will be rotated $90^{\circ}$ clockwise, and endpoint $i$ will be oriented rightward if and only if $i \in \Sc$ (see \cite[\remNinetyDegRot]{MMW1}).
\end{remark}

Next, we recall some structural definitions for Ozsv{\'a}th--Szab{\'o}'s algebras that were first introduced in \cite[Section 3.2]{OSzNew}. 

For $\x \in V(n,k)$ and $a \in [1,k]$, we let $x_a$ denote the $a^{th}$ element of $\x$ in increasing order. For $\x, \y \in V(n,k)$, define 
\[
v_i(\x,\y) := |\y \cap [i,n]| - |\x \cap [i,n]|.
\]
Let $|v|_i(\x,\y) := |v_i(\x,\y)|$.
\begin{definition}[{\cite[Definition 3.5]{OSzNew}}]\label{def:ReviewNotFarCrossed}
For $\x,\y \in V(n,k)$, we say that $\x$ and $\y$ are \emph{far} if there is some $a \in [1,k]$ with $|x_a - y_a| > 1$. Otherwise, we say that $\x$ and $\y$ are \emph{not far}.
\end{definition}
It follows from {\cite[Proposition 3.7]{OSzNew}} that if $\x$ and $\y$ are far then $\Ib_\x \B(n,k,\Sc) \Ib_\y=0$.

\begin{definition}
\label{def:CL}
 If $\x$ and $\y$ are not far, we say that $i \in [1,n]$ is a \emph{crossed line} if $v_i(\x,\y) \neq 0$.  We denote the set of crossed lines from $\x$ to $\y$ by $\CL{\x,\y}$.
\end{definition}

\begin{definition}
Given $\x, \y \in V(n,k)$, we say that a coordinate $i \in [0,n]$ is \emph{fully used} if $i \in \x \cap \y$. Otherwise, we say that $i$ is \emph{not fully used}.
\end{definition}

\begin{definition}[{\cite[Definition 3.6]{OSzNew}}]\label{def:ReviewGenInts}
Let $\x, \y \in V(n,k)$ be not far. We say that $[j+1,j+l]$ is a \emph{generating interval} for $\x$ and $\y$ if:
\begin{itemize}
\item $j$ and $j+l$ are not fully used coordinates,
\item for all $i \in [j+1,j+l-1]$, $i$ is a fully used coordinate, and
\item for all $i \in [j+1,j+l]$, $i$ is not a crossed line.
\end{itemize}
We say that $[[1,l]$ is a \emph{left edge interval} for $\x$ and $\y$ if coordinate $l$ is not fully used, but coordinate $i$ is fully used for all $i \in [0,l-1]$. Similarly, we say that $[n-l+1,n]]$ is a \emph{right edge interval} for $\x$ and $\y$ if coordinate $n-l$ is not fully used, but coordinate $i$ is fully used for all $i \in [n-l+1,n]$. In all of the above cases, we say that the \emph{length} of the generating or edge interval is $l$. Finally, if $\x = \y = [0,n]$, we say that $[[1,n]]$ is a \emph{two-faced edge interval} for $\x$ and $\y$ of length $n$.
\end{definition}

We have the following proposition from \cite{MMW1}.

\begin{proposition}[{\cite[\propQuartumNonDatur]{MMW1}}]
\label{prop:QuartumNonDatur}
Given $\x, \y \in V(n,k)$ not far, for each $i \in [1,n]$ exactly one of the following is true:
\begin{enumerate}
\item \label{it:QND1} $i\in\CL{\x,\y}$ (line $i$ is crossed);
\item \label{it:QND2} there exists a unique generating interval $G$ such that $i \in G$;
\item \label{it:QND3} there exists a unique (left, right, or two-faced) edge interval $G$ such that $i \in G$.
\end{enumerate}
\end{proposition}

For generating intervals, we use the following shorthand notation.
\begin{definition}
If $G = [j+1,j+l]$ is a generating interval for $\x$ and $\y$, then we let $p_G$ denote the monomial $U_{j+1} \cdots U_{j+l}$, an element of $\F_2[U_1,\ldots,U_n]$.
\end{definition}

Given $\x, \y \in V(n,k)$ that are not far, \cite[Proposition 3.7]{OSzNew} implies that $\Ib_{\x} \B(n,k,\Sc) \Ib_{\y}$ decomposes as a tensor product of chain complexes, with factors for the generating and edge intervals for $\x$ and $\y$ (see \cite[\corIBItoTensorProduct]{MMW1}). The factors are themselves certain special cases of the algebras $\B(n,k,\Sc)$ which we called generating and edge algebras in \cite{MMW1}, although the tensor product decomposition does not respect the multiplicative structure. See \cite[\secOSzSplittingTheorem]{MMW1} for more details.

In \cite{MMW1} we used this tensor product decomposition to compute the homology of $\B(n,k,\Sc)$; we review the result of this computation. First, we recall the definition of certain paths $\gamma_{\x,\y}$ in $\Gamma(n,k,\Sc)$.

\begin{definition}[{\cite[\defRecursive]{MMW1}}]\label{def:ReviewGammaXY}
Let $\x,\y \in V(n,k)$ be not far. Define a path $\gamma_{\x,\y}$ from $\x$ to $\y$ in $\Gamma(n,k,\Sc)$ by recursion on $k - |\x \cap \y|$ as follows.
\begin{itemize}
\item If $k - |\x \cap \y| = 0$, then $\x = \y$; define $\gamma_{\x,\y}$ to be the empty path based at $\x = \y$.
\item If $x_a < y_a$ for some $a \in [1,k]$, let $a$ be the largest such index. We have an edge $\gamma$ from $\x$ to $\x' = (\x \setminus \set{x_a}) \cup \{x_a + 1\}$ with label $R_{x_a + 1}$. Since $\x$ and $\y$ are not far, we have $y_a = x'_a$, so $k - |\x' \cap \y| = k - |\x \cap \y| - 1$. It follows that $\gamma_{\x',\y}$ is defined. Let $\gamma_{\x,\y} = \gamma \cdot \gamma_{\x',\y}$.
\item If $x_a \geq y_a$ for all $a \in [1,k]$ and $x_a > y_a$ for some $a$, let $a$ be the smallest such index. We have an edge $\gamma$ from $\x$ to $\x' = (\x \setminus \set{x_a}) \cup \{x_a - 1\}$ with label $L_{x_a}$. As before, we have $y_a = x'_a$. Thus, $k - |\x' \cap \y| = k - |\x \cap \y| - 1$ and $\gamma_{\x',\y}$ is defined. Let $\gamma_{\x,\y} = \gamma \cdot \gamma_{\x',\y}$ as above.
\end{itemize}
\end{definition}

\begin{remark}
In fact, the paths $\gamma_{\x,\y}$ can be defined even when $\x$ and $\y$ are far; in \cite[\secEqvofdescriptions]{MMW1}, we use them to prove the validity of a quiver description of Ozsv{\'a}th--Szab{\'o}'s algebra $\B_0(n,k)$.
\end{remark}

\begin{theorem}[{\cite[\thmOSzHomology]{MMW1}}]\label{thm:ReviewOSzHomology}
For $\x,\y \in V(n,k)$ that are not far, let $[j_1+1,j_1+l_1],\ldots,[j_b+1,j_b+l_b]$ be the generating intervals from $\x$ to $\y$, and let $p_1,\ldots,p_b$ be their monomials $p_G$. Choose an element $i_a \in [j_a+1,j_a+l_a] \cap \Sc$ for all $a$ such that this intersection is nonempty. We have a basis for $\Ib_{\x} H_*(\B(n,k,\Sc)) \Ib_{\y}$ in bijection with elements
\[
p \prod_{a=1}^b \left( \frac{C_{i_a}p_a}{U_{i_a}} \right)^{\varepsilon_a},
\]
where $p$ is a monomial in $\{U_i \,|\, i \in [1,n] \setminus \Sc\}$ not divisible by $p_a$ for any $a$ and $\varepsilon_a \in \{0,1\}$ is zero for $a$ such that $[j_a+1,j_a+l_a] \cap \Sc = \varnothing$. The bijection sends the element specified by ($p = 1$, $\varepsilon_a = 0$ for all $a$) to the element 
\[
[\gamma_{\x,\y}] \in \Ib_{\x} H_*(\B(n,k,\Sc)) \Ib_{\y}
\]
where $\gamma_{\x,\y}$ is the path of Definition~\ref{def:ReviewGammaXY}. It sends a more general element to the corresponding product of $[\gamma_{\x,\y}]$ with $U_i$ and $C_i$ loops, in any order.
\end{theorem}

We recall that the values of $n$, $k$, and $\Sc$ for which $\B(n,k,\Sc)$ is formal (when given the refined or unrefined Alexander multi-grading) were determined in \cite[\secFormality]{MMW1}. Given the results of this paper, the algebra $\A(n,k,\Sc)$ will be formal for the same values of $n,k,\Sc$, as stated in more detail in Corollary~\ref{cor:StrandsFormality}.

Finally, the algebras $\B(n,k,\Sc)$ have certain symmetries as described in \cite[Section 3.6]{OSzNew}. In our notation, these symmetries are called $\rho$ and $o$ (our $\rho$ is Ozsv{\'a}th--Szab{\'o}'s $\mc R$).

\begin{definition}\label{def:ReviewSymmetries}
On the vertex set $V(n,k)$ of $\Gamma(n,k,\Sc)$, define $\rho(\x) = \{n-i \,|\, i \in \x\}$ and $o(\x) = \x$. For $\Sc \subset [1,n]$, define $\rho(\Sc) = \{n+1-i \,|\, i \in \Sc\}$. Define
\[
\rho: \B(n,k,\Sc) \to \B(n,k,\rho(\Sc))
\]
by sending $\Ib_{\x}$ to $\Ib_{\rho(\x)}$ and sending edges labeled $R_i$, $L_i$, $U_i$, and $C_i$ to edges labeled $L_{n+1-i}$, $R_{n+1-i}$, $U_{n+1-i}$, and $C_{n+1-i}$ respectively. Define
\[
o: \B(n,k,\Sc) \to \B(n,k,\Sc)^{\op}
\]
by sending $\Ib_{\x}$ to $\Ib_{\x}$ and sending edges labeled $R_i$, $L_i$, $U_i$, and $C_i$ to edges labeled $L_i$, $R_i$, $U_i$ and $C_i$ respectively. We have both $\rho^2 = \id$ and $o^2 = \id$, properly interpreted. Restricting to the truncated algebras we get
\[
\rho: \B_r(n,k,\Sc) \xrightarrow{\cong} \B_l(n,k,\rho(\Sc)),
\]
\[
\rho: \B_l(n,k,\Sc) \xrightarrow{\cong} \B_r(n,k,\rho(\Sc)),
\]
and
\[
\rho: \B'(n,k,\Sc) \xrightarrow{\cong} \B'(n,k,\rho(\Sc))
\]
as well as 
\[
o: \B_r(n,k,\Sc) \xrightarrow{\cong} \B_r(n,k,\Sc)^{\op},
\]
\[
o: \B_l(n,k,\Sc) \xrightarrow{\cong} \B_l(n,k,\Sc)^{\op},
\]
and
\[
o: \B'(n,k,\Sc) \xrightarrow{\cong} \B'(n,k,\Sc)^{\op}.
\]
\end{definition}

We will relate the symmetries $\rho$ and $o$ to symmetries on the strands algebras $\A(n,k,\Sc)$ and their truncations, with visual interpretations, in Section~\ref{sec:PhiSymmetries} (see also Section~\ref{sec:StrandSymmetries}). 

\section{Chord diagrams and sutured surfaces}
\label{sec:Chord diagrams}

We now introduce a common generalization of Zarev's arc diagrams and of our example of interest (see Section \ref{sec:ExampleofInterest}).

\subsection{Definitions}

\begin{definition}\label{def:chord diagram}
A \emph{chord diagram} $\mc Z = (\mc Z, B, M)$ is a triple consisting of:
\begin{itemize}
\item a compact oriented $1$-manifold $\mc Z$;
\item a finite subset $B \subset \mc Z$ of basepoints, consisting of $2m$ points;
\item an involution $M$ on $B$ with no fixed points, called a \emph{matching}, which matches the basepoints in pairs.
\end{itemize}
The connected components of $\mc Z$ are called \emph{backbones}, and more specifically \emph{circular backbones} if they are closed and \emph{linear backbones} if they are not.

\begin{figure}
\includegraphics[scale=0.5]{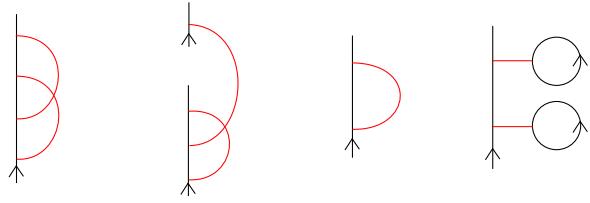}
\caption{Examples of chord diagrams.}
\label{fig:ChordDiagExamples}
\end{figure}

\begin{example}
Four examples of chord diagrams are represented visually in Figure~\ref{fig:ChordDiagExamples}. The backbones are shown in black; pairs of points matched by $M$ are connected by red arcs. The set $B$ of basepoints is the set of endpoints of the red arcs. The first three diagrams only have linear backbones; the fourth diagram has a linear backbone and two circular backbones. By convention, we will assume that all linear backbones drawn vertically in the plane are oriented upwards.
\end{example}

\begin{remark}
Chord diagrams, in several variants, appear in many places in mathematics. Perhaps the most common meaning of ``chord diagram'' is the special case of Definition~\ref{def:chord diagram} in which $\mc Z$ consists of a single circle; such chord diagrams appear (for example) in the study of Vassiliev knot invariants (see \cite{Kontsevich}). 

Like fatgraphs (a related notion), chord diagrams are often used to represent surfaces. Penner has a detailed language for referring to features of these diagrams, and we follow his terminology. The ``backbones'' terminology is part of this language; see e.g. \cite{ACPRS} for a discussion of chord diagrams with multiple backbones appearing in Teichm{\"u}ller theory and the combinatorics of RNA in biology.

In Heegaard Floer homology, chord diagrams are often called arc diagrams, following Zarev \cite{BSFH}. The connection between surface representations in bordered Floer homology and chord diagrams as studied e.g.~by Penner has already been noted in \cite[Remark 3.1]{LOTBimod}.
\end{remark}

Zarev considers only chord diagrams with no circular backbones; he interprets the basepoints on Lipshitz--Ozsv{\'a}th--Thurston's pointed matched circles as places to cut the circle open, obtaining linear backbones. Recent work of Ozsv{\'a}th--Szab{\'o} and Lipshitz--Ozsv{\'a}th--Thurston defining ``minus versions'' of bordered Heegaard Floer homology in various cases, including the constructions of \cite{OSzNew,OSzNewer} forming the subject of our study, have made use of diagrams with circular backbones (and without basepoints). 

\begin{figure}
\includegraphics[scale=0.5]{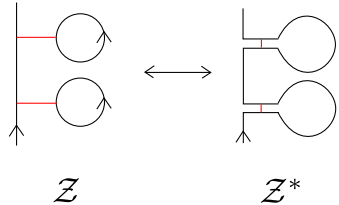}
\caption{The dual of a chord diagram.}
\label{fig:DualOfAChordDiag}
\end{figure}

Following \cite[Construction 8.18]{LOTMorphism}, a chord diagram $\mc Z$ has a \emph{dual} $\mc Z^* = (\mc Z^*,B^*,M^*)$ where $\mc Z^*$ is obtained by performing $1$-dimensional $0$-surgery on $\mc Z$ along $B$ according to $M$, $B^*$ is the union of the boundaries of the co-cores $\{1/2\} \times [0,1]$ of the surgery handles (each surgery replaces $S^0 \times D^1$ in $\mc Z$ with $D^1 \times S^0$ in $\mc Z^*$ and the surgery handle is $D^1 \times D^1$), and $M^*$ matches the two points of $B^*$ coming from each surgery handle. An example is shown in Figure~\ref{fig:DualOfAChordDiag}.

A chord diagram $\mc Z$ is called \emph{non-degenerate} if $\mc Z^*$ has no circular backbones. In Heegaard Floer homology, this condition has been most studied in the case where $\mc Z$ also has no circular backbones. A nice feature of the class of all chord diagrams, with both linear and circular backbones, is that the duality $\mc Z \leftrightarrow \mc Z^*$ gives an involution on this set of diagrams with no non-degeneracy conditions required.
\end{definition}

\subsection{Sutured surfaces}
Chord diagrams, viewed up to a natural equivalence relation given by chord-slides (sometimes called arc-slides), are a diagrammatic way of representing what Zarev calls sutured surfaces, in analogy with Gabai's sutured $3$-manifolds. Sutured surfaces share many similarities with bordered surfaces as in \cite{PennerDTT} as well as with open-closed cobordisms as in \cite{LaudaPfeiffer}. For topological motivation, we review how to get a sutured surface from a chord diagram in this section.

The following definition of sutured surface is slightly different from that of \cite{BSFH} in that we allow $S_+$ and $S_-$ to have closed components.

\begin{definition}
A \emph{sutured surface} is a triple $(F, \Lambda, S_+)$ consisting of the following data:
\begin{itemize}
\item a compact oriented surface $F$;
\item a finite collection $\Lambda \subset \de F$ of disjoint open intervals, each of which contains a point called a \emph{suture};
\item a splitting of $\de F \setminus \Lambda$ into compact submanifolds $S_+ \sqcup S_-$, such that for each component $C$ of $\Lambda$, $\de C$ intersects both $S_+$ and $S_-$.
\end{itemize}
\end{definition}

If $\mc F =(F,\Lambda,S_+)$ is a sutured surface, its dual is the sutured surface $\mc F^* = (F,\Lambda,S_-)$ in which the roles of $S_+$ and $S_-$ have been interchanged.

\begin{definition}\label{def:sutured surface for chord diagram}
Given a chord diagram $\mc Z = (\mc Z, B, M)$, we can build a sutured surface $\mc F(\mc Z) = (F(\mc Z), \Lambda)$ associated to it as follows:
\begin{itemize}
\item $F(\mc Z)$ is obtained from $\mc Z \times [0,1]$ by attaching a $1$-handle between $(z_1, 1)$ and $(z_2, 1)$ for every pair of matched basepoints $z_1$ and $z_2$ in an orientation-preserving manner;
\item $\Lambda = \de \mc Z \times (0,1)$, with sutures given by $\de \mc Z \times \{1/2\}$;
\item $S_+ = \mc Z \times \set{0}$.
\end{itemize}
\end{definition}

A sutured surface $F$ can be represented by a chord diagram if and only if each component of $F$ (not $\de F$) intersects $S_+$ and $S_-$ nontrivially.

We have $\mc F(\mc Z^*) = (\mc F(\mc Z))^*$. Thus, the non-degeneracy condition that $\mc Z^*$ has no circular backbones is equivalent to requiring that $S_-$ has no closed components in the sutured surface $(F,\Lambda,S_+)$ associated to $\mc Z$.

\subsection{The example of interest}
\label{sec:ExampleofInterest}

\begin{definition}\label{def:Z(n)}
We define the chord diagram $\mc Z(n) = (\mc Z(n), B, M)$ as follows.
\[
\mc Z(n) := [0,1] \sqcup S^1 \sqcup \cdots \sqcup S^1 \sqcup [0,1],
\]
where we take $n$ copies of $S^1:=[0,2]/(0 \sim 2)$. We can label the copies of $S^1$ from $1$ to $n$, and we denote the $i$-th copy of $S^1$ (i.e., the $i$-th circular backbone) by $S^1_i$. By analogy, we denote the two linear backbones by $[0,1]_0$ and $[0,1]_{n+1}$. For each $i=1, \ldots, n$, let $z_i^-:=[0] \in S^1_i$ and $z_i^+:=[1]\in S^1_i$ be two distinct basepoints in $S^1_i$. We also fix points $z_0^+ \in \Int([0,1]_0)$ and $z_{n+1}^- \in \Int([0,1]_{n+1})$. We define a matching $M$ on the set of basepoints $B = \set{z_0^+, z_1^\pm, \ldots, z_n^\pm, z_{n+1}^-}$ by matching $z_i^+$ with $z_{i+1}^-$, i.e.,
\[
M(z_i^+) = z_{i+1}^-.
\]
\end{definition}
Notice that we write our circles $S^1$ as $[0,2]/\sim$, rather than $[0,1]/\sim$; this allows each basepoint on each $S^1_i$ to occupy an integer value, easing various notations throughout the paper. Note that in particular the length of $S^1$ is $2$.

The sutured surface $\mc F(\mc Z(n))$ is a connected genus-zero surface with $n+1$ boundary components. One boundary component has four sutures; the rest have no sutures and are contained in $S_+$. The chord diagram $\mc Z(3)$ and the sutured surface $\mc F(\mc Z(3))$ are shown in Figure~\ref{fig:OurSuturedSurface}.

\begin{figure}
\includegraphics[scale=0.7]{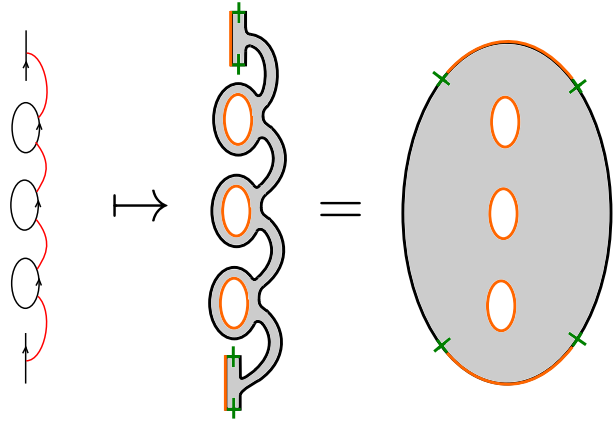}
\caption{The chord diagram $\mc Z(3)$ and the sutured surface $\mc F(\mc Z(3))$. Following Zarev's conventions from \cite{BSFH}, $S_+$ is colored orange and $S_-$ is colored black. The intervals $\Lambda$ are colored green, with sutures indicated by green marks. }
\label{fig:OurSuturedSurface}
\end{figure}

\section{The strands algebras}\label{sec:StrandsAlgDef}

Given the chord diagram $\mc Z(n)$ of Definition \ref{def:Z(n)} together with a subset $\Sc\subset[1,n]$, we would like to define a dg algebra $\sac nk\Sc$ (or equivalently a dg category, see \cite[\secAlgsAndCats]{MMW1}) called a strands algebra.  Intuitively, a general strands algebra $\A(\mc Z)$ assigned to a chord diagram $\mc Z$ should be generated by collections of homotopy classes of oriented continuous paths in $\Zc$ that both start and end at distinct basepoints.  The graphs of such paths are visualized as ``strands'' drawn on $\I \times \mc Z$, with multiplication defined via concatenation and the differential defined via resolutions of strand crossings.

In the upcoming paper \cite{ManionRouquier}, Rouquier and the first named author will define strands categories for singular curves functorially. For a chord diagram $\Zc$, this construction yields an algebra $\A(\Zc)$ defined along the above lines.  Here, though, we will follow \cite{LOT} and \cite{BSFH}, using a more combinatorial description that avoids some of the complications present in the general setting. The key point is that, in our chord diagram $\mc Z(n)$, any homotopy class of paths has a preferred representative, namely the constant-speed representative.  Such paths can be manipulated combinatorially, as we will see below.

\subsection{\texorpdfstring{$k$}{k}-strands and the pre-strands algebra}

\begin{definition}\label{def:kstrand}
A \emph{$k$-strand $s = \set{s_1, \ldots, s_k}$ on $\I \times \mc Z(n)$} is a collection of $k$ smooth functions
\[
s_a \colon \I \to \mc Z(n),
\]
called \emph{strands}, satisfying the following conditions:
\begin{itemize}
\item $s(0) := \set{s_1(0), \ldots, s_k(0)}$ consists of $k$ distinct points in $B$,
\item $s(1) := \set{s_1(1), \ldots, s_k(1)}$ consists of $k$ distinct points in $B$, and
\item for all $t \in \I$ and $1\leq a\leq k$, $\de_t s_a(t) = \alpha_a\geq 0$ for some constant speed $\alpha_a$.
\end{itemize}
We also say that $s$ is a $k$-strand \emph{from $s(0)$ to $s(1)$}. By a slight abuse of notation, we will use the notation $s_a$ for the graph of the strand $s_a$.
\end{definition}
Note that each strand is entirely determined by its starting point and its speed.
Also note that, since $S^1 = [0,2]/\sim$ has length $2$, the speed $\alpha_a$ is always a non-negative integer.

\begin{definition}\label{def:prestrand alg}
Given a subset $\Sc\subset[1,n]$, we define the \emph{pre-strands algebra} $\tsac nk\Sc$ as the algebra generated over $\F_2$ by all pairs $(s,\vec{c})$ where $s$ is a $k$-strand on $\I \times \Zc(n)$ and $\vec{c}$ is an element of $\set{0,1}^\Sc$. The multiplication on the generators of the algebra is defined via concatenation and addition as follows.
\begin{itemize}
\item If $s(1)\neq t(0)$, then $(s,\vec{c})\cdot(t,\vec{d})=0$ (we say the strands $s$ and $t$ were \emph{not concatenable}).
\item If $\vec{c}(i)+\vec{d}(i) = 2$ for any $i\in\Sc$, then $(s,\vec{c})\cdot(t,\vec{d})=0$ (we say that the multiplication produced a \emph{degenerate annulus}).  
\item Suppose $s$ contains two strands $s_a,s_b$ with speeds $\alpha_a,\alpha_b$ on one component of $\Zc(n)$, and $t$ also contains two strands $t_c,t_d$ with speeds $\beta_c,\beta_d$ such that $s_a(1)=t_c(0)$ and $s_b(1)=t_d(0)$.  If $(\alpha_a-\alpha_b)(\beta_c-\beta_d)<0$, then $(s,\vec{c})\cdot(t,\vec{d})=0$ (we say that the multiplication produced a \emph{degenerate bigon}).
\end{itemize}
If none of the three conditions above hold, we define $(s,\vec{c})\cdot(t,\vec{d})$ to be the pair $(s\cdot t,\vec{c}+\vec{d})$ where, for all $a\in[1,k]$, if $b$ is such that $s_a(1)=t_b(0)$, we define the speed of $(s\cdot t)_a$ to be $\alpha_a+\beta_b$.  Multiplication is then extended to all of $\tsac nk\Sc$ linearly.

In the case when $\Sc$ is the empty set, we often drop it from the notation and write the algebra as $\tsa nk$.
\end{definition}

\begin{remark} In \cite[Section 3.1.3]{LOT}, $k$-strands are defined algebraically as a bijection of sets $\phi:s(0)\rightarrow s(1)$ such that $\phi(b_a)\geq b_a$ for all basepoints $b_a\in s(0)$.  This definition is equivalent to ours when all backbones are linear because, once both endpoints are chosen for a strand on a linear backbone, the constant speed is also determined.  However for our circular backbones, the extra data of the speed is necessary to account for strands with nonzero wrapping number.  In this sense, our definition for $\tsa nk$ is a direct generalization of that of \cite{LOT}.
\end{remark}

\begin{remark}
In \cite[Section 3.1.3]{LOT}, Lipshitz--Ozsv{\'a}th--Thurston give another interpretation of the pre-strands algebras in terms of Reeb chords in contact $1$-manifolds, with the set of endpoints viewed as a Legendrian submanifold. This perspective is related to the interaction between strands algebras and holomorphic curve counts in bordered Floer homology. From this point of view, one can think of nonzero components $\vec{c}(i)$ of $\vec{c}$ as closed Reeb orbits; we thank Ko Honda for pointing out this connection, as well as the use of closed loops in the visual interpretation below.
\end{remark}

\subsection{Visual interpretation of the pre-strands algebra}

We visualize $k$-strands by their graph on $\I \times \mc Z(n)$, drawn ``horizontally'' as in the examples in Figure \ref{fig:kstrand examples}. The definition implies that intersections between two strands $s_a$ and $s_b$ in $s$ are transverse.  Furthermore, there are no points of triple (or more) intersection between strands in a $k$-strand, since there can be no more than two strands on any component of $\I \times \mc Z(n)$.  Meanwhile, we draw a single closed loop on the cylinder $S^1_i$ if and only if $\vec{c}(i)=1$.

We multiply by first concatenating the various $s_a,t_b$ if possible.  As long as we have not created an annulus or bigon in this way, we then homotope the result into a diagram of constant speed strands. See Figure~\ref{fig:NonDegenerateMult} for an example of a nonzero product and Figure~\ref{fig:Degeneracies} for examples of degenerate annuli and bigons.

\begin{figure}
\includegraphics[scale=0.5]{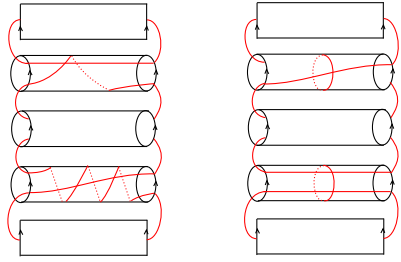}
\caption{Some examples of $k$-strands, with and without closed loops. For visual appeal, speeds are not drawn as entirely constant.}
\label{fig:kstrand examples}
\end{figure}

\begin{figure}
\includegraphics[scale=0.5]{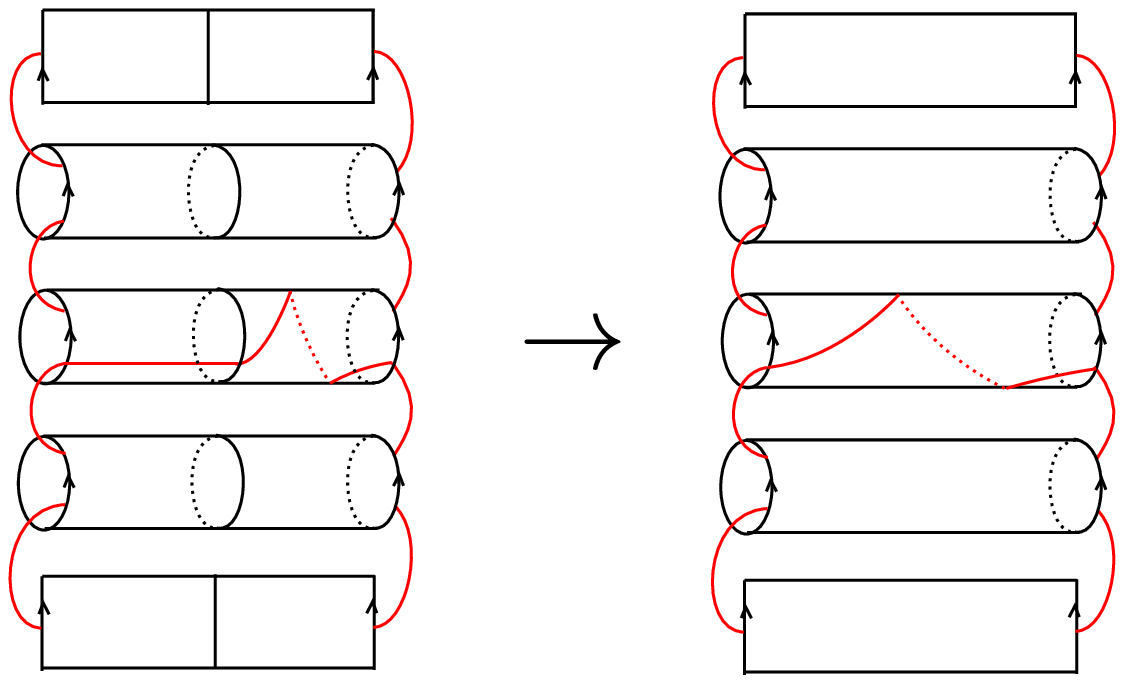}
\caption{Multiplication in $\td\A(n,k,\Sc)$.}
\label{fig:NonDegenerateMult}
\end{figure}

\begin{figure}
\includegraphics[scale=0.5]{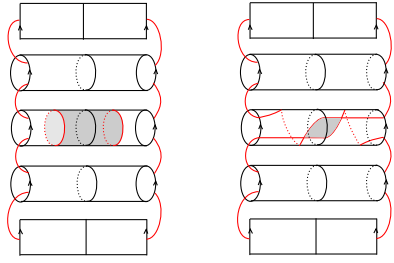}
\caption{Degenerate annuli and bigons}
\label{fig:Degeneracies}
\end{figure}

\subsection{A differential on the pre-strands algebra}
\label{sec:Differential}

\begin{definition}\label{def:differential}
Suppose $(s,\vec{c})\in \tsac nk\Sc$.  For all $i\in[1,n]$, we define the element $\partial_i^c (s,\vec{c})\in\tsac nk\Sc$ as follows.  If either $i\notin\Sc$ or $\vec{c}(i)=0$, then $\partial_i^c(s,\vec{c})=0$.  Otherwise we have $\vec{c}(i)=1$ and then we define $\partial_i^c(s,\vec{c})$ to be the sum over the following contributions.
\begin{itemize}
\item For any strand $s_a$ of $s$ on the backbone $S^1_i$ such that $s$ has no strands of strictly greater speed than $s_a$ on $S^1_i$, $\partial_i^c(s,\vec{c})$ has a contribution $(s',\pvec{c}')$ where $s'$ is obtained from $s$ by increasing the speed of $s_a$ by two and $\pvec{c}'$ is obtained from $\vec{c}$ by setting $\pvec{c}'(i) = 0$.
\end{itemize}
In particular, $\partial_i^c(s,\vec{c})$ is a sum of two distinct terms if and only if $s$ has two strands of equal speed on $S^1_i$.

We also define the element $\partial_i^0 (s,\vec{c})$ as follows.  If $s$ contains 0, 1, or 2 strands of equal speed on $S^1_i$, then $\partial^0_i(s,\vec{c})= 0$.  Otherwise $s$ contains two strands of differing speeds $p>q$ on the backbone $S^1_i$ and we have two cases.
\begin{itemize}
\item If $p-q=2$, then $\partial_i^0 (s,\vec{c})=(s',\vec{c})$ where $s'$ is the $k$-strand obtained from $s$ by replacing the two strands on $S^1_i$ by two new strands having (equal) speeds $p-1, q+1$.
\item If $p-q\geq 4$, then $\partial_i^0 (s,\vec{c})$ is a sum of two terms involving $k$-strands obtained from $s$ by replacing the two strands on $S^1_i$ by two new strands having (unequal) speeds $p-1$ and $q+1$.  There are two ways to do this, hence a sum of two terms.
\end{itemize}

We then define the \emph{$i^{th}$ differential} of $(s,\vec{c})$ to be
\[\partial_i(s,\vec{c}):= \partial_i^0(s,\vec{c}) + \partial_i^c(s,\vec{c})\]
and define the \emph{differential} of $(s,\vec{c})$ to be
\[\partial (s,\vec{c}) := \sum_{i\in[1,n]} \partial_i(s,\vec{c}).\]
\end{definition}

Visually, the case of nonzero $\partial_i (s,\vec{c})$ is precisely the case where we have strands and/or closed loops along $S^1_i$ that intersect.  We compute the differential by resolving crossings in the usual way.  The operator $\partial_i^0$ considers crossings between two strands of $s$ on $S^1_i$.  If $p-q=2$, there is only one crossing to resolve.  If $p-q\geq 4$ there are many crossings, but resolving any one other than the first or last will create a bigon, and such terms are set to zero so that we are left with two terms (which correspond to the two orderings of the new speeds $p-1$ and $q+1$).

The operator $\partial_i^c$ considers crossings between strands of $s$ and a closed loop on $S^1_i$; resolving a crossing between $s_a$ and a loop is equivalent to having the strand $s_a$ wrap once more around $S^1_i$ (corresponding to adding two to the overall speed of the strand).  If there are no other strands, this resolution cannot create any degeneracies.  If there is another strand $s_b$, the newly added ``wrapping'' of $s_a$ must intersect $s_b$ at infinite speed (before any homotopies); this resolution creates a degenerate bigon if and only if there are other crossings between $s_a$ and $s_b$ where the speed of $s_b$ is the greater of the two.  Thus with two strands of differing speeds on $S^1_i$, we keep only the resolution of the crossing between the loop and the faster strand.

In all of these nonzero cases, a simple homotopy takes the result of the resolution to a set of constant speed functions as desired. See Figure~\ref{fig:PreStrandsDifferential} for an illustration of $\partial^0$ and Figures~\ref{fig:PreStrandsDifferentialCPart1} and \ref{fig:PreStrandsDifferentialCPart2} for illustrations of $\partial^c$. Figures \ref{fig:PreStrandsDifferentialCPart1} and \ref{fig:PreStrandsDifferentialCPart2} in particular demonstrate that, although the result of the crossing resolutions defining $\partial^c$ could \emph{a priori} depend on the position of the closed loop, in fact the result is always given by our combinatorial formula.

\begin{figure}
\includegraphics[scale=0.6]{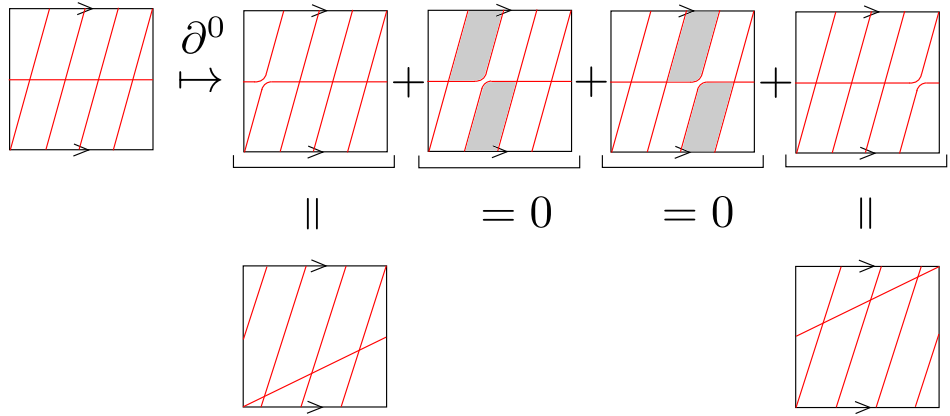}
\caption{The component $\partial^0$ of the differential on the pre-strands algebra.}
\label{fig:PreStrandsDifferential}
\end{figure}

\begin{figure}
\includegraphics[scale=0.5]{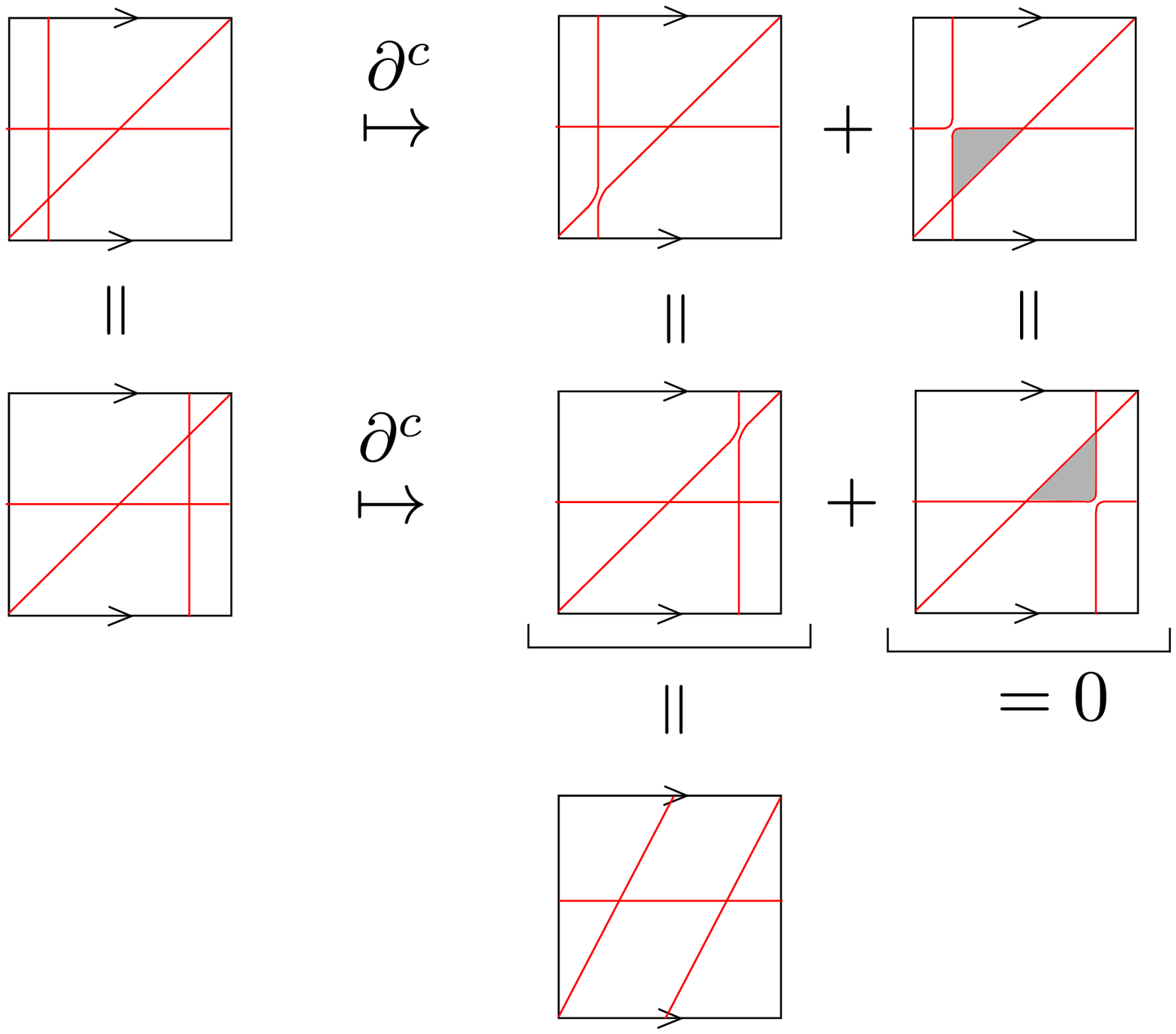}
\caption{The component $\partial^c$ of the differential on the pre-strands algebra.}
\label{fig:PreStrandsDifferentialCPart1}
\end{figure}

\begin{figure}
\includegraphics[scale=0.5]{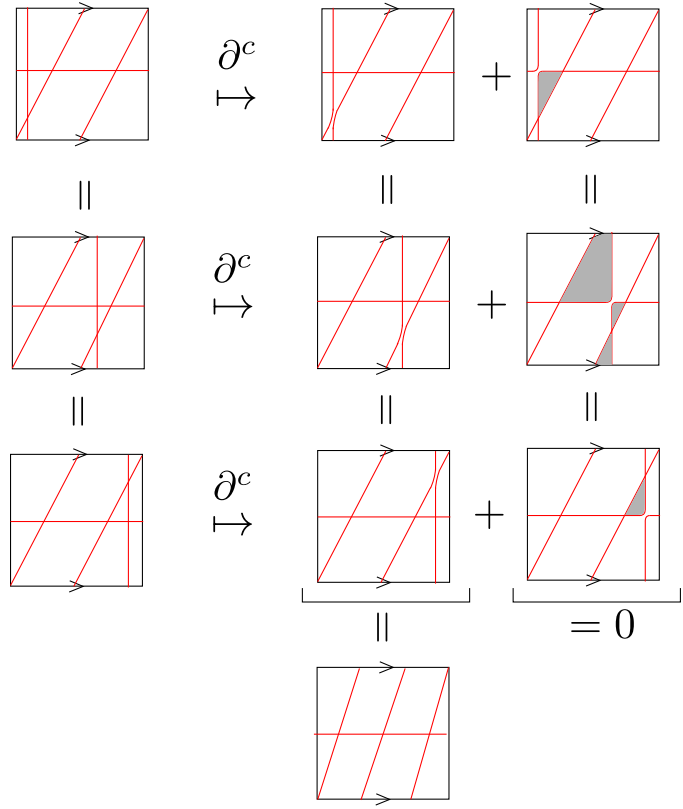}
\caption{The component $\partial^c$ in another example.}
\label{fig:PreStrandsDifferentialCPart2}
\end{figure}

\begin{lemma}
We have $\partial^2 =0$ for the differential on the pre-strands algebra.
\end{lemma}
\begin{proof}
It is clear that $\partial_i$ and $\partial_j$ commute, so that it is enough (over $\F_2$) to check that any $\partial_i^2=0$. The term $(\partial_i^c)^2$ is trivially zero due to the condition on $\vec{c}(i)$.  The reader may check that $(\partial_i^0)^2=0$ in both of the non-trivial cases (notice that, in the case of having two terms for $\partial_i^0(s,\vec{c})$, these two have equal image under $\partial^0_i$).  Finally, the fact that $\partial_i^c$ and $\partial_i^0$ commute is also a case-by-case check for which it is helpful to note that neither $\partial_i^c$ nor $\partial_i^0$ can change the number of strands of $s$ that are present on $S_i^1$.
\end{proof}

\subsection{The strands algebra}

We now begin to incorporate the matching $M$ for our chord diagram (see Section \ref{sec:Chord diagrams}) into our definitions.  We begin with some notation.  For any subset of basepoints $\Xb\subset B$, let $M(\Xb)$ denote the transformation of this set under the matching $M$.  That is, if $\Xb=\set{x_1,\ldots,x_k}\subset B$, then $M(\Xb)=\set{M(x_1),\ldots,M(x_k)}$.

Now let $s$ be a $k$-strand, and label the basepoints of $s(0) = \set{w_1, \ldots, w_k}$ in such a way that $w_a$ is the starting point of the strand $s_a$.
\begin{definition}\label{def:Ib notations}
Let $\Ib\subset B$ denote the set of elements $w_a \in s(0)$ such that $s_a$ is a constant strand.  For any subset $\ib \subset \Ib$, define the further notations $s(0)_{\ib} = (s(0) \sm \ib) \cup M(\ib)$ and $s(1)_{\ib} = (s(1) \sm \ib) \cup M(\ib)$ (note that $\ib \subset \Ib \subset s(0) \Longrightarrow \ib \subset s(1)$ as well). 
\end{definition}

\begin{lemma}\label{lem:matched const strands element}
If $s$ is a $k$-strand as above with $s(0)\cap M(s(0)) = s(1)\cap M(s(1))=\emptyset$, then for any subset $\ib\subset\Ib$, there is a well-defined $k$-strand $s_{\ib}$ from $s(0)_\ib$ to $s(1)_\ib$ defined by
\[
(s_{\ib})_a =
\begin{cases}
s_{a} & w_a \notin \ib \\
\const_{M(w_a)} & w_a \in \ib
\end{cases}
\]
for all $w_a \in s(0)$, where $\const_{M(w_a)}$ is the constant strand at $M(w_a)$.
\end{lemma}
\begin{proof}
Clearly the functions $(s_{\ib})_a$ are all still constant speed.  The fact that $s(0)\cap M(s(0))=\emptyset$ ensures that $s(0)_{\ib}$ still has $k$ elements, while $s(1)\cap M(s(1))=\emptyset$ ensures that $(s_{\ib})_a(1)\neq(s_{\ib})_b(1)$ for any $a\neq b$ in $[1,k]$.
\end{proof}
Note that, with this notation, we have $s_\ib(0)=s(0)_\ib$ and $s_\ib(1)=s(1)_\ib$ for any $\ib\subset\Ib$ by definition.

Now, given a generator $(s,\vec{c})\in\tsac nk\Sc$, we can use Lemma \ref{lem:matched const strands element} to introduce the further notation
\begin{equation}
\label{eq:E(S,T,s)}
E(s,\vec{c}) =
\begin{cases}
\displaystyle\sum_{\ib \subset \Ib} (s_{\ib},\vec{c}) & \text{if $s(0)\cap M(s(0)) = s(1)\cap M(s(1))=\emptyset$}\\
0 & \text{otherwise}
\end{cases}.
\end{equation}
We then extend $E$ linearly to a map $E:\tsac nk\Sc\rightarrow \tsac nk\Sc$.

\begin{lemma} \label{lem:E(S,T,s)}
Let $s,s'$ be $k$-strands such that 
\[s(0)\cap M(s(0)) = s(1)\cap M(s(1)) = s'(0)\cap M(s'(0)) = s'(1)\cap M(s'(1)) = \emptyset,\]
and let $\vec{c},\pvec{c}'\in\set{0,1}^\Sc$ be arbitrary.  Then we have $E(s,\vec{c})=E(s',\pvec{c}')$ in $\tsac nk\Sc$ if and only if $\vec{c}=\pvec{c}'$ and $s'=s_\jb$ for some subset $\jb\subset\Ib\subset s(0)$ of basepoints that are starting points of constant strands in $s$.
\end{lemma}
\begin{proof}
The necessity of $\vec{c}=\pvec{c}'$ is clear from the definition.  Note also that if $E(s',\pvec{c}')=E(s,\vec{c})$, then $(s',\pvec{c}')$ must be equal to one of the summands of $E(s,\vec{c})$, implying that $s'=s_\jb$ for some $\jb\subset\Ib$ as desired.

For the other direction, we write $\Ib_{\jb}$ for the set of starting points of constant strands in $s'=s_{\jb}$. Given $\ib \subset \Ib$, we can define $\kb = (\ib\sm\jb) \cup M(\jb \sm \ib)$, a subset of $\Ib_{\jb}$.  The map sending $\ib$ to $\kb$ is a bijection 
\[
\{\textrm{subsets of } \Ib\} \xrightarrow{\cong} \{\textrm{subsets of } \Ib_{\jb}\}
\]
with inverse sending $\kb \subset \Ib_{\jb}$ to $\ib = (\kb \sm M(\jb)) \cup (\jb\sm M(\kb)) \subset \Ib$.  Thus, there is a bijective correspondence between the summands of $E(s,\vec{c})$ and those of $E(s_{\jb},\vec{c})$.
One can then check that
\[
(s_{\ib},\vec{c}) = ((s_{\jb})_{\kb},\vec{c}),
\]
finishing the proof.
\end{proof}

\begin{lemma}\label{lem:EscLinIndep}
The set
\[
\set{E(s,\vec{c}) \,\middle|\, s(0) \cap M(s(0)) = s(1) \cap M(s(1)) = \varnothing}
\]
is a linearly independent subset of $\tsac nk\Sc$.
\end{lemma}
\begin{proof}
It is sufficient to note that any element of the form $E(s,\vec{c})$ can be expanded, in view of equation \eqref{eq:E(S,T,s)}, as the sum of basis vectors of $\tsac nk\Sc$, and, by Lemma \ref{lem:E(S,T,s)}, if $E(s,\vec{c}) \neq E(s',\pvec{c}')$, then their expansions do not contain any common basis vector $(t,\vec{d}) \in \tsac nk\Sc$.
\end{proof}

\begin{definition}\label{def:StrandsAlgBasis} As a vector space over $\F_2$, we define $\sac nk\Sc$ to be the subspace $\im(E)\subset\tsac nk\Sc$ spanned by the elements $E(s,\vec{c})$ (in the case when $\Sc$ is the empty set, we again drop it from the notation and write $\sa nk$ for $\sac nk\varnothing$).  By Lemma~\ref{lem:EscLinIndep}, the set
\[
\set{E(s,\vec{c}) \,\middle|\, \vec{c} \in \{0,1\}^{\Sc}, s(0) \cap M(s(0)) = s(1) \cap M(s(1)) = \varnothing}
\]
is an additive basis for $\sac nk\Sc$ over $\F_2$. We call it the \emph{standard basis}.  Propositions \ref{prop:StrandsAlg closed under mult}, \ref{prop:AnkSubcomplex}, and \ref{prop:sac identity} below will endow $\sac nk\Sc$ with a dg algebra structure over $\F_2$, with gradings described in Section \ref{sec:gradings}.  We will refer to $\sac nk\Sc$ as the \emph{strands algebra}.
\end{definition}

\begin{proposition}\label{prop:StrandsAlg closed under mult}
The vector space $\sac nk\Sc$ is closed under the multiplication inherited from $\tsac nk\Sc$.
\end{proposition}
\begin{proof}
Consider two basis elements $E(s,\vec{c})$ and $E(t,\vec{d})$ in $\sac nk\Sc$.  If there is some index $i\in\Sc$ with $\vec{c}(i)+\vec{d}(i)=2$, then $E(s,\vec{c})\cdot E(t,\vec{d})=0$ regardless of $s$ and $t$ (every concatenable term in the sum forms a degenerate annulus).  Thus we only need to check the case where $\vec{c}+\vec{d}\in\set{0,1}^\Sc$.

Let $\Ib\subset s(0)$ and $\Jb\subset t(0)$ denote the sets of basepoints that are starting points for constant strands in $s$ and $t$ respectively, so that
\begin{align*}
E(s,\vec{c})\cdot E(t,\vec{d}) &= \left(\sum_{\ib\subset\Ib} (s_\ib,\vec{c}) \right) \cdot \left( \sum_{\jb\subset\Jb}(t_\jb,\vec{d}) \right)\\
&= \sum_{\ib\subset\Ib}\sum_{\jb\subset\Jb} (s_\ib,\vec{c}) \cdot (t_\jb,\vec{d}).
\end{align*}
We see that, if $s_\ib(1) \neq t_\jb(0)$ for all $\ib\subset\Ib,\jb\subset\Jb$, then all of the strands $s_\ib,t_\jb$ are non-concatenable and the entire sum is zero.  Otherwise there are some $\ib,\jb$ such that $s_\ib(1) = t_\jb(0)$.  After using Lemma \ref{lem:E(S,T,s)} to replace $s$ with $s_{\ib}$ and $t$ with $t_{\jb}$, we can assume that $s(1)=t(0)$.

For the summand $(s_\ib,\vec{c}) \cdot (t_\jb,\vec{d})$ to be concatenable, we must have $s(1)_\ib = s_\ib(1) = t_\jb(0) = t(0)_\jb$, and since $s(1)=t(0)$, it follows that $\ib=\jb$.  Thus we have
\begin{equation}\label{eq:SumForProduct}
\begin{aligned}[b]
E(s,\vec{c})\cdot E(t,\vec{d}) &= \sum_{\ib\subset\Ib}\sum_{\jb\subset\Jb} (s_\ib,\vec{c}) \cdot (t_\jb,\vec{d})\\
&= \sum_{\ib\subset(\Ib\cap\Jb)} (s_\ib,\vec{c})\cdot(t_\ib,\vec{d})\\
&= \sum_{\substack{\ib\subset(\Ib\cap\Jb) \\ s_\ib \cdot t_\ib \text{ non-degenerate}}} (s_\ib\cdot t_\ib,\vec{c}+\vec{d})
\end{aligned}
\end{equation}
(we have already ruled out the possibility of a degenerate annulus due to $\vec{c}$ and $\vec{d}$). We claim that the concatenation $s_\ib \cdot t_\ib$ has a degenerate bigon if and only if $s \cdot t$ does.  Indeed, neither edge of a degenerate bigon can be a constant strand.  Thus, a degenerate bigon in $s\cdot t$ is bounded by two strands that are also included in $s_\ib\cdot t_\ib$ and vice versa.  It follows that this sum is zero if and only if $s \cdot t$ is degenerate.  

For non-degenerate $s\cdot t$, we have $s_\ib \cdot t_\ib = (s\cdot t)_\ib$ for all $\ib \subset \Ib \cap \Jb$, and $\Ib\cap\Jb$ is the set of starting basepoints of constant speed strands for $s\cdot t$ (speeds add, and $a+b\geq 0$ with equality only when $a=b=0$).  Thus the sum in \eqref{eq:SumForProduct} is the basis element for $s\cdot t$ and $\vec{c}+\vec{d}$, i.e.~we have proven that
\begin{equation}
\label{eq:sa nk multiplication}
E(s,\vec{c})\cdot E(t,\vec{d}) = E((s,\vec{c})\cdot(t,\vec{d})) = E(s\cdot t,\vec{c}+\vec{d})
\end{equation}
for non-degenerate $s\cdot t$.
\end{proof}

We envision the basis elements $E(s,\vec{c})$ as single diagrams, called \emph{basis diagrams}, comprised of non-constant solid strands in $\I \times \Zc(n)$ (corresponding to the non-constant strands of $s$) together with pairs of constant dashed strands in $\I \times \Zc(n)$ whose endpoints are matched by $M$ (corresponding to the constant strands of $s$ which lead to choices for $\ib$ in the sum for $E$).  In this way, a single pair of matched dashed strands indicates a sum of two strands diagrams. In each diagram we remove one of the two dashed strands and replace the other one with a solid strand.  See Figure~\ref{fig:ExpandingDottedLines} for an illustration.

\begin{figure}
\includegraphics[scale=0.5]{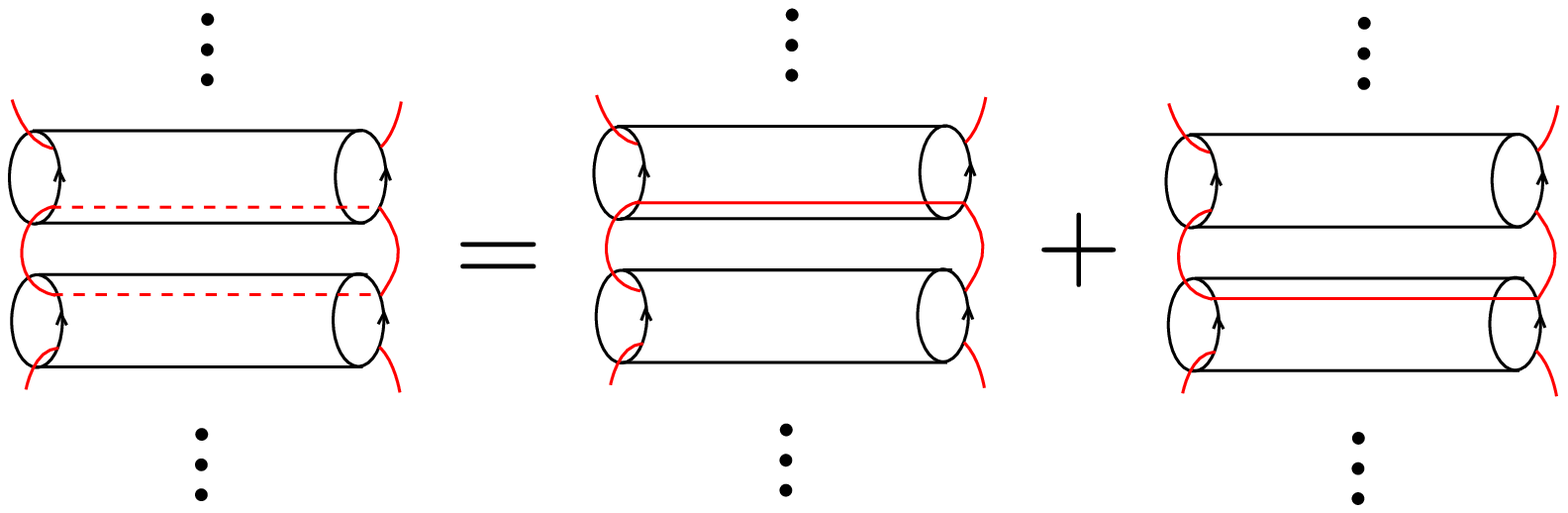}
\caption{Dotted-line pictures of basis elements of $\sac nk\Sc$.}
\label{fig:ExpandingDottedLines}
\end{figure}

The concatenation of basis elements $b_1$ and $b_2 \in \sac nk\Sc$ can be described pictorially in terms of basis diagrams. When a pair of dashed strands matches another pair of dashed strands, then they appear in the basis diagram of $b_1b_2$ as well. When a pair of dashed strands matches a single solid strand, then the dashed strand matched with the solid strand is treated as solid, whereas the other one disappears (see Figure~\ref{fig:DottedTimesSolid}). Finally, whenever a solid strand (or a pair of matched strands) of $b_1$ does not match any (solid or pair of dashed) strands of $b_2$, the product $b_1 b_2$ vanishes.

\begin{figure}
\includegraphics[scale=0.5]{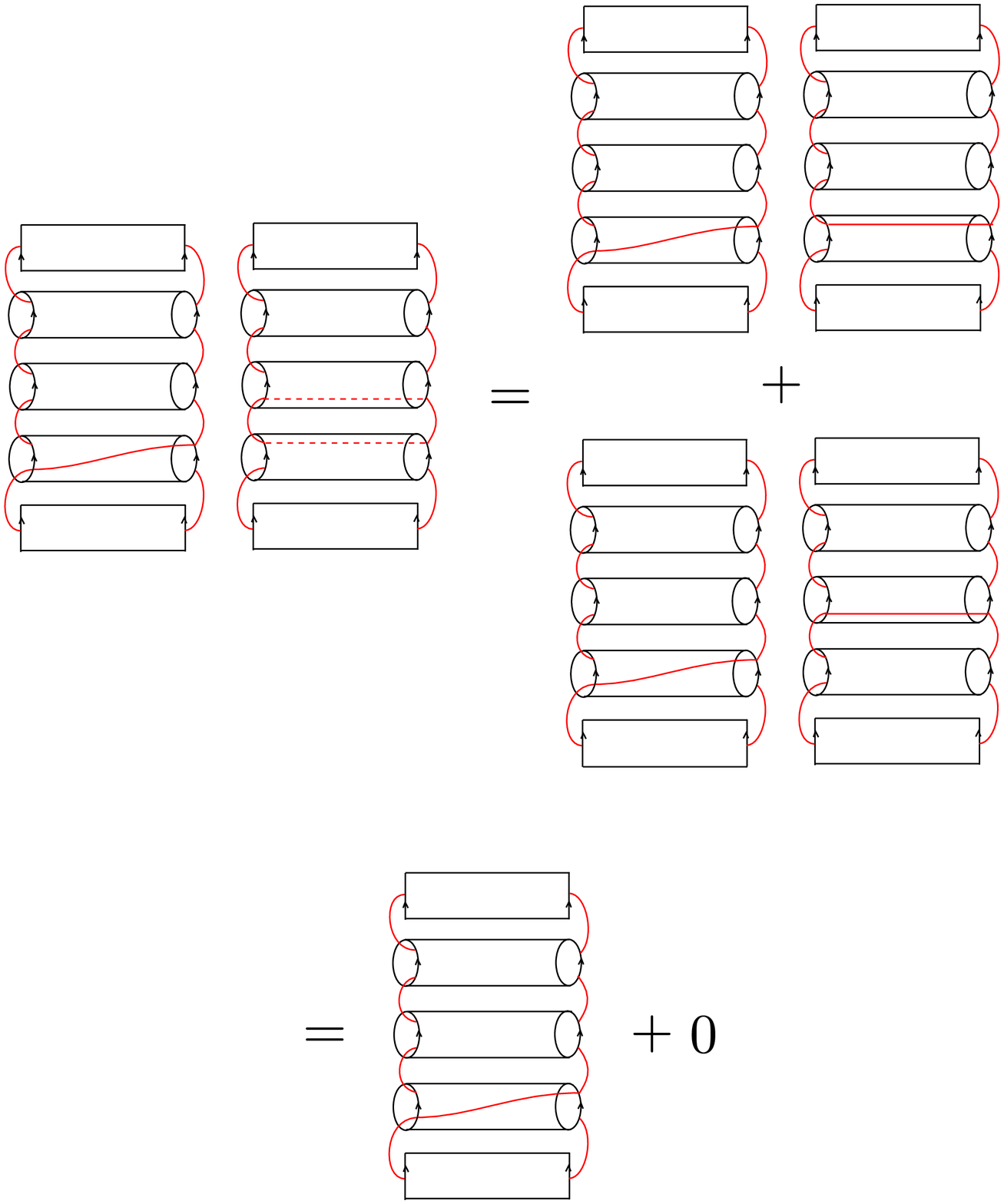}
\caption{Multiplying basis elements of $\sac nk\Sc$.}
\label{fig:DottedTimesSolid}
\end{figure}

Considering basis diagrams makes the visual interpretation of the differential clear as well. Viewing $a$ as a single basis diagram of solid and dashed strands as above, we express $b=\de a$ as a sum of crossing resolutions of $a$. Terms coming from resolving a crossing between solid strands clearly give further basis diagrams (the orientation-preserving property of strands ensures that a non-constant strand cannot suddenly become constant after a crossing resolution).  Meanwhile, crossings between solid and dashed strands can only contribute terms when the intersecting dashed strand is considered solid, and its matched partner is missing, giving a basis diagram after resolution with one fewer pair of dashed lines (note that after resolution, the formerly constant solid strand is no longer constant, and neither is the other solid strand).  See Figure~\ref{fig:DifferentialOfDotted} for an illustration.  This argument indicates that $\sac nk\Sc$ should inherit the differential of $\tsac nk\Sc$, which we prove using the following sequence of lemmas whose proofs are structurally very similar to each other. We will continue to use the notations of Definition \ref{def:Ib notations} throughout.

\begin{figure}
\includegraphics[scale=0.5]{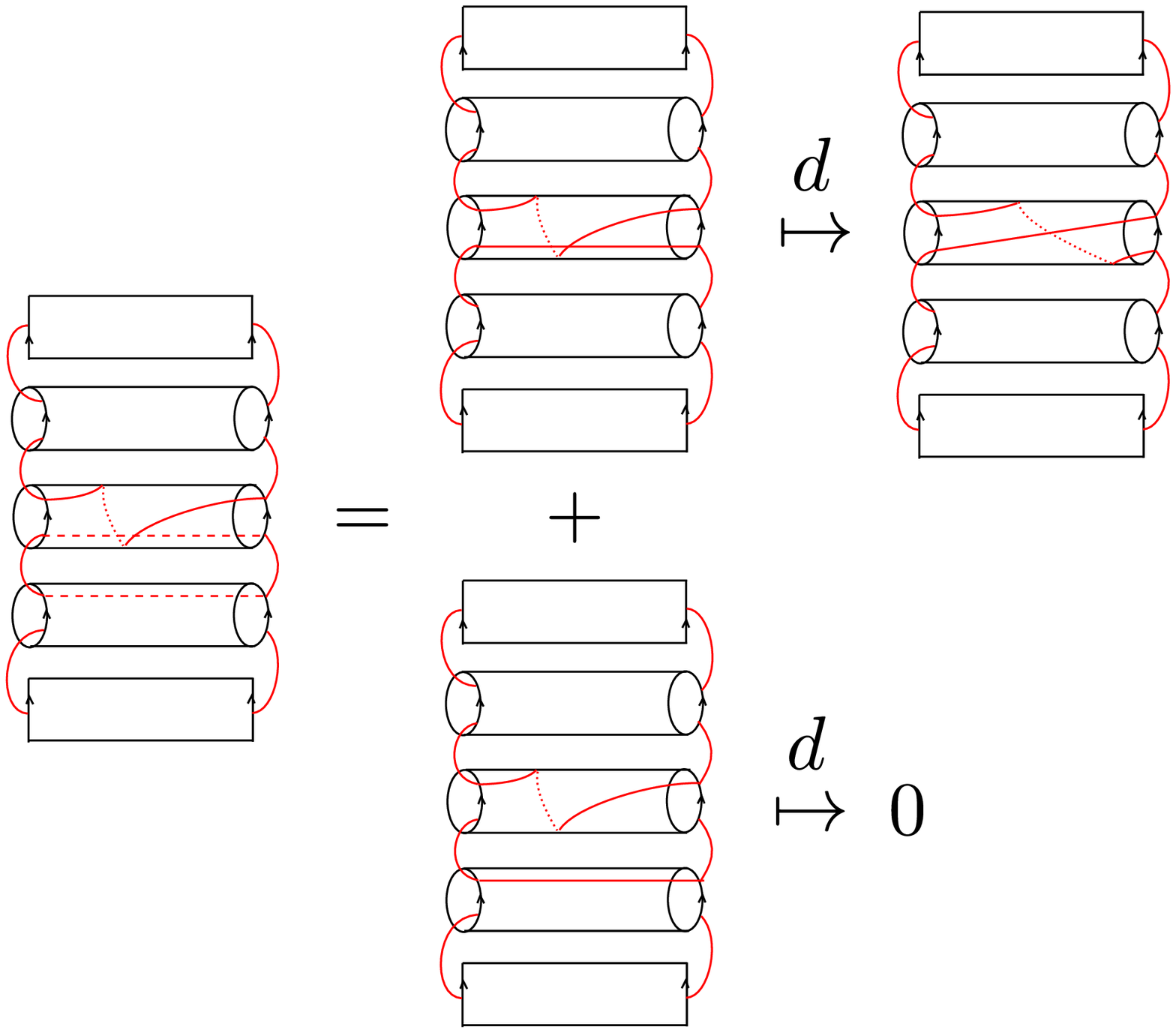}
\caption{Differentiating basis elements of $\sac nk\Sc$.}
\label{fig:DifferentialOfDotted}
\end{figure}

\begin{lemma}\label{lem:delc preserves Ank}
For any $a\in\sac nk\Sc$, we have $\de_j^c a\in\sac nk\Sc$ for all $j\in[1,n]$.
\end{lemma}
\begin{proof}
By linearity, we can suppose that $a = E(s, \vec{c})$ where $s(0) \cap M(s(0)) = s(1) \cap M(s(1)) = \varnothing$.
If $s$ has any non-constant strands on the backbone $S^1_j$, then all constant strands of $s$ remain constant in any summand of $\de_j^c(s,\vec{c})$.  In such a case, $\Ib$ is the set of starting points for constant strands in any summand of $\de_j^c(s,\vec{c})$. Thus, we have
\begin{align*}\de_j^c(E(s,\vec{c})) &= \de_j^c\left(\sum_{\ib\subset\Ib} (s_\ib,\vec{c})\right)\\
&= \sum_{\ib\subset\Ib} \de_j^c(s_\ib,\vec{c})\\
&= E(\de_j^c(s,\vec{c})) \in \sac nk\Sc.
\end{align*}

Now we consider the case where $s$ has no strands of speeds greater than zero on the backbone $S^1_j$.  If $s_\ib$ has no strands at all on $S^1_j$ for all $\ib\subset\Ib$, then every term $\de_j^c (s_\ib,\vec{c})$ in the sum for $\de_j^c E(s,\vec{c})$ is zero.  Similarly, if $\vec{c}(j)=0$, every term in the sum is zero as well.  Thus, we can assume that $\vec{c}(j)=1$ and there is some $\ib\subset\Ib$ such that $s_\ib$ has at least one constant strand on $S^1_j$. 
After replacing $s$ with some $s_{\ib}$ if necessary, we can also assume that the number of (constant) strands of $s$ on $S^1_j$ is maximal among elements of $\{s_{\ib} \,|\, \ib \subset \Ib\}$.
From here we consider two further subcases.

\begin{enumerate}
\item (One dashed strand and one loop) If $s$ contains only one constant strand on $S^1_j$, with starting (and ending) basepoint $z_j^\circ$ (here $\circ$ can be a $+$ or a $-$), we compute as follows:
\begin{align*}
\de_j^c E(s,\vec{c}) &= \de_j^c \sum_{\ib\subset\left(\Ib\sm\set{z_j^\circ}\right)} \left( (s_\ib,\vec{c}) + (s_{\ib\cup\set{z_j^\circ}},\vec{c})\right)\\
&= \sum_{\ib\subset\left(\Ib\sm\set{z_j^\circ}\right)} \left( \de_j^c(s_\ib,\vec{c}) + \de_j^c(s_{\ib\cup\set{z_j^\circ}},\vec{c})\right)\\
&= \sum_{\ib\subset\left(\Ib\sm\set{z_j^\circ}\right)} \left( (s_\ib',\pvec{c}') + 0 \right)
\end{align*}
where $s_\ib'$ and $\pvec{c}'$ are defined as in Definition \ref{def:differential}.  The zero term in the last line follows from the fact that $s_{\ib\cup\set{z_j^\circ}}$ cannot have any strands on $S^1_j$. Meanwhile, the terms $(s_\ib',\pvec{c}')$ will be nonzero by assumption, and since $s'$ in Definition \ref{def:differential} replaces the formerly constant strand at $z_j^\circ$ by one with speed 2 (but does not change any other strands), the set $\Ib\sm\set{z_j^\circ}$ is precisely the set of starting basepoints of constant strands for $s'$. Thus we have
\begin{equation*}
\de_j^c E(s,\vec{c}) = E(s', \pvec{c}')
= E(\de_j^c(s,\vec{c})) \in \sac nk\Sc
\end{equation*}
just as in the case when we had non-constant strands.

\item (Two dashed strands and one loop) If $s$ contains two constant strands on $S^1_j$, with starting basepoints $z_j^-,z_j^+$, we begin in the same way:
\begin{align*}
\de_j^c E(s,\vec{c}) &= \de_j^c \sum_{\ib\subset\left(\Ib\sm\set{z_j^-,z_j^+}\right)} \left( (s_\ib,\vec{c}) + (s_{\ib\cup\set{z_j^-}},\vec{c}) + (s_{\ib\cup\set{z_j^+}},\vec{c}) + (s_{\ib\cup\set{z_j^-,z_j^+}},\vec{c}) \right)\\
&= \sum_{\ib\subset\left(\Ib\sm\set{z_j^-,z_j^+}\right)} \left( \de_j^c (s_\ib,\vec{c}) + \de_j^c (s_{\ib\cup\set{z_j^-}},\vec{c}) + \de_j^c (s_{\ib\cup\set{z_j^+}},\vec{c}) + 0 \right)
\end{align*}
where we get the zero term using the same reasoning as before.  Now we write $\de_j^c(s_\ib,\vec{c})=(s_\ib^-,\pvec{c}')+(s_\ib^+,\pvec{c}')$ where $s_\ib^\pm$ is the $k$-strand defined by replacing the constant strand at $z_j^\pm$ by a new strand of speed 2, while maintaining the other strands (including the constant strand at $z_j^\mp$).  We can then write our sum as
\[\de_j^c E(s,\vec{c}) = \sum_{\ib\subset\left(\Ib\sm\set{z_j^-,z_j^+}\right)} \left( (s_\ib^+,\pvec{c}') + (s'_{\ib\cup\set{z_j^-}},\pvec{c}') + (s_\ib^-,\pvec{c}') + (s'_{\ib\cup\set{z_j^+}},\pvec{c}') \right)
\]
where $s'_{\ib\cup\set{z_j^\pm}}$ are defined as in Definition \ref{def:differential}.  The key point is to recognize that $(s_\ib^+,\pvec{c}')$ and $(s'_{\ib\cup\set{z_j^-}},\pvec{c}')$ are complementary in the sense that both have a strand of speed 2 starting and ending at $z_j^+$, and indeed have the same strands everywhere except that $(s_\ib^+,\pvec{c}')$ has a constant strand at $z_j^-$, while $(s'_{\ib\cup\set{z_j^-}},\pvec{c}')$ has a constant strand at $z_{j-1}^+=M(z_j^-)$.  This reasoning implies that
\[\sum_{\ib\subset\left(\Ib\sm\set{z_j^-,z_j^+}\right)} \left( (s_\ib^+,\pvec{c}') + (s'_{\ib\cup\set{z_j^-}},\pvec{c}') \right) = \sum_{\ib\subset\left(\Ib\sm\set{z_j^+}\right)} (s^+_\ib,\pvec{c}')\]
and similarly we have
\[\sum_{\ib\subset\left(\Ib\sm\set{z_j^-,z_j^+}\right)} \left( (s_\ib^-,\pvec{c}') + (s'_{\ib\cup\set{z_j^+}},\pvec{c}') \right) = \sum_{\ib\subset\left(\Ib\sm\set{z_j^-}\right)} (s^-_\ib,\pvec{c}').\]
Since $\Ib\sm\set{z_j^\pm}$ is the set of constant strand starting points for $s^\pm$, we have
\begin{align*}
\de_j^c E(s,\vec{c}) &= \sum_{\ib\subset\left(\Ib\sm\set{z_j^+}\right)} (s^+_\ib,\pvec{c}') + \sum_{\ib\subset\left(\Ib\sm\set{z_j^-}\right)} (s^-_\ib,\pvec{c}')\\
&= E(s^+,\pvec{c}') + E(s^-,\pvec{c}')\\
&= E(\de_j^c(s,\vec{c})) \in\sac nk\Sc.
\end{align*}
\end{enumerate}

Thus in all cases we see that, after replacing $s$ with some $s_{\ib}$ if necessary as described above,
\begin{equation}\label{eq:delc commutes with E}
\de_j^c(E(s,\vec{c})) = E(\de_j^c(s,\vec{c}))\in\sac nk\Sc,
\end{equation}
proving the lemma.
\end{proof}

\begin{lemma}\label{lem:del0 preserves Ank}
For any $a\in\sac nk\Sc$, we have $\de_j^0 a\in\sac nk\Sc$ for all $j\in[1,n]$.
\end{lemma}
\begin{proof}
As in the proof of Lemma \ref{lem:delc preserves Ank}, we can suppose that $a = E(s, \vec{c})$ where $s(0) \cap M(s(0)) = s(1) \cap M(s(1)) = \varnothing$.
If $\de^0_j(s_\ib,\vec{c})=0$ for all $\ib\subset\Ib$, then we have $\de^0_j a = 0\in \sac nk\Sc$ trivially.  Thus it is enough to consider the case where there is some subset $\ib\subset\Ib$ such that $\de^0_j(s_\ib,\vec{c})\neq 0$.  By Lemma \ref{lem:E(S,T,s)} we can replace $s$ by this $s_\ib$ without changing $a$, and so we may assume without loss of generality that $s$ itself satisfies $\de^0_j (s,\vec{c}) \neq 0$.

In particular, we may assume that $s$ contains two strands on the backbone $S^1_j$ having unequal speeds $p>q$.  We now split into two further subcases:

\begin{enumerate}
\item $q\neq 0$ (Two solid strands): In this case, neither strand on $S^1_j$ is constant, so every term in the sum $\de^0_j a = \sum_{\ib\subset\Ib} (\de^0_j (s_\ib, \vec{c}))$ is nonzero, and is a sum of one term (if $p-q=2$) or two terms (if $p-q>2$) with non-constant strands.  Furthermore, since $\de^0_j$ does not affect any strand away from the backbone $S^1_j$, constant strands of $s$ remain constant in any summand of $\de^0_j (s,\vec{c})$. Thus the set $\Ib$ of constant strand starting points for $s$ is the set of constant strand starting points for any summand of $\de^0_j (s,\vec{c})$ in this case.  For $\ib \subset \Ib$, we may write $\de^0_j(s_\ib, \vec{c}) = (\de^0_j (s, \vec{c}))_\ib$, so our sum becomes
\[\de^0_j a = \sum_{\ib\subset\Ib} (\de^0_j (s, \vec{c}))_{\ib} = E(\de^0_j(s, \vec{c})) \in\sac nk\Sc,\]
extending $()_{\ib}$ linearly.

\item\label{it:Del0DashedConstant} $q=0$ (One dashed strand): In this case, $s$ has a constant strand on some basepoint $z_j^\circ\in S^1_j$, where $\circ \in \{+,-\}$.  Thus for any $\ib\subset\Ib$ containing $z_j^\circ$, $s_\ib$ does not contain this constant strand, so $\de^0_j(s_\ib,\vec{c})=0$. It follows that the only terms in the sum that matter are those that come from subsets not containing $z_j^\circ$. For such subsets $\ib$, we may again write $\de^0_j (s_\ib,\vec{c}) = (\de^0_j (s,\vec{c}))_\ib$.

Meanwhile, the terms in $\de^0_j (s,\vec{c})$ will contain strands of speeds $p-1$ and $1$ on $S^1_j$. There will be two such terms if $p\neq 2$ and one such term if $p=2$.  All other strands of $s$ are maintained.  We cannot have $p=1$ since the strands of $s$ end on distinct basepoints in $s(1)$.  Thus, the set of constant strand starting points for $\de^0_j (s,\vec{c})$ is precisely $\Ib\setminus\set{z_j^\circ}$.  Altogether, we can write our sum as
\[\de^0_j a = \sum_{\ib\subset\Ib\sm\set{z_j^\circ}} (\de^0_j (s_\ib, \vec{c})) = \sum_{\ib\subset(\Ib\setminus\set{z_j^\circ})} (\de_j (s,\vec{c}))_\ib = E(\de^0_j (s,\vec{c})) \in\sac nk\Sc\]
where we have again extended $()_\ib$ linearly.
\end{enumerate}

As in Lemma~\ref{lem:delc preserves Ank}, we have
\begin{equation}\label{eq:del0 commutes with E}
\de_j^0(E(s,\vec{c})) = E(\de_j^0(s,\vec{c}))\in\sac nk\Sc
\end{equation}
in all cases, again after replacing $s$ with some $s_\ib$ if necessary, proving the lemma.
\end{proof}

\begin{proposition}\label{prop:AnkSubcomplex}
The subspace $\sac nk\Sc$ of $\tsac nk\Sc$ is preserved by the differential on $\tsac nk\Sc$.
\end{proposition}
\begin{proof}
Since $\de=\sum_{j\in[1,n]} \de_j^c + \de_j^0$, this proposition follows from Lemmas \ref{lem:delc preserves Ank} and \ref{lem:del0 preserves Ank}.
\end{proof}

\subsection{Idempotents and the unit}\label{sec:IdemsAndUnit}

At this point, we can almost say that $\sac nk\Sc$ is a differential algebra (we will see in Section \ref{sec:gradings} that it is in fact a dg algebra). One subtlety is that the unit of $\tsac nk\Sc$ is not an element of $\sac nk\Sc$. However, $\sac nk\Sc$ has its own unit, which we define below.

Let $B$ be the set of basepoints in $\mc Z(n)$ as above, and let $M$ be the matching on $B$.  We write $B/M:=B/(z\sim M(z))$ and consider the quotient map $q: B\rightarrow B/M$.  We can identify $B/M$ with $[0,n]$ by sending $\{z_i^+, z_{i+1}^-\} \in B/M$ to the index $i \in [0,n]$.

Let $\x$ be a $k$-element subset $\x$ of $B/M \cong [0,n]$, i.e. an element of $V(n,k)$ in the notation of Definition~\ref{def:ReviewVnk} (we will resume using this notation below). Following Lipshitz--Ozsv{\'a}th--Thurston, we will call a subset $\Sb \subset B$ a \emph{section of} $\x$ if $\Sb$ is the image of a section of the quotient map $q$ over $\x$.

\begin{definition}\label{def:Strands Idempotents}
For $\x \in V(n,k)$, let $\Jb_{\x}$ be the element $E(\const_{\Sb},\vec{0})$ of $\sac nk\Sc$ where $\Sb$ is any section of $\x$ and $\const_{\Sb}$ is the $k$-strand of constant strands at each basepoint in $\Sb$.  Note that this definition is independent of the choice of $\Sb$ by Lemma \ref{lem:E(S,T,s)}. Define 
\[
1_{\sac nk\Sc} := \sum_{\x \in V(n,k)} \Jb_{\x}.
\]
\end{definition}

A section $\Sb$ of $\x$ can always be chosen by the rule that $i\in\x$ if and only if $z_i^+\in\Sb$.  Regardless of the choice of section, however, the element $\Jb_\x$ is visually interpreted as the diagram consisting of a constant dashed strand at each point of $q^{-1}(\x) \subset B$.  The elements $\Jb_\x$ for $\x \in V(n,k)$ constitute a set of pairwise orthogonal idempotents in $\sac nk\Sc$.

\begin{proposition}
\label{prop:sac identity}
The element $1_{\sac nk\Sc}$ is an identity element for $\sac nk\Sc$.
\end{proposition}

\begin{proof}
Let $E(s,\vec{c})$ denote a standard basis element of $\sac nk\Sc$.  Let $\y,\z\in V(n,k)$ denote the images of $q(s(0))$ and $q(s(1))$ respectively under the identification of $B/M$ with $[0,n]$.  Note that, if $\x\neq\y$, then $\Jb_\x \cdot E(s,\vec{c})=0$ because none of the summands of $\Jb_{\x}$ and $E(s,\vec{c})$ are concatenable, while $\Jb_\y \cdot E(s,\vec{c})=E(s,\vec{c})$ (see the proof of Proposition \ref{prop:StrandsAlg closed under mult} where it is shown that $E(t,\vec{d})\cdot E(s,\vec{c}) = E(t\cdot s, \vec{d}+\vec{c})$ after $t$ and $s$ are chosen appropriately).  Similarly, $E(s,\vec{c})\cdot\Jb_\x=0$ for $\x\neq\z$, while $E(s,\vec{c})\cdot\Jb_\z=E(s,\vec{c})$.  Thus we have
\begin{gather*}1_{\sac nk\Sc} \cdot E(s,\vec{c}) = \Jb_\y \cdot E(s,\vec{c}) + \sum_{\x\neq \y} \Jb_\x \cdot E(s,\vec{c}) = E(s,\vec{c}),\\
E(s,\vec{c}) \cdot 1_{\sac nk\Sc} = E(s,\vec{c})\cdot \Jb_\z + \sum_{\x\neq \z} E(s,\vec{c})\cdot \Jb_\x = E(s,\vec{c}).
\end{gather*}
\end{proof}

We see that $\sac nk\Sc$ is a differential algebra over $\F_2$ (gradings will be discussed in Section~\ref{sec:gradings}). Moreover, $\sac nk\Sc$ can be viewed as an algebra over the idempotent ring $\Ib(n,k)=\F_2^{V(n,k)}$ via the ring homomorphism sending the indicator function of $\x\in V(n,k)$ to $\Jb_\x\in \sac nk\Sc$, and we have a natural splitting
\begin{equation}\label{eq:StrandsAlgIdemSplitting}
\sac nk\Sc = \bigoplus_{\x,\y \in V(n,k)} {\Jb_\x} \sac nk\Sc {\Jb_\y}
\end{equation}
(see \cite[\lemOrthogonalIdempotents]{MMW1} for more details).

Each element of the basis for $\sac nk\Sc$ from Definition~\ref{def:StrandsAlgBasis} is homogeneous with respect to the decomposition of \eqref{eq:StrandsAlgIdemSplitting}. The basis element $E(s,\vec{c})$ lies in $\Jb_{\x} \sac nk\Sc {\Jb_{\y}}$ if and only if $\x$ and $\y$ are the projections of $s(0), s(1) \subset B$ to $k$-element subsets of $(B/M) \cong [0,n]$ respectively (in other words, $\x$ and $\y$ are the unique I-states such that $s(0)$ is a section of $\x$ and $s(1)$ is a section of $\y$). Thus, we have the following lemma.

\begin{lemma}
Let $\x, \y \in V(n,k)$. An $\F_2$-basis of the summand ${\Jb_\x} \sac nk\Sc {\Jb_\y}$ of $\sac nk\Sc$ consists of all standard basis elements of $\sac nk\Sc$ of the form $E(s,\vec{c})$, where $s(0)$ and $s(1)$ are sections of $\x$ and $\y$ respectively.
\end{lemma}

\subsection{Far states and the strands algebra}

Recall that for $\x, \y \in V(n,k)$ that are ``far'' in the sense of Definition~\ref{def:ReviewNotFarCrossed}, we have $\Ib_{\x} \B(n,k,\Sc) \Ib_{\y} = 0$. Below we prove a similar result for the strands algebra.

\begin{lemma}\label{lem:FarStatesStrandAlgZero}
Let $\x, \y \in V(n,k)$. If $\x, \y$ are far, then $\Jb_{\x} \A(n,k,\Sc) \Jb_{\y} = 0$. 
\end{lemma}

\begin{proof}
We will show the contrapositive.  Let $\x = \set{x_1<\cdots<x_k}$ and $\y = \set{y_1<\cdots<y_k}$ and suppose that $\Jb_{\x} \A(n,k,\Sc) \Jb_{\y} \neq 0$. Then there exists a $k$-strand $s$ from a section $\Sb$ of $\x$ to a section $\Tb$ of $\y$, which in turn gives a bijection $\varphi \colon \x \to \y$ with the property that for all $a \in [1,k]$ we have
\[
|\varphi(x_a) - x_a| \leq 1.
\]
We wish to show that this condition implies $|y_a - x_a| \leq 1$ for all $a$ as well (and hence $\x$ and $\y$ are not far). To prove that $y_a-x_a\leq1$, assume by contradiction that the set $\set{a \in [1,k] \,\middle|\, y_a - x_a > 1}$ is non-empty, and let $m$ be its minimum. Then, for all $b=1, \ldots, m-1$, we have that
\[
\varphi(x_b) \leq 1 + x_b < y_m.
\]
Moreover, $\varphi(x_m)- x_m \leq 1 < y_m - x_m$, so $\varphi(x_m)<y_m$ too. Then $\varphi$ is injective from $\set{x_1, \ldots, x_m}$ to $\set{y_1, \ldots, y_{m-1}}$, which is a contradiction.  To prove that $y_a-x_a\geq-1$, we can apply the same reasoning to the maximum of the set $\set{a \in [1,k] \,\middle|\, y_a - x_a < -1}$.  Thus we must have that $|y_a - x_a| \leq 1$, so $\x$ and $\y$ are not far.
\end{proof}

\subsection{Idempotent-truncated strands algebras}

As in \cite[\secOSzTruncatedAlgs]{MMW1}, one can define truncated versions of $\sac nk\Sc$ (see also \cite[Section 12]{OSzNew}).

\begin{definition}\label{def:TruncatedStrandsAlgs}
For $0 \leq k \leq n$, define $\A_r(n,k,\Sc)$ to be
\[
\bigg( \sum_{\x: 0 \notin \x} \Jb_{\x} \bigg) \A(n,k,\Sc) \bigg( \sum_{\x: 0 \notin \x} \Jb_{\x} \bigg).
\]
Similarly, define $\A_l(n,k,\Sc)$ to be
\[
\bigg( \sum_{\x: n \notin \x} \Jb_{\x} \bigg) \A(n,k,\Sc) \bigg( \sum_{\x: n \notin \x} \Jb_{\x} \bigg).
\]
For $0 \leq k \leq n-1$, define $\A'(n,k,\Sc)$ to be
\[
\bigg( \sum_{\x: 0,n \notin \x} \Jb_{\x} \bigg) \A(n,k,\Sc) \bigg( \sum_{\x: 0,n \notin \x} \Jb_{\x} \bigg).
\]
\end{definition}
As with the truncations of $\B(n,k,\Sc)$, one can also describe these algebras in terms of full subcategories of the dg category corresponding to $\A(n,k,\Sc)$; see \cite[\defTruncatedOSzAlgs]{MMW1}.

In fact, as with $\sac nk\Sc$, the truncated algebras are special cases of strands algebras for chord diagrams that will be defined in \cite{ManionRouquier}. We describe these diagrams below.
\begin{definition}\label{def:Zr(n)}
We define the chord diagram $\mc Z_r(n)$ to be $(\mc Z_r(n), B, M)$ where
\[
\mc Z_r(n) := S^1 \sqcup \cdots \sqcup S^1 \sqcup [0,1].
\]
For $i=2, \ldots, n$, let $z_i^-:=[0] \in S^1_i$ and $z_i^+:=[1]\in S^1_i$ be two distinct basepoints in $S^1_i$. We also fix points $z_1^+ \in S^1_1$ and $z_{n+1}^- \in \Int([0,1]_{n+1})$. We define a matching $M$ on the set of basepoints $B = \set{z_1^+, z_2^\pm, \ldots, z_n^\pm, z_{n+1}^-}$ by matching $z_i^+$ with $z_{i+1}^-$, i.e.,
\[
M(z_i^+) = z_{i+1}^-.
\]
We define $\mc Z_l(n)$ and $\mc Z'(n)$ similarly, with
\[
\mc Z_l(n) := [0,1] \sqcup S^1 \sqcup \cdots \sqcup S^1
\]
and
\[
\mc Z'(n) := S^1 \sqcup \cdots \sqcup S^1.
\]
\end{definition}

Both $\mc F(\mc Z_r(n))$ and $\mc F(\mc Z_l(n))$ are a connected genus-zero sutured surface with $n+1$ boundary components. One boundary component has two sutures; the rest have no sutures and are contained in $S_+$. The sutured surface $\mc F(\mc Z'(n))$ is a connected genus-zero surface with $n+1$ boundary components and no sutures. All boundary components are contained in $S^+$ except the outermost one, which is contained in $S^-$.

The chord diagrams $\mc Z_r(3)$, $\mc Z_l(3)$, and $\mc Z'(3)$ and the sutured surfaces $\mc F(\mc Z_r(3))$, $\mc F(\mc Z_l(3))$, and $\mc F(\mc Z'(3))$ are shown in Figure~\ref{fig:TruncatedSuturedSurface}.

\begin{figure}
\includegraphics[scale=0.7]{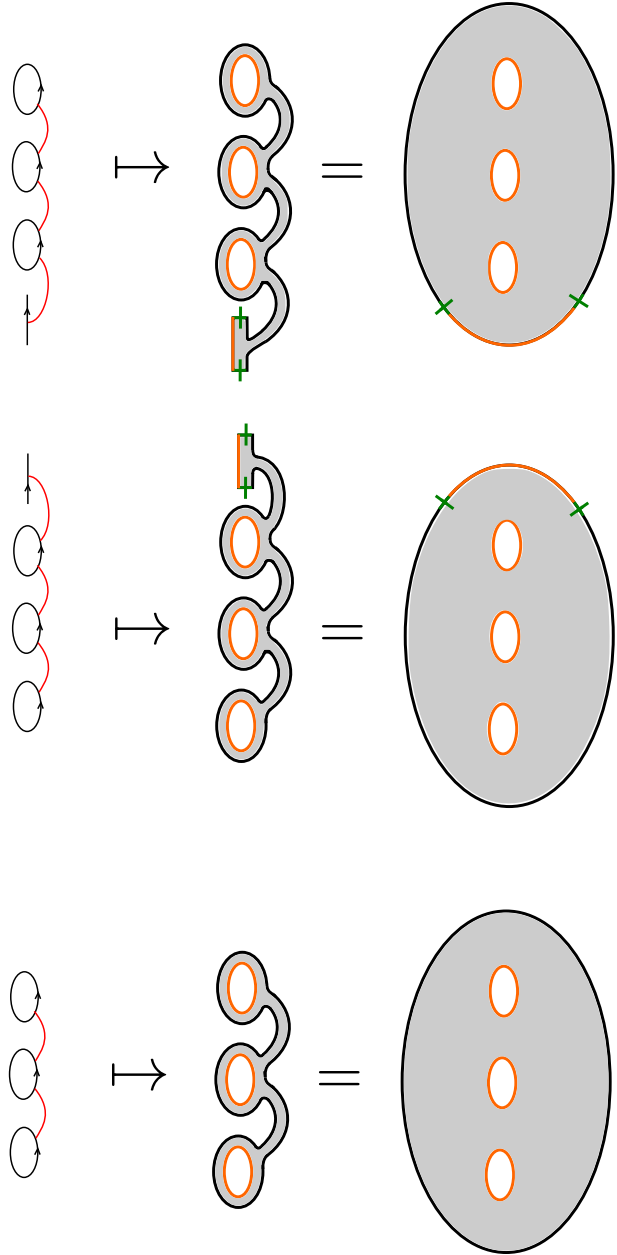}
\caption{The chord diagrams $\mc Z_l(3)$, $\mc Z_r(3)$, and $\mc Z'(3)$, and the sutured surfaces $\mc F(\mc Z_r(3))$, $\mc F(\mc Z_l(3))$, and $\mc F(\mc Z'(3))$.}
\label{fig:TruncatedSuturedSurface}
\end{figure}

\section{Structure of the strands algebras}\label{sec:StrandStructure}

\subsection{Notation and explicit bases for summands of \texorpdfstring{$\sac nk\Sc$}{A(n,k,\Sc)}}\label{sec:StrandsExplicitBasis}
As mentioned below Definition \ref{def:kstrand}, a $k$-strand can be described entirely by specifying the starting points and (constant) speeds of each strand.  When we pass to the subalgebra $\sac nk\Sc$ within $\tsac nk\Sc$, we treat constant strands somewhat differently from non-constant strands, but basis elements $E(s,\vec{c})$ should still be determined by starting points and speeds of each strand of $s$ (together with $\vec{c}$), where a speed of zero corresponds visually to a dashed strand rather than a solid strand.

Because the majority of results in this paper hinge upon the splitting
\[\sac nk\Sc \cong \bigoplus_{\x,\y\in V(n,k)} \Jb_{\x} \sac nk\Sc {\Jb_{\y}},\]
we allow our notation to take the starting idempotent $\x$ as a given.  That is, given a starting idempotent, we seek a notation that allows an immediate combinatorial and visual grasp of any given basis element $E(s,\vec{c})$ for any section $s(0)$ of $\x$.  With all of this in mind, we present the following definition starting from pairs $(s,\vec{c})$ in the pre-strands algebra.

\begin{definition}\label{def:piqi}
Suppose $(s,\vec{c})\in\tsac nk\Sc$.  Let $\Ar(s,\vec{c})$ denote the following combination of a square-free monomial in variables $C_i$ for $i \in \Sc$ together with an array of vectors:
\[\Ar(s,\vec{c}):= \prod_{i\in\Sc}C_i^{\vec{c}(i)} \vv{p_1}{q_1}{1}\vv{p_2}{q_2}{2}\cdots\vv{p_n}{q_n}{n}\]
where each $p_i,q_i$ is defined as follows.
\begin{itemize}
\item If $z_{i}^-$ is the starting point of a non-constant strand $s_a$ of $s$, then $p_i$ is the (constant) speed of this strand; otherwise we set $p_i$ equal to $0$.
\item if $z_{i}^+$ is the starting point of a non-constant strand $s_b$ of $s$, then $q_i$ is the (constant) speed of this strand; otherwise we set $q_i$ equal to $0$.
\end{itemize}
We may also omit columns of all zeros from the array.  In particular the following notation will be used often:
\begin{equation} \label{eq:pq single column}
\vv{p}{q}{i} \quad := \quad \vv{0}{0}{1} \cdots \vv{0}{0}{i-1} \vv{p}{q}{i} \vv{0}{0}{i+1} \cdots \vv{0}{0}{n}.
\end{equation}
\end{definition}
Recall that we have defined our circles $S^1_i$ and basepoints $z_i^\pm$ so that our speeds are integers, and a speed of 2 indicates a degree one map to the circle.  In particular, a strand starts and ends at the same basepoint if and only if its speed is even.

\begin{lemma}\label{lem:piqi well-defined}
Fix some starting idempotent $\x\in V(n,k)$.  For any two basis elements $(s,\vec{c})$ and $(s',\pvec{c}')$ of $\Jb_{\x} \tsac nk\Sc$ with $s(0) \cap M(s(0)) = s(1) \cap M(s(1)) = \varnothing$ and $s'(0) \cap M(s'(0)) = s'(1) \cap M(s'(1)) = \varnothing$, we have $E(s,\vec{c})=E(s',\pvec{c}')$ if and only if $\Ar(s,\vec{c})=\Ar(s',\pvec{c}')$.
\end{lemma}

\begin{proof}
It is clear that if $\vec{c}\neq\pvec{c}'$, then we have both $E(s,\vec{c})\neq E(s',\pvec{c}')$ and $\Ar(s,\vec{c})\neq \Ar(s',\pvec{c}')$.  If $\vec{c}=\pvec{c}'$, Lemma \ref{lem:E(S,T,s)} shows that $E(s,\vec{c})=E(s',\pvec{c}')$ if and only if $s'=s_\ib$ for some $\ib\subset\Ib$, where $\Ib$ is the set of basepoints that are starting points for constant strands of $s$ as usual.

We claim that $s'=s_\ib$ for some $\ib\subset\Ib$ if and only if $\Ar(s,\vec{c})=\Ar(s',\vec{c})$.  Indeed if $s'=s_\ib$, then the only difference between $s$ and $s'$ is the placement of certain constant strands, which the notation of Definition \ref{def:piqi} ignores.  

Conversely, if $\Ar(s,\vec{c})=\Ar(s',\vec{c})$, the nonzero entries of the arrays demand that $s$ and $s'$ have the same non-constant strands, so that they (possibly) differ only in the placement of their constant strands.  Then since $s(0)$ and $s'(0)$ are both sections of $\x$, we must have $s'=s_\ib$ for some $\ib\subset\Ib$.
\end{proof}

Lemma \ref{lem:piqi well-defined} shows that $\Ar(s,\vec{c})$ descends to a well-defined notation for basis elements $E(s,\vec{c})$ in $\Jb_{\x} \sac nk\Sc$ once $\x$ has been fixed.  Notice that an entry of zero in $\Ar(s,\vec{c})$ can mean two different things for the corresponding $k$-strand $s$---it can mean that there is no strand at all at the given basepoint, or it can mean that there is a constant (i.e. speed 0) strand at the given basepoint.  In particular, the case $q_i=p_{i+1}=0$ can mean there are no strands present at all, or that there is a single constant strand starting at either $z_i^+$ or $z_{i+1}^-$, but it cannot mean that there are constant strands at both $z_i^+$ and $z_{i+1}^-$ since we have $s(0)\cap M(s(0))=\emptyset$.

Lemma \ref{lem:piqi well-defined} views this ambiguity as a helpful feature of the notation due to the ambiguity inherent in Lemma \ref{lem:E(S,T,s)}.  However, one might object that the notation alone does not distinguish between a constant strand starting point, an empty basepoint that is matched to a constant strand starting point, and an empty basepoint that is not matched to a constant strand starting point.  (As an extreme example, every idempotent element $\Jb_\x$ is written as an array of all zeros, regardless of $\x$.)  To address this objection, we always work with a fixed starting I-state $\x$, implying the existence (or lack thereof) of strands starting from certain matched pairs of basepoints.  We summarize this point with the following remark.

\begin{remark}
The notation of Definition \ref{def:piqi} is only well-defined for basis elements of $\Jb_\x \sac nk\Sc$ for some fixed beginning I-state $\x$.  It is therefore not helpful as a notation for general basis elements in $\sac nk\Sc$.  For computations in this paper using this notation, we will focus on a single summand $\Jb_\x \sac nk\Sc$ of $\sac nk\Sc$ at a time.
\end{remark}

Visually, once we have fixed a starting I-state $\x$, the notation $C_{i_1}\cdots C_{i_l} \vv{p_1}{q_1}{1}\cdots\vv{p_n}{q_n}{n}$ indicates a specific basis diagram in which solid strands are drawn according to their speeds (the placement of $p_i$ above $q_i$ in the notation is a reminder that $p_i$ is the speed starting from the upper basepoint, while $q_i$ starts from the lower basepoint).  Constant dashed strands are drawn on any matched pair of basepoints that are contained in $\x$ but have no solid strands coming from them. Closed loops are drawn on any cylinder whose $C_i$ variable appears in the monomial. See Figure~\ref{fig:PiQiBasis} for an example.

\begin{figure}
\includegraphics[scale=0.5]{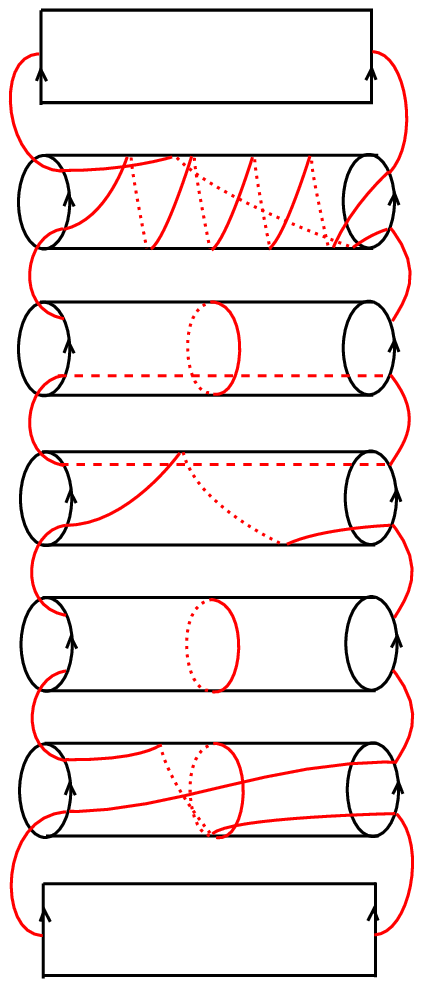}
\caption{The basis element $C_2 C_4 C_5 \binom{1}{9}_1 \binom{0}{0}_2 \binom{0}{2}_3 \binom{0}{0}_4 \binom{1}{1}_5$ of $\sac 56\Sc$, also denoted $C_2 C_4 C_5 \binom{1}{9}_1 \binom{0}{2}_3 \binom{1}{1}_5$, starting at the idempotent $\Jb_{\x}$ where $\x = \{0,1,2,3,4,5\}$ and $\Sc \subset [1,5]$ contains $\{2,4,5\}$.}
\label{fig:PiQiBasis}
\end{figure}

It should be clear that only certain arrays can appear in valid basis elements for a given summand $\Jb_\x \sac nk\Sc$ of the strands algebra.  Furthermore, given a valid array, the ending idempotent of the corresponding strands algebra element is also determined.  Visually all that is required is that no two solid strands start on matched basepoints, and that no two strands (whether solid or dashed) end on basepoints that are either the same or matched.  The following lemma describes the precise combinatorics involved; see Figure~\ref{fig:PiQiLemmaHelper} for reference.

\begin{figure}
\includegraphics[scale=0.5]{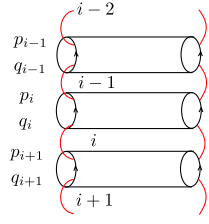}
\caption{Reference diagram for the proof of Lemma~\ref{lem:piqi}.}
\label{fig:PiQiLemmaHelper}
\end{figure}

\begin{lemma}
\label{lem:piqi}
For $\x \in V(n,k)$, the map $(s,c) \mapsto \Ar(s,\vec{c})$ of Definition~\ref{def:piqi} descends to a one-to-one correspondence between basis elements $E(s,\vec{c})$ of ${\Jb_\x}\sac nk\Sc$ and expressions
\[
C_{i_1} \cdots C_{i_l} \vv{p_1}{q_1}{1} \cdots \vv{p_n}{q_n}{n}
\]
satisfying the following conditions.
\begin{enumerate}[(i)]
\item \label{it:no double C} The indices $i_1, \ldots, i_l$ are distinct elements of $\Sc$.
\item \label{it:pq no matched start} For every $i \in [1,n-1]$, $q_{i} p_{i+1} = 0$ (i.e. $q_i$ and $p_{i+1}$ cannot both be nonzero).
\item \label{it:pq start from x} For every $i \in [0,n] \setminus \x$, $q_{i} = p_{i+1} = 0$.
\item \label{it:pq no dots pinzer} For every $i \in [1,n-1]$, $p_i q_{i+1}$ is even (i.e. $p_i$ and $q_{i+1}$ cannot both be odd).
\item \label{it:pq disjoint endpoints for non-const} For every $i \in [1,n]$, if $p_i,q_i$ are both nonzero, then $p_i \equiv q_i \pmod2$.
\item \label{it:pq dots move other dots} Suppose $i \in [1,n]$ and $\{i-1,i\}\subset\x$.  If $p_i$ is odd and $q_i=0$, then we must have $p_{i+1}$ odd (and thus $q_{i+1}=0$ by \eqref{it:pq no dots pinzer} and \eqref{it:pq disjoint endpoints for non-const}).  Symmetrically, if $p_i=0$ and $q_i$ is odd, then we must have $q_{i-1}$ odd (and thus $p_{i-1}=0$ by \eqref{it:pq no dots pinzer} and \eqref{it:pq disjoint endpoints for non-const}).
\end{enumerate}
By convention, we always set $q_0=p_{n+1}=0$.

Given an array expression satisfying the above conditions, let $E(s,\vec{c})$ denote the corresponding basis element of $\Jb_\x \sac nk\Sc$. We have $E(s,\vec{c}) \in \Jb_\x \sac nk\Sc \Jb_\y$ where $\y \in V(n,k)$ is the unique vertex satisfying the following conditions for $i \in [0,n]$.
\begin{itemize}
\item If $i \in \x$ and $q_i$ is odd, then $i-1 \in \y$.
\item If $i \in \x$ and $p_{i+1}$ is odd, then $i+1 \in \y$.
\item If $i \in \x$ and $q_i$ and $p_{i+1}$ are both even, then $i \in \y$.
\end{itemize}
Note that if $q_i$ or $p_{i+1}$ is odd, then $i \in \x$ follows from condition~\eqref{it:pq start from x} above.
\end{lemma}

\begin{proof}
Lemma~\ref{lem:piqi well-defined} implies that the map under consideration is well-defined and injective into the set of all possible array expressions. We want to show that the image of this map lies in the subset consisting of array expressions satisfying conditions \eqref{it:no double C}--\eqref{it:pq dots move other dots}, and that the map is surjective onto this subset. 

Indeed, condition~\eqref{it:no double C} follows from the requirement that $\vec{c}(i) \in \{0,1\}$ for all $i \in \Sc$. Condition~\eqref{it:pq no matched start} follows from $s(0) \cap M(s(0)) = \varnothing$, and condition~\eqref{it:pq start from x} follows from the fact that $s(0)$ is a section of $\x$. Condition~\eqref{it:pq no dots pinzer} follows from $s(1) \cap M(s(1)) = \varnothing$. Condition~\eqref{it:pq disjoint endpoints for non-const} follows from the fact that $s(1)$ is a set of $k$ distinct basepoints, since $s$ is a $k$-strand. Finally, condition~\eqref{it:pq dots move other dots} also follows from $s(1) \cap M(s(1)) = \varnothing$ and the fact that $s$ is a $k$-strand. Thus, array expressions in the image of the map under consideration satisfy the listed conditions.

For surjectivity, given an array expression satisfying the conditions, we can form a putative $k$-strand $s$ by interpreting $p_i$ (respectively $q_i$) as the speed of a strand starting at $z_i^-$ (respectively $z_i^+$), filling in constant strands compatibly with $\x$, and translating the monomial $C_{i_1} \cdots C_{i_l}$ into a function $\vec{c} \in \{0,1\}^{\Sc}$ (this last step is possible by condition~\eqref{it:no double C}). By construction, $s(0)$ consists of $k$ distinct basepoints; the same is true for $s(1)$ by conditions~\eqref{it:pq disjoint endpoints for non-const} and \eqref{it:pq dots move other dots}, so $s$ is a $k$-strand. We have $s(0) \cap M(s(0)) = \varnothing$ by condition~\eqref{it:pq no matched start}, and we have $s(1) \cap M(s(1)) = \varnothing$ by conditions~\eqref{it:pq no dots pinzer} and \eqref{it:pq dots move other dots}. Finally, $s(0)$ is a section of $\x$ by condition~\eqref{it:pq start from x} and the fact that constant strands of $s$ were chosen to be compatible with $\x$.

It follows that the map under consideration is indeed a one-to-one correspondence. The determination of $\y$ from $\x$ and the parity of the speeds $p_i,q_i$ is a straightforward computation; we leave it to the reader.
\end{proof}

\subsubsection{An important special case}\label{sec:ATildeSpecialCase}

The following special case of the summands $\Jb_{\x} \A(n,k,\Sc) \Jb_{\y}$ will be important below.
\begin{definition}\label{def:GenAlgebra}
For $n \geq 1$, we define the \emph{generating algebra} to be
\[
\overline{A}(n,\Sc) := \Jb_{[1,n-1]} \A(n,n-1,\Sc) \Jb_{[1,n-1]}.
\]
While $\overline{A}(n,\Sc)$ is naturally a dg algebra, we will focus below on its structure as a chain complex over $\F_2$.
\end{definition}

\begin{lemma}
\label{lem:BasistdA}
A basis over $\F_2$ of $\overline{A}(n,\Sc)$ is given by square-free monomials in the $C_i$ variables as usual times all arrays $\vv{p_1}{q_1}{1} \cdots \vv{p_n}{q_n}{n}$ such that
\begin{enumerate}
\item \label{it:A1} for $i \in [1, n-1]$, $q_i p_{i+1} = 0$ (i.e. $q_i$ and $p_{i+1}$ cannot both be nonzero);
\item \label{it:A2} $p_1 = q_n = 0$;
\item \label{it:A3} for all $i \in [1,n]$, $p_i \equiv q_i \pmod 2$.
\end{enumerate}
\end{lemma}
\begin{proof}

Condition \eqref{it:A1} here is the same as \eqref{it:pq no matched start} from the general Lemma \ref{lem:piqi}.  In $\overline{A}(n, \Sc)$ where $\x=[1,n-1]$, condition \eqref{it:A2} here is equivalent to \eqref{it:pq start from x} from that lemma.  It remains to see that conditions \eqref{it:pq no dots pinzer}, \eqref{it:pq disjoint endpoints for non-const}, and \eqref{it:pq dots move other dots} from the general lemma are equivalent in $\overline{A}(n, \Sc)$ to condition \eqref{it:A3} here.

We show this by considering negations, assuming the first two conditions here.  Suppose condition \eqref{it:pq no dots pinzer} from the general lemma is false, so there exists some $i$ with $p_i,q_{i+1}$ both odd.  Then since at least one of $q_i,p_{i+1}$ must be zero by \eqref{it:A1}, we have some index $j$ where $p_j,q_j$ have opposite parity, so condition \eqref{it:A3} here is false.

Note also that if condition \eqref{it:pq disjoint endpoints for non-const} or condition \eqref{it:pq dots move other dots} from the general lemma is false, then condition \eqref{it:A3} here is false. 

Conversely, suppose that condition \eqref{it:A3} here is false, so there exists an index $i\in[1,n]$ where $p_i,q_i$ have opposite parity. By condition \eqref{it:pq disjoint endpoints for non-const}, we must have either $p_i = 0$ or $q_i = 0$; without loss of generality, we may assume that $p_i$ is odd and $q_i$ is zero. Let $i$ be the maximal such index. By condition \eqref{it:pq dots move other dots}, we have $i = n$. Since $p_n$ is odd, we have $n \in \y$ for the right idempotent $\Jb_{\y}$ of the basis element under consideration, contradicting the fact that the basis element lives in $\overline{A}(n,\Sc) := \Jb_{[1,n-1]} \A(n,n-1,\Sc) \Jb_{[1,n-1]}$.
\end{proof}

We will also need to consider the following three variants of $\overline{A}(n,\Sc)$.
\begin{definition}\label{def:EdgeAlgebra}
For $n \geq 1$, we define the \emph{edge algebras}:
\begin{itemize}
\item $\overline{A}_\lda(n,\Sc) := \Jb_{[0,n-1]} \A(n,n,\Sc) \Jb_{[0,n-1]}$,
\item $\overline{A}_\rho(n,\Sc) := \Jb_{[1,n]} \A(n,n,\Sc) \Jb_{[1,n]}$, and
\item $\overline{A}_{\lda\rho}(n,\Sc) := \A(n,n+1,\Sc)$.
\end{itemize}
\end{definition} 

The following three lemmas are analogous to Lemma \ref{lem:BasistdA}, and their proofs are omitted.

\begin{lemma}
\label{lem:BasistdAl}
A basis over $\F_2$ of $\overline{A}_\lda(n,\Sc)$ is given by square-free monomials in the $C_i$ as usual times all arrays $\vv{p_1}{q_1}{1} \cdots \vv{p_n}{q_n}{n}$ such that
\begin{enumerate}
\item \label{it:Al1} for $i \in [1, n-1]$, $q_i p_{i+1} = 0$;
\item \label{it:Al2} $q_n = 0$;
\item \label{it:Al3} for $i \in [1,n]$, $p_i \equiv q_i \pmod 2$.
\end{enumerate}
\end{lemma}

\begin{lemma}
\label{lem:BasistdAr}
A basis over $\F_2$ of $\overline{A}_\rho(n,\Sc)$ is given by square-free monomials in the $C_i$ as usual times all arrays $\vv{p_1}{q_1}{1} \cdots \vv{p_n}{q_n}{n}$ such that
\begin{enumerate}
\item \label{it:Ar1} for $i \in [1, n-1]$, $q_i p_{i+1} = 0$;
\item \label{it:Ar2} $p_1 = 0$;
\item \label{it:Ar3} for $i \in [1,n]$, $p_i \equiv q_i \pmod 2$.
\end{enumerate}
\end{lemma}

\begin{lemma}
\label{lem:BasistdAlr}
A basis over $\F_2$ of $\overline{A}_{\lda\rho}(n,\Sc)$ is given by square-free monomials in the $C_i$ as usual times all arrays $\vv{p_1}{q_1}{1} \cdots \vv{p_n}{q_n}{n}$ such that
\begin{enumerate}
\item \label{it:Alr1} for $i \in [1, n-1]$, $q_i p_{i+1} = 0$;
\item \label{it:Alr3} for $i \in [1,n]$, $p_i \equiv q_i \pmod 2$.
\end{enumerate}
\end{lemma}

\subsection{Products and differentials of explicit basis elements}
\label{sec:mult and diff with pq}
In this section we wish to derive formulas for products and differentials of explicit basis elements written in the notation of Definition \ref{def:piqi}. Since $\sac nk\Sc$ is closed under multiplication and the differential, such products and differentials are sums of basis elements; we wish to write these sums explicitly in the same notation.

\subsubsection{Products of basis elements}
As seen in the proof of Proposition \ref{prop:StrandsAlg closed under mult}, in order for the product of two basis elements $E(s,\vec{c})\cdot E(t,\vec{d})$ to be nonzero, we must have some $\ib\subset\Ib,\jb\subset\Jb$ such that $s(1)_\ib=t(0)_\jb$ (see that proof for an explanation of the notation).  If we recall that $q:B\rightarrow B/M\cong [1,n]$ denotes the quotient map, this requirement implies that $q(s(1))=q(t(0))$ as $k$-element subsets of $[1,n]$ (the converse is not true, as we will explore shortly).

Because $E(s,\vec{c})\cdot E(t,\vec{d}) \neq 0$ at least requires $q(s(1))=q(t(0))$, we only write down formulas for the product of $a,a'\in \sac nk\Sc$ in the case where $a\in\Jb_\x \sac nk\Sc \Jb_\y$ and $a'\in\Jb_{\x'} \sac nk\Sc $ with $\x'=\y$. All other cases have trivial product.  Note that this assumption enforces certain conventions in our formulas regarding the meaning of zeros in the second (or third, etc) factor in a product.  For instance, as elements of $\Jb_\x \sac 22\Sc$ with $\x=\set{0,2}$, the formula
\[\vv{1}{0}{1}\vv{0}{0}{2} \cdot \vv{0}{0}{1} \vv{0}{2}{2} = \vv{1}{0}{1}\vv{0}{2}{2}\]
presumes that the starting I-state of $\vv{0}{0}{1} \vv{0}{2}{2}$ is $\y:=\set{1,2}$, and thus the entries $q_1=p_2=0$ for this term are forced to represent constant dashed strands, while the entry $p_1=0$ is forced to represent an empty space.  See Figure \ref{fig:easy pq mult example}.  In short, fixing $\x$ fixes the meaning of the notation for $a\in\Jb_\x \sac nk\Sc$, which in turn fixes $\y$, which then fixes the meaning of the notation for $a'\in\Jb_\y \sac nk\Sc$ when considering a product $a \cdot a'$.

\begin{figure}
\includegraphics[scale=0.5]{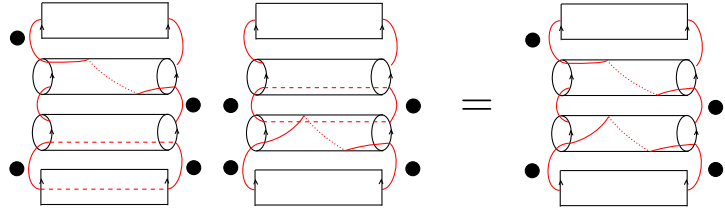}
\caption{A figure illustrating the formula $\binom{1}{0}_{1} \binom{0}{0}_{2} \cdot \binom{0}{0}_{1} \binom{0}{2}_{2} = \binom{1}{0}_{1} \binom{0}{2}_{2}$ in the case when the starting idempotent was $\x=\{0,2\}$, with dots placed on occupied matchings.}
\label{fig:easy pq mult example}
\end{figure}

However, the condition $\x'=\y$ above does not guarantee that $E(s,\vec{c})\cdot E(t,\vec{d}) \neq 0$, or even that there exist $\ib, \jb$ with $s(1)_{\ib} = t(0)_{\jb}$.  The elements may still be not concatenable, and even if they are, we may still create degenerate annuli or bigons upon concatenation (see Definition \ref{def:prestrand alg} and the discussion below equation \eqref{eq:sa nk multiplication}).  If we translate all of our monomials in $C_i$ variables and $p_i,q_i$-arrays into graphs of solid and dashed strands with closed loops, these situations become visually clear.  The following lemma presents the combinatorics that result from this analysis, including the formulas for the nonzero products.

\begin{lemma}
\label{lem:concatenable}
Let $a \in \Jb_{\x} \sac nk\Sc \Jb_\y$ and $a' \in \Jb_{\y} \sac nk\Sc$ be basis elements, represented by expressions
\[
a = C_{i_1} \cdots C_{i_l} \vv{p_1}{q_1}{1} \cdots \vv{p_n}{q_n}{n} \qquad \text{and} \qquad a' = C_{i'_1} \cdots C_{i'_{l'}} \vv{p_1'}{q_1'}{1} \cdots \vv{p_n'}{q_n'}{n}.
\]
Then $a\cdot a'\neq 0$ if and only if for all $i=1, \ldots, n$ the following conditions hold:
\begin{enumerate}[(I)]
\item \label{it:p odd concat} if $p_i$ is odd, then $p_{i+1}' = 0$;
\item \label{it:p even concat} if $p_i \neq 0$ and is even, then $q_{i-1}' = 0$;
\item \label{it:q odd concat} if $q_i$ is odd, then $q_{i-1}' = 0$;
\item \label{it:q even concat} if $q_i \neq 0$ and is even, then $p_{i+1}' = 0$;
\item \label{it:evens no bigons} if $p_i,q_i$ are both even, then $(p_i - q_i)(p_i'-q_i')\geq 0$;
\item \label{it:odds no bigons} if $p_i,q_i$ are both odd, then $(p_i-q_i)(p_i' - q_i') \leq 0$;
\item \label{it:no double loops} no $C_i$ variable appears in the monomial for both $a$ and $a'$.
\end{enumerate}
Moreover, when $a\cdot a'\neq 0$ we also have the following formulas
\[
a \cdot a' = C_{i_1} \cdots C_{i_l} C_{i'_1} \cdots C_{i'_{l'}} \vv{r_1}{s_1}{1} \cdots \vv{r_n}{s_n}{n} \in \Jb_\x \A,
\]
where
\[
r_i :=
\begin{cases}
p_i + q_i' & \text{if $p_i$ is odd}\\
p_i + p_i' & \text{if $p_i$ is even and $q_i$ is even}\\
0 & \text{if $p_i = 0$ and $q_i$ is odd}
\end{cases}
\]
and
\[
s_i :=
\begin{cases}
q_i + p_i' & \text{if $q_i$ is odd}\\
q_i + q_i' & \text{if $q_i$ is even and $p_i$ is even}\\
0 & \text{if $q_i = 0$ and $p_i$ is odd}
\end{cases}
\]
\end{lemma}

In the above lemma, by convention we always set $q_0 = p_{n+1} = q_0' = p_{n+1}' = 0$. In the definition of $r_i$, note that if $p_i$ is even and $q_i$ is odd, then we must have $p_i = 0$, hence we cover all the cases (and similarly for $s_i$).

\begin{proof}
We first prove the `only if' direction. We write $a=E(s,\vec{c})$ and $a'=E(s',\pvec{c}')$ with $\Ib,\Jb$ denoting the starting basepoints of constant strands from $s$ and $s'$ respectively, as in the proof of Proposition \ref{prop:StrandsAlg closed under mult}.

Suppose that item \eqref{it:p odd concat} fails for some fixed index $i$.  Since $p_i$ is odd, the point $z_i^+$ is the endpoint of a non-constant strand of $s$, so $z_{i+1}^- =M(z_i^+)$ is not in $s(1)$, and indeed not in $s(1)_\ib$ for any $\ib\subset\Ib$.  Meanwhile if $p_{i+1}' \neq 0$, then there is a non-constant strand departing from $z_{i+1}^-$ in $s'$, meaning $z_{i+1}^-\in s'(0)_\jb$ for all $\jb\subset\Jb$.  Thus each product in the double sum for $a\cdot a'$ is not concatenable, so $a\cdot a'=0$.  Items \eqref{it:p even concat}, \eqref{it:q odd concat}, and \eqref{it:q even concat} are similar.  Visually, these four items cover the cases when a solid strand in $a'$ has no strand (solid or dashed) in $a$ to concatenate with.

To show that item \eqref{it:evens no bigons} holds, first suppose that $\{i-1,i\}$ is not a subset of $\x$. If the quantity in \eqref{it:evens no bigons} is negative, then either $p_i$ and $q'_i$ are both nonzero or $q_i$ and $p'_i$ are both nonzero. Since $p_i$ and $q_i$ are even, both cases contradict the assumption that the right idempotent of $a$ is the left idempotent of $a'$. If $\{i-1,i\} \subset \x$ and $a \cdot a' \neq 0$, then there exist representatives $(s,\vec{c}), (t,\vec{d})$ for $a,a'$ such that $s,t$ have two strands each on the backbone $S^1_i$. These representatives satisfy the ``no degenerate bigon'' condition of Definition \ref{def:prestrand alg}, implying item \eqref{it:evens no bigons}. 

The argument for item~\eqref{it:odds no bigons} is similar (note that when $p_i,q_i$ are odd, the relative positions of the starting points of the strands swaps, causing a flip in the sign of $p_i'-q_i'$ relative to the phrasing in Definition \ref{def:prestrand alg}). Finally, negating item \eqref{it:no double loops} means that we have $\vec{c}(i) + \pvec{c}'(i) = 2$ for some $i$, so $a \cdot a'$ is zero. Visually, negating item \eqref{it:evens no bigons} or \eqref{it:odds no bigons} results in a degenerate bigon after concatenation, while negating item \eqref{it:no double loops} results in a degenerate annulus.

In the other direction, the proof of Proposition \ref{prop:StrandsAlg closed under mult} shows that $a\cdot a'\neq 0$ so long as there exist some $\ib\subset\Ib,\jb\subset\Jb$ with $s(1)_\ib=s'(0)_\jb$ and such that the concatenation $s_\ib \cdot s'_\jb$ has no degenerate bigons or annuli.  Annulus creation violates item \eqref{it:no double loops}. Bigon creation between two strands must take place on some fixed backbone $S^1_i$; the reader may verify that the result violates one of items \eqref{it:evens no bigons} or \eqref{it:odds no bigons} depending on the parity of $p_i,q_i$.  Thus it is enough to show that, if we assume items \eqref{it:p odd concat}, \eqref{it:p even concat}, \eqref{it:q odd concat}, and \eqref{it:q even concat} (along with $q(s(1))=q(s'(0))=\y$), then we can find the requisite $\ib,\jb$ making $s_\ib,s'_\jb$ concatenable.

Suppose $s(1)\neq s'(0)$, so there exists some basepoint $z_i^{\pm}\in s(1) \setminus s'(0)$.  Since $q(s(1))=q(s'(0))$, the matched basepoint $M(z_i^{\pm})$ must be an element of $s'(0)$.  If the basepoint $z_i^{\pm}\in s(1)$ is the endpoint of a non-constant strand, then we are in one of the cases covered by items \eqref{it:p odd concat}, \eqref{it:p even concat}, \eqref{it:q odd concat}, and $\eqref{it:q even concat}$, forcing $M(z_i^\pm)$ to be the starting point of a constant strand in $s'$.  This means we can choose $\jb=\set{M(z_i^\pm)}$; using Lemma \ref{lem:E(S,T,s)}, we can replace $s'$ by $s'_\jb$ and begin again with one fewer element in $s(1) \setminus s'(0)$.  On the other hand, if $z_i^{\pm}\in s(1)$ was the endpoint of a constant strand, then we have $z_i^{\pm}\in s(0)$ as well and we can choose $\ib=\set{z_i^\pm}$ to accomplish the same goal after replacing $s$ with $s_\ib$.  In either case, we decrease the size of $s(1) \setminus s'(0)$. This process does not change the elements $a,a'$, so it preserves the entire list of conditions above. Since the sets $s(1)$ and $s'(0)$ are finite, we must eventually make $s$ and $s'$ concatenable, proving the characterization of nonzero products $a \cdot a'$.

Assuming that $a\cdot a'\neq 0$, choose $(s,\vec{c})$ and $(s',\pvec{c}')$ with $a = E(s,\vec{c})$ and $a'=E(s',\pvec{c}')$ such that $s \cdot s'$ is nondegenerate. Equation \eqref{eq:sa nk multiplication} shows us that to compute $a \cdot a'$, we need only take the product of $(s,\vec{c})$ and $(s',\pvec{c}')$ in $\tsac nk\Sc$, where speeds of various strands add.  With this observation in mind, the formulas above follow so long as one recalls that strands with odd speeds start and end at opposite basepoints, essentially reversing the role of $p_i'$ and $q_i'$.

\end{proof}

\subsubsection{Differentials of basis elements}
According to Proposition \ref{prop:AnkSubcomplex}, the differential of Definition \ref{def:differential} descends to the strands algebra $\sac nk\Sc$; the proofs of Lemma \ref{lem:delc preserves Ank} and Lemma \ref{lem:del0 preserves Ank} show how to compute the differential on $\sac nk\Sc$.  Therefore, we will refrain from a detailed proof of the resulting formulas when applying this reasoning to basis elements written in our $p_i,q_i$-notation.

\begin{lemma}\label{lem:general pq differential}
Let
\[
a=C_{i_1} \cdots C_{i_l} \vv{p_1}{q_1}{1}\cdots\vv{p_n}{q_n}{n}
\]
be a basis element of the summand $\Jb_\x \sac nk\Sc$.  For $1 \leq i \leq n$, $\de_i^0 a$ is the element of $\Jb_{\x} \sac nk\Sc$ given as follows:
\begin{equation}\label{eq:general pq differential bad idemp}
\de^0_i a = 0 \qquad \text{if $\set{i-1,i}\not\subset\x$},
\end{equation}
and otherwise, defining $m_i:=\min(p_i,q_i)$ and $M_i:=\max(p_i,q_i)$,
\begin{equation}\label{eq:general pq differential good idemp}
\de^0_i a = \begin{cases}
C_{i_1} \cdots C_{i_l} \vv{p_1}{q_1}{1}\cdots\vv{m_i+1}{M_i-1}{i} \cdots \vv{p_n}{q_n}{n} + \\
C_{i_1} \cdots C_{i_l} \vv{p_1}{q_1}{1}\cdots\vv{M_i-1}{m_i+1}{i} \cdots \vv{p_n}{q_n}{n}
 & \text{if $q_{i-1}=p_{i+1}=0$ and $M_i-m_i\geq 4$;}\\
 \\
C_{i_1} \cdots C_{i_l} \vv{p_1}{q_1}{1}\cdots\vv{m_i+1}{M_i-1}{i} \cdots \vv{p_n}{q_n}{n} & \text{if $q_{i-1}=p_{i+1}=0$ and $M_i-m_i=2$;}\\
\\
0 & \text{if either of $q_{i-1},p_{i+1}$ is $\neq 0$,}
\\
 & \text{or if $(M_i-m_i)=0$.}
\end{cases}.
\end{equation}
\end{lemma}
The ellipses of equation \eqref{eq:general pq differential good idemp} are meant to indicate that all entries of the array for $a$ have been kept the same except for those in the $i^\text{th}$ column.
\begin{proof}
If $\set{i-1,i}\not\subset\x$, write $a = E(s,\vec{c})$. For each term $(s_{\ib},\vec{c})$ in the sum defining $E(s,\vec{c})$, the $k$-strand $s_{\ib}$ can have at most one strand on the backbone $S^1_i$, so $\de^0_i a = 0$.  Similarly, if either of $q_{i-1},p_{i+1}$ is $\neq 0$, each $s_\ib$ can have at most one strand on $S^1_i$, so $\de^0_i a=0$.  Also, $M_i-m_i=0$ if and only if $p_i=q_i$, again indicating that $\de^0_i a=0$.  The only cases remaining are those where $\set{i-1,i}\subset\x$, $q_{i-1}=p_{i+1}=0$, and $p_i\neq q_i$.  In these cases, if both $p_i,q_i$ are nonzero, we have the formula immediately from Definition \ref{def:differential}. If $p_i = 0$ and $q_i \neq 0$, recall that $q_{i-1} = 0$ and $i-1 \in \x$; if $q_i = 0$ and $p_i \neq 0$, recall that $p_{i+1} = 0$ and $i \in \x$. The proof of item~\eqref{it:Del0DashedConstant} in Lemma~\ref{lem:del0 preserves Ank} now implies the stated formula.
\end{proof}

\begin{lemma}\label{lem:general pq C-differential}
Let
\[
a=C_{i_1} \cdots C_{i_l} \vv{p_1}{q_1}{1}\cdots\vv{p_n}{q_n}{n}
\]
be a basis element of the summand $\Jb_\x \sac nk\Sc$.  For $1 \leq i \leq n$, $\de_i^c a$ is the element of $\Jb_{\x} \sac nk\Sc$ given as follows:
\begin{equation}\label{eq:general pq C-differential no C}
\de_i^c a =0 \qquad \text{if $C_i$ does not appear in the monomial $C_{i_1} \cdots C_{i_l}$};
\end{equation}
otherwise, as long as $p_i,q_i$ are not both zero,
\begin{equation}\label{eq:general pq C-differential}
\de_i^c a = \begin{cases}
\frac{C_{i_1} \cdots C_{i_l}}{C_i} \vv{p_1}{q_1}{1}\cdots\vv{p_i+2}{q_i}{i} \cdots \vv{p_n}{q_n}{n} + \\
\frac{C_{i_1} \cdots C_{i_l}}{C_i} \vv{p_1}{q_1}{1}\cdots\vv{p_i}{q_i+2}{i} \cdots \vv{p_n}{q_n}{n}
 & \text{if $p_i=q_i\neq 0$;}\\
 \\
\frac{C_{i_1} \cdots C_{i_l}}{C_i} \vv{p_1}{q_1}{1} \cdots \vv{p_i+2}{q_i}{i} \cdots \vv{p_n}{q_n}{n} & \text{if $p_i>q_i$;}\\
\\
\frac{C_{i_1} \cdots C_{i_l}}{C_i} \vv{p_1}{q_1}{1} \cdots \vv{p_i}{q_i+2}{i} \cdots \vv{p_n}{q_n}{n} & \text{if $p_i<q_i$.}
\end{cases}
\end{equation}
If $p_i=q_i=0$, we have a potential sum of terms depending on $\x$ and the entries $q_{i-1}$ and $p_{i+1}$ as follows:
\begin{equation}\label{eq:general pq C-differential zeroes}
\begin{split}
\de_i^c \left( C_{i_1} \cdots C_{i_l} \vv{p_1}{q_1}{1}\cdots\vv{0}{0}{i}\cdots\vv{p_n}{q_n}{n} \right)= 
& \delta_{i-1} \left( \frac{C_{i_1} \cdots C_{i_l}}{C_i} \vv{p_1}{q_1}{1} \cdots \vv{2}{0}{i} \cdots \vv{p_n}{q_n}{n} \right) + \\
& \epsilon_{i} \left( \frac{C_{i_1} \cdots C_{i_l}}{C_i} \vv{p_1}{q_1}{1} \cdots \vv{0}{2}{i} \cdots \vv{p_n}{q_n}{n} \right)
\end{split}
\end{equation}
where $\delta_{i-1}$ and $\epsilon_i$ are defined as
\[ \delta_{i-1} := \begin{cases}
1 & \text{if $i-1\in\x$ and $q_{i-1}=0$}\\
0 & \text{otherwise}
\end{cases}, \qquad
\epsilon_i := \begin{cases}
1 & \text{if $i\in\x$ and $p_{i+1}=0$}\\
0 & \text{otherwise}
\end{cases}.
\]

\end{lemma}
\begin{proof}
Equations \eqref{eq:general pq C-differential no C} and \eqref{eq:general pq C-differential} are straightforward translations of Definition \ref{def:differential} into this notation (note that the ambiguity of a zero entry is irrelevant for $\de_i^c$ if there is another strand of positive speed on $S^1_i$).  Meanwhile, equation \eqref{eq:general pq C-differential zeroes} splits $\de_i^c$ into a sum of terms---the first term appears if and only if the entry $p_i=0$ refers to a dashed strand at $z_i^-$ in the visual representation for $a$, while the second term appears if and only if the entry $q_i=0$ refers to a dashed strand at $z_i^+$. One can check that equation \eqref{eq:general pq C-differential zeroes} also follows from Definition \ref{def:differential}.
\end{proof}

\subsection{More results on the strands algebra}\label{sec:StrandsMoreResults}

Because the idempotents $\Jb_\x\in\sac nk\Sc$ are indexed by subsets $\x\in V(n,k)$, we can extend some of the terminology of Section~\ref{sec:MMW1Review} to our current setting. By Lemma \ref{lem:FarStatesStrandAlgZero}, we already know that $\Jb_\x \sac nk\Sc \Jb_\y = 0$ when $\x$ and $\y$ are far as in Definition \ref{def:ReviewNotFarCrossed}.  The following lemma relates our $p_i,q_i$-notation to the entries $v_i(\x,\y)$ of the relative weight vectors (Definition \ref{def:ReviewVnk}) and the notion of crossed lines from $\x$ to $\y$ (Definition~\ref{def:CL}).

\begin{lemma}
\label{lem:cl}
Let $a = C_{i_1} \cdots C_{i_l} \vv{p_1}{q_1}{1} \cdots \vv{p_n}{q_n}{n} \in \Jb_{\x} \sac nk\Sc \Jb_{\y}$ be a standard basis element of the strands algebra.
Then $\x$ and $\y$ are not far (in the sense of Definition \ref{def:ReviewNotFarCrossed}), and the following conditions are equivalent:
\begin{enumerate}
\item line $i$ from $\x$ to $\y$ is crossed;
\item for any $k$-strand $s$ such that $a=E(s,\vec{c})$, $s$ has only one strand mapping to the $i$-th circular backbone, and this strand connects either $z_i^-$ to $z_i^+$ or $z_i^+$ to $z_i^-$;
\item $p_i \not\equiv q_i \pmod2$.
\end{enumerate}
Moreover, in such a case, the following are equivalent too:
\begin{enumerate}
\item $v_i(\x, \y) = 1$ (resp.~$v_i(\x, \y) = -1$);
\item the strand of $s$ on the $i$-th circular backbone connects $z_i^-$ to $z_i^+$ (resp.~$z_i^+$ to $z_i^-$);
\item $q_i = 0$ (resp.~$p_i=0$).
\end{enumerate}
\end{lemma}
\begin{proof}
The claim that $\x$ and $\y$ are not far follows from Lemma \ref{lem:FarStatesStrandAlgZero}.  If we write $a=E(s,\vec{c})$, we see that for any $i\in[1,n]$ we have
\[
|\x \cap [i, n]| = |s(0) \cap \set{z_i^+, z_{i+1}^\pm, \ldots, z_n^\pm, z_{n+1}^-}|
\]
and
\[
|\y \cap [i, n]| = |s(1) \cap \set{z_i^+, z_{i+1}^\pm, \ldots, z_n^\pm, z_{n+1}^-}|.
\]
The strands of $s$ on the circular backbones $S^1_j$ for $j\geq i+1$ (and the final linear backbone) give a one-to-one correspondence between
\[
s(0) \cap \set{z_{i+1}^\pm, \ldots, z_n^\pm, z_{n+1}^-} \qquad \mbox{and} \qquad s(1) \cap \set{z_{i+1}^\pm, \ldots, z_n^\pm, z_{n+1}^-},
\]
so there are only three possibilities for $v_i(\x,\y) = |\y \cap [i, n]| - |\x \cap [i, n]|$.
\begin{itemize}
\item If $z_i^+ \in s(0) \setminus s(1)$, then $v_i(\x,\y) = -1$ (line $i$ is crossed); in such a case, there must be only a single strand on $S^1_i$ starting from $z_i^+$ and ending at $z_i^-$, which is equivalent to $p_i=0$ and $q_i$ odd.
\item If $z_i^+ \in s(1) \setminus s(0)$, then $v_i(\x,\y) = 1$ (line $i$ is crossed); in such a case, there must be only a single strand on $S^1_i$ ending at $z_i^+$ and starting from $z_i^-$, which is equivalent to $q_i=0$ and $p_i$ odd.
\item If $z_i^+ \in s(0) \cap s(1)$ or $z_i^+ \not\in s(0) \cup s(1)$, then $v_i(\x,\y) = 0$. If $z_i^+ \in s(0) \cap s(1)$, then either $s$ has a single strand from $z_i^+$ to $z_i^+$ ($p_i=0,q_i$ even), or $s$ has at least two strands on the $i$-th cylinder ($p_i\equiv q_i \pmod 2$). If $z_i^+ \not\in s(0) \cup s(1)$, then $s$ can have no strand starting from or ending in $z_i^+$ ($q_i=0$ and $p_i$ even).
\end{itemize}
The assertions of the lemma follow.
\end{proof}

\begin{corollary}
\label{cor:pi+qimod2}
Let $\x, \y\in V(n,k)$. If 
\[
a = C_{i_1} \cdots C_{i_l} \vv{p_1}{q_1}{1} \cdots \vv{p_n}{q_n}{n}
\]
and 
\[
a' = C_{i'_1} \cdots C_{i'_{l'}} \vv{p_1'}{q_1'}{1} \cdots \vv{p_n'}{q_n'}{n}
\]
are basis elements of $\Jb_\x \sac nk\Sc \Jb_\y$, then for all $i = 1, \ldots, n$, we have $p_i + q_i \equiv p_i' + q_i' \pmod2$.
\end{corollary}
\begin{proof}
By Lemma \ref{lem:cl}, the parity of $p_i + q_i$ is determined by whether or not line $i$ is crossed, which depends only on $\x$ and $\y$.
\end{proof}

For I-states $\x$ and $\y$ that are not far, there is a unique minimally winding basis element of $\sac nk\Sc$, which should be viewed as an analogue to the generator $f_{\x, \y}$ of $\B(n,k,\Sc)$ as in \cite[\defOSzStyleDef]{MMW1}. Visually, this element is found by placing speed zero strands for each stationary dot (in the sense of the motions of dots in \cite[\secGraphicalInterp]{MMW1}), and placing speed one strands for each moving dot. The following lemma presents the combinatorics of this construction; see Figure \ref{fig:gXYexample} for an example.

\begin{lemma}
\label{lem:gxy}
If $\x$ and $\y$ are not far, then there exists a unique basis element
\[
g_{\x, \y} = \vv{p_1}{q_1}{1} \cdots \vv{p_n}{q_n}{n} \in {\Jb_{\x}} \sac nk\Sc {\Jb_{\y}}
\]
with the following properties:
\begin{itemize}
\item $p_i = 1$ if $v_i(\x, \y) = 1$;
\item $q_i = 1$ if $v_i(\x, \y) = -1$;
\item $p_i$ and $q_i$ are $0$ in all other cases.
\end{itemize}
Moreover, if $C_{i_1} \cdots C_{i_l} \vv{r_1}{s_1}{1} \cdots \vv{r_n}{s_n}{n}$ is a basis element of ${\Jb_{\x}} \sac nk\Sc {\Jb_{\y}}$, then, for all $i = 1, \ldots, n$, we have $r_i \geq p_i$ and $s_i \geq q_i$.
\end{lemma}

\begin{proof}
The three properties listed above completely determine all entries $p_i$ and $q_i$. We need to check that such an array of vectors defines an element of ${\Jb_{\x}} \sac nk\Sc {\Jb_{\y}}$, i.e., that it satisfies the properties of Lemma \ref{lem:piqi}. Condition \eqref{it:no double C} is automatic.

If $q_i \neq 0$, then $v_i(\x, \y) = -1$, so $v_{i+1}(\x, \y)$ must be $-1$ or $0$, hence $p_{i+1} = 0$. Thus $q_i p_{i+1} = 0$, and condition \eqref{it:pq no matched start} holds.

If $p_i$ is odd, then $v_i(\x,\y) = 1$ so that $v_{i+1}(\x,\y)$ must be $1$ or $0$, hence $q_{i+1}=0$.  Thus $p_iq_{i+1} = 0$ is even, and condition \eqref{it:pq no dots pinzer} holds.

If $i \not\in \x$, then there are two cases. If $i \not \in \y$, then $v_i(\x, \y) = v_{i+1}(\x, \y) = 0$, otherwise $\x$ and $\y$ would be far. If $i \in \y$, then $v_i(\x, \y) = v_{i+1}(\x, \y) + 1$, so $v_i(\x, \y) \neq -1$ and $v_{i+1}(\x,\y) \neq 1$. In all these cases, we have $q_i = p_{i+1} = 0$ and condition \eqref{it:pq start from x} holds. Condition \eqref{it:pq disjoint endpoints for non-const} is immediate because $p_i$ and $q_i$ are never both nonzero.

Finally, if $p_i$ is odd and $q_i=0$ (respectively $p_i=0$ and $q_i$ is odd), we have $v_i(\x,\y)=1$ (respectively $v_i(\x,\y)=-1$).  Assuming $\{i-1,i\}\subset\x$, we then have $v_{i+1}(\x,\y)=1$ (respectively $v_{i-1}(\x,\y)=-1$), so that $p_{i+1}=1$ is odd (respectively $q_{i-1}=1$ is odd). Thus, condition \eqref{it:pq dots move other dots} is also satisfied.

Thus $g_{\x, \y} \in \Jb_\x \sac nk\Sc$. Let $\y'$ denote the ending I-state of $g_{\x,\y}$, so that $g_{\x, \y} \in \Jb_\x \sac nk\Sc \Jb_{\y'}$. If $i \in \x$ and $p_{i+1}$ is odd, then $i+1 \in \y'$. On the other hand, $p_{i+1}$ is odd if and only if $v_{i+1}(\x, \y) = 1$, which, by the closeness of $\x$ and $\y$, implies that $i \in \x$ and $i+1 \in \y$. Analogously, if $i \in \x$ and $q_i$ is odd, then we deduce both $i-1 \in \y'$ and $i-1 \in \y$.

If $i \in \x$ and $q_i$ and $p_{i+1}$ are both even, then $v_i(\x, \y) \geq 0$ and $v_{i+1}(\x, \y) \leq 0$. By the fact that
\[
0 \leq v_{i}(\x,\y) - v_{i+1}(\x,\y) = \delta_{i \in \y} - \delta_{i \in \x} = \delta_{i \in \y} - 1,
\]
we deduce that $i \in \y$. Thus, by Lemma \ref{lem:piqi}, $\y$ and $\y'$ coincide.

Lastly, to check that $r_i \geq p_i$ and $s_i \geq q_i$ for the general basis element of $\Jb_{\x} \sac nk\Sc \Jb_{\y}$, we use Lemma \ref{lem:cl}. If $v_i(\x, \y) = 1$, then $r_i \not\equiv0 \pmod2$, so $r_i \geq 1 = p_i$. If $v_i(\x, \y) = -1$, then $s_i \not\equiv0 \pmod2$, so $s_i \geq 1 = q_i$.
\end{proof}

\begin{figure}
\includegraphics[scale=0.5]{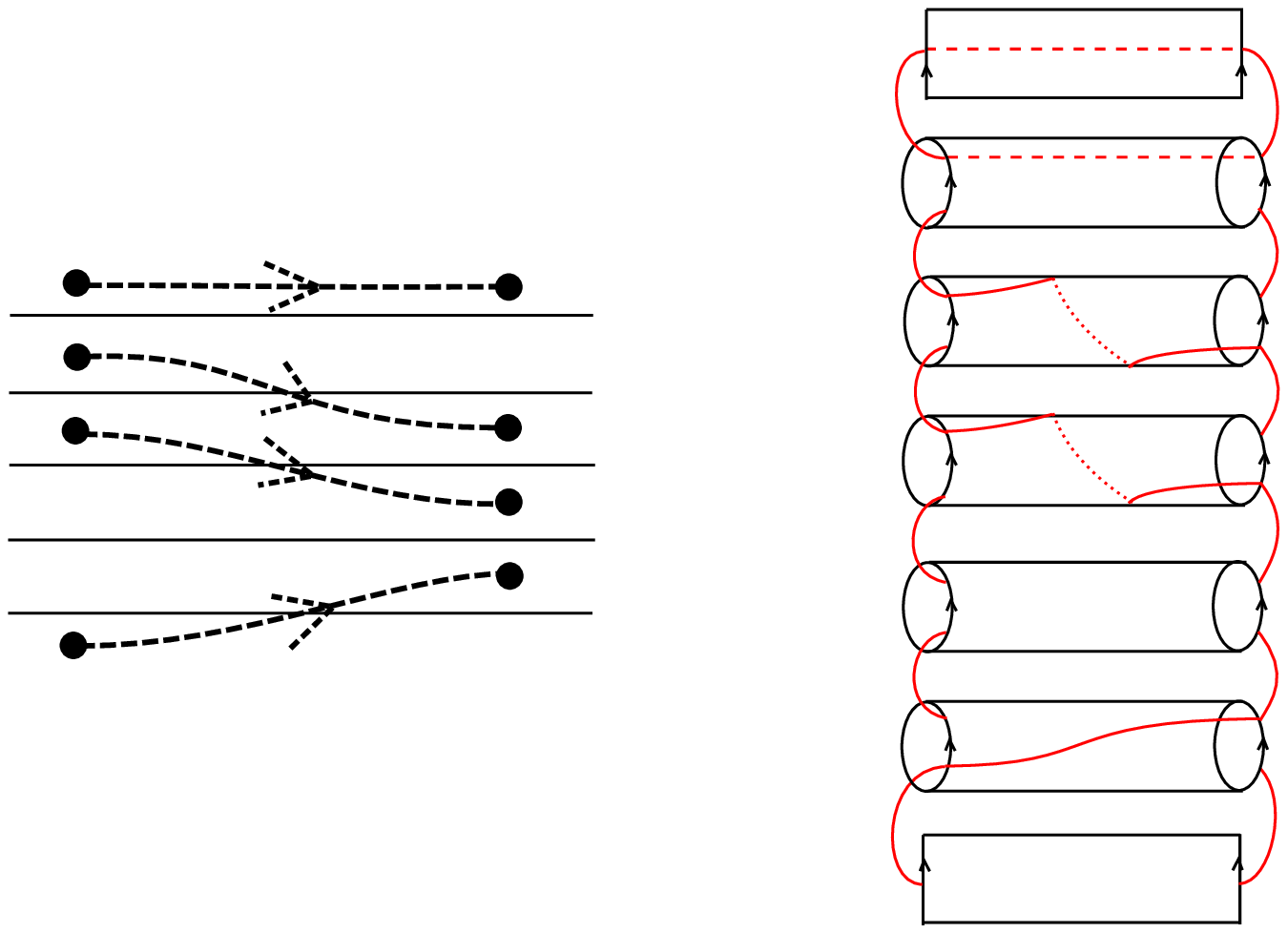}
\caption{The elements $f_{\x,\y}$ of $\Path(K(5,4))$ and $g_{\x,\y}$ of $\mc A(5,4)$ for $\x = \{0,1,2,5\}$ and $\y = \{0,2,3,4\}$.}
\label{fig:gXYexample}
\end{figure}

\begin{corollary}
\label{cor:xAy}
The summand ${\Jb_{\x}} \sac nk\Sc {\Jb_{\y}}$ of the strands algebra is nonzero if and only if $\x$ and $\y$ are not far.
\end{corollary}

\section{Gradings}
\label{sec:gradings}

In this section we endow our strands algebra $\sac nk\Sc$ with several gradings, defined combinatorially in terms of the $(p_i,q_i)$-notation of Definition \ref{def:piqi}. We then illustrate the relationship between our gradings and the group-valued gradings of \cite{LOT} in Sections~\ref{sec:UnrefinedGradings} and \ref{sec:RefinedGradings}.  Throughout this section, we extend the function $\vec{c}\in\{0,1\}^\Sc$ to a function $\vec{c}\in\{0,1\}^{[1,n]}$ by declaring that $\vec{c}(i)=0$ if $i\notin\Sc$.

\subsection{The gradings, combinatorially}\label{sec:CombinatorialStrandsGradings}

\begin{definition}\label{def:gradings}
Let $a=E(s,\vec{c})\in\Jb_\x\sac nk\Sc$ be a basis element; we can write $a$ as
\[
a=C_1^{\vec{c}(1)} \cdots C_n^{\vec{c}(n)} \vv{p_1}{q_1}{1}\cdots\vv{p_n}{q_n}{n}.
\]
Let $\chi_{\Sc}: [1,n] \to \{0,1\}$ be the indicator function of $\Sc \subset [1,n]$ ($0$ if $i \notin \Sc$ and $1$ if $i \in \Sc$). As in Definition~\ref{def:ReviewOSzGradings}, let $\tau_1,\beta_1, \ldots, \tau_n, \beta_n$ denote the standard basis of $\Z^{2n}$, while $e_1,\ldots,e_n$ denotes the standard basis of $\Z^n$. We have the following four notions of a degree for $a$:
\begin{enumerate}
\item The \emph{Maslov grading} $\m:\sac nk\Sc \rightarrow \Z$ is defined by
\[
\m(a):= \sum_{i=1}^n \left( \frac{|p_i - q_i|}{2} - (p_i + q_i) + (-1)^{\chi_{\Sc}(i)} \left(\vec{c}(i) + \frac{p_i + q_i}{2} \right)\right).
\]
\item The \emph{unrefined Alexander grading} $w^{\un}:\sac nk\Sc \rightarrow \Z^{2n}$ is defined by
\[ 
w^{\un}(a) := \sum_{i=1}^n (w^{\un,\tau}_i(a) \tau_i +w^{\un,\beta}_i(a) \beta_i)
\]
where
\[ 
w^{\un,\tau}_i(a) := \vec{c}(i) + \floor*{\frac{p_i}2}+\ceil*{\frac{q_i}2}
\]
and
\[
w^{\un,\beta}_i(a) := \vec{c}(i) + \ceil*{\frac{p_i}2}+\floor*{\frac{q_i}2}.
\]

\item The \emph{refined Alexander grading} $w:\sac nk\Sc \rightarrow (\frac{1}{2}\Z)^n$ is defined by
\[ w(a) = \sum_{i=1}^n \left(\vec{c}(i) + \frac{p_i + q_i}2\right)e_i. \]
As in Definition~\ref{def:ReviewOSzGradings}, $w$ is recovered from $w^{\un}$ by the homomorphism sending both $\tau_i,\beta_i$ to $\frac 12 e_i$.

\item The \emph{single Alexander grading} $\Alex: \sac nk\Sc \to \frac{1}{2}\Z$ is defined by
\[
\Alex(a) = \sum_{i=1}^{n} (-1)^{\chi_{\Sc}(i)}\left(\vec{c}(i) + \frac{p_i + q_i}{2}\right).
\]
\end{enumerate}
\end{definition}

Visually, the entries of the unrefined Alexander grading count how often any strand traverses each arc between basepoints on the circular backbones (there are $2n$ such arcs), while the entries of the refined Alexander grading count the total winding number of all strands on each circular backbone. The Maslov grading is a bit more complicated.

With these definitions in place, the reader can use Lemmas \ref{lem:concatenable} and \ref{lem:general pq differential} to verify the homogeneity of both multiplication and differentiation, as described by the following proposition.

\begin{proposition}
\label{prop:effectsofde}
For $a \in \sac nk\Sc$ homogeneous with respect to any of the following gradings, we have:
\[
\m(\de a) = \m(a) - 1; \qquad w^{\un}(\de a) = w^{\un}(a); \qquad w(\de a) = w(a); \qquad \Alex(\de a) = \Alex(a).
\]
Moreover, all of these gradings are additive with respect to multiplication in the algebra: for $a, b \in \sac nk\Sc$ $g$-homogeneous (where $g$ is any of the gradings introduced so far) and such that $a \cdot b \neq 0$, we have
\[
g(a \cdot b) = g(a) + g(b).
\]
\end{proposition}

\subsection{Where the unrefined gradings come from, topologically}\label{sec:UnrefinedGradings}
Lipshitz, Ozsv\'ath, and Thurston discuss gradings on the strands algebra associated to a pointed matched circle in \cite[Section 3.3]{LOT}.  Their ideas are easily carried over to the case of a general chord diagram $\mathcal{Z}=(\mc Z,B,M)$.
The unrefined gradings of \cite{LOT} take values in a subgroup of a central extension by $\frac12\Z$ of $H_1(\mc Z,B)$ determined by $M$; in general, such an extension gives a nonabelian group. We will see that, in the case of our specific chord diagram $\mathcal{Z}(n)=(\mc Z(n),B,M)$, this extension is in fact trivial, leading to the unrefined grading group of Definition \ref{def:gradings}. We begin with a definition.

\begin{definition}[cf. \cite{LOT}]\label{def:general multiplicity}
Let $\mathcal{Z}=(\mc Z,B,M)$ be a chord diagram as in Definition \ref{def:chord diagram}.  For $p\in B$ and $\alpha\in H_1(\mc Z,B)$, the \emph{multiplicity} $m(\alpha,p)$ of $p$ in $\alpha$ is the average multiplicity with which $\alpha$ covers the two arcs on either side of $p$.  Extend $m$ to a map $H_1(\mc Z,B)\times H_0(B)\rightarrow \frac{1}{2}\Z$ bilinearly.
\end{definition}

Using the multiplicity $m$, \cite{LOT} define a bilinear ``linking'' function $L:H_1(\mc Z,B)\times H_1(\mc Z,B)\rightarrow \frac{1}{2}\Z$ as
\[L(\alpha_1,\alpha_2):=m(\alpha_2,\de \alpha_1)\]
where $\de$ is the connecting homomorphism $\de: H_1(\mc Z,B)\rightarrow H_0(B)$ from the long exact sequence for the pair $(\mc Z,B)$. Note that $L$ is antisymmetric; equivalently, $L(\alpha,\alpha) = 0$ for any $\alpha$. Using $L$, we can define a group $G'(\mathcal{Z})$ as follows.

\begin{definition}[Definition 3.33 of \cite{LOT}]\label{def:general unrefined group}
Define $\varepsilon: H_1(\mc Z,B) \to (\frac{1}{2} \Z)/\Z$ by
\[
\varepsilon(\alpha) = \frac{1}{4}\#(\textrm{parity changes in }\alpha) \mod 1,
\]
where a parity change in $\alpha$ is a point $p \in B$ such that $m(\alpha,p)$ is a half-integer. The \emph{unrefined grading group} $G'(\mathcal{Z})$ is the subset of $\frac{1}{2}\Z\times H_1(\mc Z,B)$ consisting of pairs $(j,\alpha)$ satisfying
\[
j \equiv \varepsilon(\alpha) \textrm{ mod } 1.
\]
The multiplication on $G'(\Zc)$ is given by
\[(j_1,\alpha_1)\cdot(j_2,\alpha_2):= (j_1+j_2+L(\alpha_1,\alpha_2),\alpha_1+\alpha_2);\]
one can check that the condition $j \equiv \varepsilon(\alpha)$ mod $1$ is satisfied for the product.
\end{definition}

For our chord diagram $\Zc(n)$, the linking function is trivial as shown below.
\begin{lemma}\label{lem:L vanishes}
Consider the chord diagram $\mathcal{Z}(n)=(\mc Z(n),B,M)$ of Definition \ref{def:Z(n)}. For any $\alpha_1, \alpha_2 \in H_1(\mc Z(n), B)$, we have $L(\alpha_1,\alpha_2) = 0$.
\end{lemma}
\begin{proof} Any standard basis element $\alpha_1\in H_1(\mc Z(n),B) \cong \Z^{2n}$ will lie entirely on some $S^1_i$, and so will have either $\de \alpha_1=0$ or $\de \alpha_1 = \pm(z_i^+-z_i^-)$.  Since the arcs on either side of $z_i^+$ are the same as the arcs on either side of $z_i^-$, we have $m(\alpha_2,\de \alpha_1)=0$ in all cases.
\end{proof}

\begin{corollary} For the chord diagram $\mathcal{Z}(n)$, the unrefined grading group $G'(\mathcal{Z}(n))$ of \cite{LOT} is isomorphic to the subgroup of $\frac{1}{2}\Z\times H_1(\mc Z(n),B) \cong \frac{1}{2}\Z\times \Z^{2n}$ consisting of pairs $(j,\alpha)$ with $j \equiv \varepsilon(\alpha)$ mod $1$.
\end{corollary}

The next lemma shows that $G'(\mathcal{Z}(n))$ is non-canonically isomorphic to $\Z \times \Z^{2n}$.
\begin{lemma}\label{lem:NonCanonicalGradingIso}
Write $\tau_1, \beta_1, \ldots, \tau_n, \beta_n$ for the generators of $H_1(\mc Z(n),B) \cong \Z^{2n}$. For $1 \leq i \leq n$, choose $j_i^{\tau}, j_i^{\beta} \in \frac{1}{2}\Z \setminus \Z$. The elements
\[
\{ \lambda = (1,0), (j_1^{\tau},\tau_1), (j_1^{\beta},\beta_1), \ldots, (j_n^{\tau},\tau_n), (j_n^{\beta},\beta_n) \}
\]
form a basis of $G'(\Zc(n))$ as a free abelian group.
\end{lemma}

\begin{proof}
The set is independent, so it suffices to show these elements generate $G'(\Zc(n))$. Indeed, let $(j, \alpha)$ be an arbitrary element of $G'(\Zc(n))$, where $\alpha = \sum_{i=1}^n (a_i^{\tau} \tau_i + a_i^{\beta} \beta_i)$ for some integers $a_i^{\tau}, a_i^{\beta}$. We have
\[
\prod_{i=1}^n (j_i^{\tau},\tau_i)^{a_i^{\tau}} (j_i^{\beta},\beta_i)^{a_i^{\beta}} = (j',\alpha)
\]
for some $j' \in \frac{1}{2}\Z$, and since $j \equiv \varepsilon(\alpha) \equiv j'$ mod $1$, we have $(j,\alpha) = \lambda^a (j',\alpha)$ for some $a \in \Z$.
\end{proof}

We will use $\Sc$ to choose the half-integers $j_i$ above.
\begin{definition}\label{def:ThetaGradingMap}
If $i \in \Sc$, pick $j_i^\tau = j_i^\beta = 1/2$ in Lemma~\ref{lem:NonCanonicalGradingIso}. If $i \notin \Sc$, pick $j_i^\tau = j_i^\beta = -1/2$. We get an isomorphism $\Theta_{\Sc}$ from $G'(\Zc(n))$ to $\Z \times \Z^{2n}$ by sending:
\begin{itemize}
\item $\lambda = (1,0) \mapsto (1,0)$
\item $(j_i^{\tau},\tau_i) \mapsto (0,\tau_i)$
\item $(j_i^{\beta},\beta_i) \mapsto (0,\beta_i)$.
\end{itemize}
\end{definition}

We now define a grading by $G'(\Zc(n))$ on $\sac nk\Sc$, following \cite[Definition 3.38]{LOT}. Applying $\Theta_{\Sc}$ to this grading, we will get the combinatorially defined Maslov and unrefined Alexander gradings from Section~\ref{sec:CombinatorialStrandsGradings}. We require one further definition.

\begin{definition}\label{def:NumberOfInversions}
Let $(s,\vec{c})$ be a generator of the pre-strands algebra $\tsac nk\Sc$. The \emph{number of inversions} of $(s,\vec{c})$, denoted by $\inv(s,\vec{c})$, is defined as
\[
\inv(s,\vec{c}) = \sum_{i=1}^n \inv_i(s,\vec{c})
\]
where $\inv_i(s,\vec{c})$ is defined as follows:
\begin{itemize}
\item If $s(0) \cap \{z_i^{\pm}\} = \varnothing$, then $\inv_i(s,\vec{c}) := 0$.
\item If $|s(0) \cap \{z_i^{\pm}\}| = 1$, then $\inv_i(s,\vec{c}) := \vec{c}(i)$.
\item If $|s(0) \cap \{z_i^{\pm}\}| = 2$ and the two strands of $s$ on $S^1_i$ have speeds $p_i$ and $q_i$, then
\[
\inv_i(s,\vec{c}) := \frac{|p_i - q_i|}{2} + 2\vec{c}(i).
\]
\end{itemize}
\end{definition}

Note that $\inv(s,\vec{c})$ is only well-defined for elements $(s,\vec{c})$ of $\tsac nk\Sc$, not for elements of $\sac nk\Sc$. Visually speaking, $\inv(s,\vec{c})$ counts the number of crossings between strands in a pictorial representative of $(s,\vec{c})$ where the positions of the closed loops are chosen to avoid triple intersections.  The terms $\vec{c}(i)$ and $2\vec{c}(i)$ then account for crossings between strands of $s$ and closed loops.
\begin{definition}
Let $(s,\vec{c})$ be a basis element of $\tsac nk\Sc$. Let $\tau_1, \beta_1, \ldots, \tau_n, \beta_n$ denote the generators of $H_1(\mc Z(n),B) \cong \Z^{2n}$. The \emph{homology class} of $(s,\vec{c})$, denoted by $[s,\vec{c}] \in H_1(\mc Z(n),B)$, is the sum of the relative homology classes represented by the strands of $s$ viewed as paths in $\mc Z(n)$, together with the additional term
\[
\sum_{i=1}^n \vec{c}(i)(\tau_i + \beta_i)
\]
accounting for closed loops.
\end{definition}

As in \cite[Definition 3.38]{LOT}, we can use the homology classes $[s,\vec{c}]$ and the multiplicity function $m$ to ``correct'' the quantity $\inv(s,\vec{c})$, allowing it to descend to $\sac nk\Sc$.

\begin{lemma}\label{lem:HomologyClassDependsOnlyOnMoving}
Let $(s,\vec{c})$ be a basis element of the pre-strands algebra $\tsac nk\Sc$. We have 
\[
[s,\vec{c}] = [s',\vec{c}]
\]
and
\[
\inv(s,\vec{c}) - m([s,\vec{c}],[s(0)]) = \inv(s',\vec{c}) - m([s',\vec{c}],[s'(0)])
\]
where $s'$ is obtained from $s$ by removing all constant strands.
\end{lemma}

\begin{proof}
The first claim is true because constant strands of $s$ represent zero in $H_1(\mc Z(n),B)$. The second claim is similar to that of \cite[Proposition 3.40]{LOT}, but we give details for completeness.

We can write $\inv(s,\vec{c}) - \inv(s',\vec{c})$ as the sum of $\vec{c}(i)$ over $i \in [1,n]$ with $|s(0) \cap \{z_i^{\pm}\}| = 1$ such that the strand of $s$ on $S^1_i$ is constant, plus the sum of $\frac{|p_i - q_i|}{2} + \xi\vec{c}(i)$ over $i \in [1,n]$ with $|s(0) \cap \{z_i^{\pm}\}| = 2$ such that at least one of the two strands of $s$ on $S^1_i$ is constant, where $\xi$ is the number of constant strands of $s$ on $S^1_i$.

On the other hand, we have $[s,\vec{c}] = [s',\vec{c}]$, and we can write $m([s,\vec{c}], [s(0)]) - m([s,\vec{c}], [s'(0)])$ as the sum of two terms. The first term is 
\[
m\left(\sum_i \vec{c}(i)(\tau_i + \beta_i), [s(0)]\right) - m\left(\sum_i \vec{c}(i) (\tau_i + \beta_i), [s'(0)]\right),
\]
agreeing with the contribution of $\vec{c}$ to $\inv(s,\vec{c}) - \inv(s',\vec{c})$. 

The second term is the sum over $a \in [1,k]$ of $m([s_a], [s(0)]) - m([s_a], [s'(0)])$. For a given index $a$, this difference is the sum of $m([s_a], [s_b(0)])$ over $b$ such that $s_b$ is constant. There can be at most one nonzero term $m([s_a], [s_b(0)])$, and if this term is nonzero then $s_a$ and $s_b$ are different strands on the same backbone $S^1_i$ for some $i$. In such a case, we have $m([s_a], [s_b(0)]) = \frac{|p_i - q_i|}{2}$ (recall that at least one of $p_i,q_i$ vanishes). It follows that $\inv(s,\vec{c}) - \inv(s',\vec{c}) = m([s,\vec{c}],[s(0)]) - m([s',\vec{c}], [s'(0)])$, proving the lemma.
\end{proof}

\begin{definition}
For a basis element $a = E(s,\vec{c})$ of $\sac nk\Sc$, we define 
\[
\deg'(a) = (\inv(s,\vec{c}) - m([s,\vec{c}],[s(0)]), [s,\vec{c}]) \in G'(\Zc(n)).
\]
\end{definition}
Note that $\deg'(a)$ is well-defined by Lemma~\ref{lem:HomologyClassDependsOnlyOnMoving}, since all terms of $E(s,\vec{c})$ give the same element $(s',\vec{c})$ when constant strands are removed. For the condition $j \equiv \varepsilon(\alpha)$ mod $1$, note that $\inv(s,\vec{c})$ is an integer and we can ignore integer contributions to $m([s,\vec{c}],[s(0)])$. For an index $b$ with $s_b(0) = s_b(1)$, there are no half-integer contributions to $m([s,\vec{c}],[s(0)])$ from $s_b$, and there is also no contribution to $\varepsilon(\alpha)$ from $s_b$. For $b$ with $s_b(0) \neq s_b(1)$, we get a contribution of $\frac{1}{2}$ to $m([s,\vec{c}],[s(0)])$ as well as to $\varepsilon(\alpha)$, since $\alpha$ has two parity changes from $s_b$. See also \cite[Proposition 3.39]{LOT}.

\begin{warning}
Lipshitz, Ozsv{\'a}th, and Thurston refer to $\inv(s,\vec{c}) - m([s,\vec{c}],[s(0)])$ as the ``Maslov component'' and $[s,\vec{c}]$ as the ``$\Spinc$ component'' of $\deg'(a)$. However, this Maslov component (a half-integer in general) is different from the Maslov grading by $\Z$ that we will extract from $\deg'$.
\end{warning}

The quantity $\inv(s,\vec{c}) - m([s,\vec{c}],[s(0)])$ is independent of $\vec{c}$, as we prove below.
\begin{lemma}\label{lem:MaslovCptIndepOfClosedLoops}
For a basis element $(s,\vec{c})$ of $\tsac nk\Sc$, we have
\[
\inv(s,\vec{c}) - m([s,\vec{c}],[s(0)]) = \inv(s,\vec{0}) - m([s,\vec{0}],[s(0)]).
\]
\end{lemma}

\begin{proof}
For $1 \leq i \leq n$, the contributions $0$, $\vec{c}(i)$, or $2\vec{c}(i)$ to $\inv(s,\vec{c})$ in Definition~\ref{def:NumberOfInversions} are cancelled by the contribution
\[
\vec{c}(i) m(\tau_i + \beta_i, [s(0) \cap \{z_i^{\pm}\}])
\]
to $m([s,\vec{c}],[s(0)])$.
\end{proof}

\begin{lemma}\label{lem:PiQiHomologyClass}
Let $a=E(s,\vec{c})\in\Jb_\x\sac nk\Sc$ be a basis element; write $a$ as
\[
a=C_1^{\vec{c}(1)} \cdots C_n^{\vec{c}(n)} \vv{p_1}{q_1}{1}\cdots\vv{p_n}{q_n}{n}.
\]
We have
\[
[s,\vec{c}] = \sum_{i=1}^n (a^{\tau}_i \tau_i + a^{\beta}_i \beta_i),
\]
where
\[
a^{\tau}_i = \vec{c}(i) + \floor*{\frac{p_i}{2}} + \ceil*{\frac{q_i}{2}}
\]
and
\[
a^{\beta}_i = \vec{c}(i) + \ceil*{\frac{p_i}{2}} + \floor*{\frac{q_i}{2}}.
\]
\end{lemma}

\begin{proof}
The terms $\vec{c}(i)$ are present by the definition of $[s,\vec{c}]$. For the other terms, note that the strands of $s$ representing nonzero homology classes correspond to nonzero entries $p_i$ or $q_i$ in the representation of $a$. An entry $p_i > 0$ represents a homology class of a path traversing $S^1_i$ for $p_i$ half-turns, with the segment $\beta_i$ being traversed one more time than $\tau_i$ if $p_i$ is odd (where we identify $\tau_i$ and $\beta_i$ with the oriented segments on $\mathcal{Z}$ in their respective relative homology classes). An entry $q_i > 0$ is similar, except that $\tau_i$ is traversed one more time than $\beta_i$ if $q_i$ is odd. Counting up how many times the segments $\tau_i$ and $\beta_i$ are traversed by all strands of $s$, we get the formulas of the lemma.
\end{proof}

\begin{lemma}\label{lem:MaslovComponentComputation}
Let $a=E(s,\vec{c})\in\Jb_\x\sac nk\Sc$ be a basis element; write $a$ as
\[
a=C_1^{\vec{c}(1)} \cdots C_n^{\vec{c}(n)} \vv{p_1}{q_1}{1}\cdots\vv{p_n}{q_n}{n}.
\]
We have
\[
\inv(s,\vec{c}) - m([s,\vec{c}],[s(0)]) = \sum_{i=1}^n \left( \frac{|p_i - q_i|}{2} - (p_i + q_i) \right).
\]
\end{lemma}

\begin{proof}
By Lemma~\ref{lem:HomologyClassDependsOnlyOnMoving}, we may assume that $s$ has no horizontal strands, and by Lemma~\ref{lem:MaslovCptIndepOfClosedLoops}, we may assume that $\vec{c} = \vec{0}$. For $1 \leq i \leq n$, we consider three cases:
\begin{itemize}
\item If $s(0) \cap \{z_i^{\pm}\} = \varnothing$, then $\inv_i(s,\vec{0}) = 0$ and $[s,\vec{0}]$ has coefficient zero on $\tau_i$ and $\beta_i$, so $m([s,\vec{0}],[s(0)])$ has no contribution from the cylinder $\I \times S^1_i$. We also have $p_i = q_i = 0$.
\item If $|s(0) \cap \{z_i^{\pm}\}| = 1$, then $\inv_i(s,\vec{0}) = 0$. Assume first that $p_i > 0$ and $q_i = 0$. We have $m(\tau_i,[s(0)]) = m(\beta_i,[s(0)]) = \frac{1}{2}$. Thus, by Lemma~\ref{lem:PiQiHomologyClass}, we have
\begin{align*}
m([s,\vec{0}],[s(0) \cap \{z_i^{\pm}\}]) &= \frac{\ceil*{\frac{p_i}{2}} + \floor*{\frac{p_i}{2}}}{2} \\
&= \frac{p_i}{2}.
\end{align*}
It follows that 
\begin{align*}
\inv_i(s,\vec{c}) - m([s,\vec{c}],[s(0) \cap \{z_i^{\pm}\}]) &= -\frac{p_i}{2} \\
&= \frac{|p_i - 0|}{2} - (p_i + 0).
\end{align*}
The case where $p_i = 0$ and $q_i > 0$ is similar.
\item If $\{z_i^{\pm}\} \subset s(0)$, then $\inv_i(s,\vec{0}) = \frac{|p_i - q_i|}{2}$. We have $m(\tau_i,[s(0)]) = m(\beta_i,[s(0)]) = 1$, so by Lemma~\ref{lem:PiQiHomologyClass}, we have
\begin{align*}
m([s,\vec{0}],[s(0)]) &= \floor*{\frac{p_i}{2}} + \ceil*{\frac{q_i}{2}} + \ceil*{\frac{p_i}{2}} + \floor*{\frac{q_i}{2}} \\
&= p_i + q_i.
\end{align*}
It follows that
\[
\inv_i(s,\vec{c}) - m([s,\vec{c}],[s(0) \cap \{z_i^{\pm}\}]) = \frac{|p_i - q_i|}{2} - (p_i + q_i).
\]
\end{itemize}
The lemma follows from summing over $i$.
\end{proof}

\begin{proposition}\label{prop:UnrefinedGradingsCorrespond}
For a basis element $a = E(s,\vec{c})$ of $\sac nk\Sc$, we have $\Theta_{\Sc}(\deg'(a)) = (\m(a), w^{\un}(a))$, where $\Theta_{\Sc}$ is defined in Definition \ref{def:ThetaGradingMap}. 
\end{proposition}

\begin{proof}
Let $a^{\tau}_i$, $a^{\beta}_i$ be defined as in Lemma~\ref{lem:PiQiHomologyClass}. By definition, $\Theta_{\Sc}$ sends the element
\[
\sum_{i=1}^n \left((-1)^{\chi_{\Sc}(i)+1} \frac{a^{\tau}_i + a^{\beta}_i}{2}, a^{\tau}_i \tau_i + a^{\beta}_i \beta_i\right)
\]
to an element of $\Z \times \Z^{2n}$ with first component zero. Thus, the first component of $\Theta_{\Sc}(\deg'(a))$ is
\[
\inv(s,\vec{c}) - m([s,\vec{c}],[s(0)]) + \sum_{i=1}^n (-1)^{\chi_{\Sc}(i)} \frac{a^{\tau}_i + a^{\beta}_i}{2}.
\]
By Lemmas~\ref{lem:MaslovComponentComputation} and \ref{lem:PiQiHomologyClass}, this quantity equals
\[
\sum_{i=1}^n \left( \frac{|p_i - q_i|}{2} - (p_i + q_i) + (-1)^{\chi_{\Sc}(i)} \left(\vec{c}(i) + \frac{p_i + q_i}{2} \right)\right),
\]
which is $\m(a)$ by Definition~\ref{def:gradings}.

For the rest of the components, we have $\pi_{\Z^{2n}}(\Theta_{\Sc}(j,\alpha)) = \alpha$, where we are identifying $H_1(\mc Z(n),B)$ with $\Z^{2n}$ as usual. Thus, Lemma~\ref{lem:PiQiHomologyClass} implies that $\pi_{\Z^{2n}}(\Theta_{\Sc}(\deg'(a))) = w^{\un}(a)$.
\end{proof}

\subsection{Where the refined gradings come from, topologically}\label{sec:RefinedGradings}

Let $\Zc = (\mc Z,B,M)$ be a chord diagram and let $q: B \to B/M$ be the quotient map defined in Section~\ref{sec:IdemsAndUnit}.
\begin{definition}[Section 3.3.2 of \cite{LOT}]
The \emph{refined grading group} $G(\mc Z)$ of $\Zc$ is the subgroup of $G'(\mc Z)$ consisting of elements $(j,\alpha)$ with $q_* \circ \de(\alpha) = 0$, where
\[
q_* \circ \de \colon H_1(\mc Z, B) \to H_0(B/M)
\]
is the composition of $q_*$ with $\de: H_1(\mc Z,B) \to H_0(B)$.
\end{definition}

Recall from Definition \ref{def:sutured surface for chord diagram} that a chord diagram $\mc Z=(\mc Z,B,M)$ determines a sutured surface $F(\mc Z)$ that is built by attaching 1-handles to $\mc Z\times[0,1]$ according to the matching $M$. Lipshitz, Ozsv\'ath, and Thurston~\cite[Section 3.3.2]{LOT} show how to identify the kernel of $q_* \circ \de$ with the homology group $H_1(F(\mc Z))$. Correspondingly, they identify $G(\mc Z)$ (non-canonically) with a central extension of $H_1(F(\mc Z))$ by $\Z$, where $gh = hg \lambda^{2[g] \cap [h]}$ for $g,h \in G(\mc Z)$. Here $[g]$ denotes the image of $g$ in $H_1(F(\mc Z))$.
\begin{remark}
One can also describe $G(\mc Z)$ in terms of nonvanishing vector fields as in Seiberg--Witten theory; see \cite[Remark 3.48]{LOT}.
\end{remark}

As discussed in Section~\ref{sec:ExampleofInterest}, the surface $F(\mc Z(n))$ is an $n$-punctured disc; the circular backbones of $\mc Z(n)$ provide a basis for $H_1(F(\mc Z(n)))$. The intersection pairing on $H_1(F(\mc Z(n)))$ is trivial, so $G(\mc Z(n))$ is abelian (in fact, $G(\mc Z(n))$ is a subgroup of $G'(\mc Z(n))$ and $G'(\mc Z(n))$ is already abelian).

In \cite[Remark 3.47]{LOT}, Lipshitz--Ozsv{\'a}th--Thurston mention that in some cases one can obtain a grading by $G(\Zc)$ from a grading by $G'(\Zc)$ by applying a homomorphism from $G'(\Zc)$ to $G(\Zc)$ fixing $G(\Zc)$ as a subgroup of $G'(\Zc)$ (extension of scalars is usually required to define such a homomorphism, and even then it does not always exist). 

In our case, the homomorphism exists and the extension of scalars is unproblematic, so we do not need to make choices for each idempotent as in \cite[Section 3.3.2]{LOT}. Indeed, the isomorphism $\Theta_{\Sc}$ from Definition~\ref{def:ThetaGradingMap} sends $G(\Zc(n))$ to the subgroup of $\Z \times \Z^{2n}$ generated by $(1,0)$ and $(0,\tau_i + \beta_i)$ for all $i$, regardless of $\Sc$. This subgroup is isomorphic to $\Z \times \Z^{n}$ where we identify $(0,\tau_i + \beta_i)$ with $(0,e_i)$. We can thus extend scalars by replacing $G(\Zc(n))$ with $\Z \times (\frac{1}{2}\Z)^n$. We have a homomorphism $\Psi$ from $\Z \times \Z^{2n}$ to $\Z \times (\frac{1}{2}\Z)^n$ sending:
\begin{itemize}
\item $(1,0) \mapsto (1,0)$,
\item $(0,\tau_i) \mapsto (0,\frac{e_i}{2})$, and
\item $(0,\beta_i) \mapsto (0,\frac{e_i}{2})$.
\end{itemize}
Conjugating by the isomorphisms $\Theta_{\Sc}$, we get a homomorphism $\Psi_{\Sc}: G'(\Zc(n)) \to G(\Zc(n))$ such that the diagram
\[
\xymatrix{
G'(\Zc(n)) \ar[r]^{\Theta_{\Sc}}_{\cong} \ar[d]_{\Psi_{\Sc}} & \Z \times \Z^{2n} \ar[d]^{\Psi} \\ G(\Zc(n)) \ar[r]_{\Theta_{\Sc}}^{\cong} & \Z \times (\frac{1}{2}\Z)^n
}
\]
commutes.
The generators $(1,0)$ and $(0,\tau_i + \beta_i)$ of $\Theta_{\Sc}(G(\Zc(n)))\subset\Theta_\Sc(G'(\Zc(n)))$ are sent to themselves by $\Psi$, so $\Psi_{\Sc}$ fixes the original (unextended) $G(\Zc(n))$ as a subgroup of $G'(\Zc(n))$.
Note that $\Psi_{\Sc}$ is independent of $\Sc$.
\begin{definition}
For a basis element $a$ of $\sac nk\Sc$, define $\deg(a) := \Psi_{\Sc}(\deg'(a))$, an element of the (extended) grading group $G(\Zc(n)) \cong \Z \times (\frac{1}{2}\Z)^n$. Since $\Psi$ (and thus $\Psi_{\Sc}$) is a group homomorphism preserving $\lambda = (1,0)$, this grading is well-defined.
\end{definition}

\begin{corollary}
For a basis element $a = E(s,\vec{c})$ of $\sac nk\Sc$, we have $\Theta_{\Sc}(\deg(a)) = (\m(a), w(a))$.
\end{corollary}

\begin{proof}
By definition, $\Theta_{\Sc}(\deg(a)) = \Psi(\Theta_{\Sc}(\deg'(a)))$, which equals $\Psi(\m(a),w^{\un}(a))$ by Proposition~\ref{prop:UnrefinedGradingsCorrespond}. Since $\Psi$ sends both $\tau_i$ and $\beta_i$ to $\frac{e_i}{2}$, we have $\Psi(\m(a),w^{\un}(a)) = (\m(a), w(a))$.
\end{proof}

\section{Symmetries}\label{sec:StrandSymmetries}
Now we will define analogues of the symmetries $\rho$ and $o$ from \cite[\secOSzSymmetries]{MMW1} for the strands algebras $\sac nk\Sc$ (see Definition~\ref{def:ReviewSymmetries} for a brief review, as well as \cite[Section 3.6]{OSzNew} where these symmetries were first introduced). We use the notation of \cite[\secOSzSymmetries]{MMW1} and Definition~\ref{def:ReviewSymmetries}.

\begin{proposition}\label{prop:ConstructingRhoOnStrandsAlg}
For a generator $a$ of $\Jb_{\x} \sac nk\Sc \Jb_{\y}$, write
\[
a = C_{i_1} \cdots C_{i_l} \vv{p_1}{q_1}{1} \cdots \vv{p_n}{q_n}{n}
\]
as in Lemma~\ref{lem:piqi}. Define an array of vectors $\vv{p'_i}{q'_i}{i}$ by $p'_i = q_{n+1-i}$ and $q'_i = p_{n+1-i}$. The expression 
\[
C_{n+1-i_1} \cdots C_{n+1-i_l} \vv{p'_1}{q'_1}{1} \cdots \vv{p'_n}{q'_n}{n}
\]
represents a valid generator of $\Jb_{\rho(\x)} \sac nk{\rho(\Sc)} \Jb_{\rho(\y)}$.
\end{proposition}

\begin{proof}
The conditions of Lemma~\ref{lem:piqi} are invariant under replacing $C_i$ with $C_{n+1-i}$, $p_i$ with $q_{n+1-i}$, $q_i$ with $p_{n+1-i}$, $\x$ with $\rho(\x)$, $\y$ with $\rho(\y)$, and $\Sc$ with $\rho(\Sc)$. For example, to see that the new condition~\eqref{it:pq start from x} is satisfied, note that if $i \notin \rho(\x)$, then $n-i \notin \x$, so $q_{n-i} = p_{n+1-i} = 0$ by the old condition~\eqref{it:pq start from x}. The rest of the conditions are similar. Also note that $\rho(\y)$ is the element of $V(n,k)$ that Lemma~\ref{lem:piqi} constructs given $\rho(\x)$ and the array of vectors $\vv{p'_1}{q'_1}{1} \cdots \vv{p'_n}{q'_n}{n}$.
\end{proof}

\begin{definition}
For a generator $a$ of $\Jb_{\x} \sac nk\Sc \Jb_{\y}$, define 
\[
\rho(a) \in \Jb_{\rho(\x)} \sac nk{\rho(\Sc)} \Jb_{\rho(\y)}
\]
to be the generator constructed in Proposition~\ref{prop:ConstructingRhoOnStrandsAlg}.
\end{definition}

We thus have an $\Ib(n,k)$-linear map $\rho: \sac nk\Sc \to \sac nk{\rho(\Sc)}$, where the action of $\Ib(n,k)$ on $\sac nk{\rho(\Sc)}$ is modified so that $\Jb_{\x}$ acts via the usual action by $\Jb_{\rho(\x)}$. We claim that $\rho$ is an involution of dg algebras over $\Ib(n,k)$, after suitable modifications to the gradings.

\begin{proposition}\label{prop:RhoMultiplicativeStrands}
As in \cite[\secOSzSymmetries]{MMW1}, modify the unrefined Alexander multi-grading on $\sac nk{\rho(\Sc)}$ by postcomposing the degree function with the involution of $\Z^{2n}$ sending $\tau_i$ to $\beta_{n+1-i}$ and sending $\beta_i$ to $\tau_{n+1-i}$. Then the map $\rho$ is a homomorphism of dg algebras from $\sac nk\Sc$ to $\sac nk{\rho(\Sc)}$ and satisfies $\rho^2 = \id$. Similar grading statements hold for the refined and single Alexander gradings.
\end{proposition}

\begin{proof}
The equation $\rho^2 = \id$ is immediate from the definition of $\rho$. To see that $\rho$ respects multiplication, let $a \in \Jb_{\x} \sac nk\Sc \Jb_{\y}$ and $a' \in \Jb_{\y} \sac nk\Sc \Jb_{\z}$. The product $a \cdot a'$ is given by Lemma~\ref{lem:concatenable}. If it is zero, then one can check that $\rho(a) \cdot \rho(a')$ is also zero. Otherwise, $\rho(a) \cdot \rho(a')$ has the same idempotents, $C_i$ variables, and array of vectors as $\rho(a \cdot a')$, so $\rho(a) \cdot \rho(a') = \rho(a \cdot a')$.  Similarly, Lemma~\ref{lem:general pq differential} implies that $\rho$ respects the differential. One can check that $\rho$ respects the gradings of Definition~\ref{def:gradings} after the above modification. 
\end{proof}

Next we define a symmetry $o$ on our strands algebras.

\begin{proposition}\label{prop:ConstructingOOnStrandsAlg}
For a generator $a$ of $\Jb_{\x} \sac nk\Sc \Jb_{\y}$, write
\[
a = C_{i_1} \cdots C_{i_l} \vv{p_1}{q_1}{1} \cdots \vv{p_n}{q_n}{n}
\]
as in Lemma~\ref{lem:piqi}. The expression
\[
C_{i_1} \cdots C_{i_l} \vv{q_1}{p_1}{1} \cdots \vv{q_n}{p_n}{n}
\]
represents a valid generator of $\Jb_{\y} \sac nk\Sc \Jb_{\x}$.
\end{proposition}

\begin{proof}
As with Proposition~\ref{prop:ConstructingRhoOnStrandsAlg}, one can check that the conditions of Lemma~\ref{lem:piqi} for the old expression imply the conditions for the new expression, and that $\x \in V(n,k)$ is the vertex selected by Lemma~\ref{lem:piqi} given $\y$ and the new array of vectors.
\end{proof}

\begin{definition}
For a generator $a$ of $\Jb_{\x} \sac nk\Sc \Jb_{\y}$, define 
\[
o(a) \in \Jb_{\y} \sac nk\Sc \Jb_{\x}
\]
to be the generator constructed in Proposition~\ref{prop:ConstructingOOnStrandsAlg}.
\end{definition}

We thus have a $\Ib(n,k)$-linear map $o: \sac nk\Sc \to \sac nk\Sc$, where the $\Ib(n,k)$-algebra structure on the target side is unmodified (unlike for $\rho$). We claim that $o$ respects multiplication, differential, and gradings when we take the opposite algebra on the target side.

\begin{figure}
\includegraphics[scale=0.5]{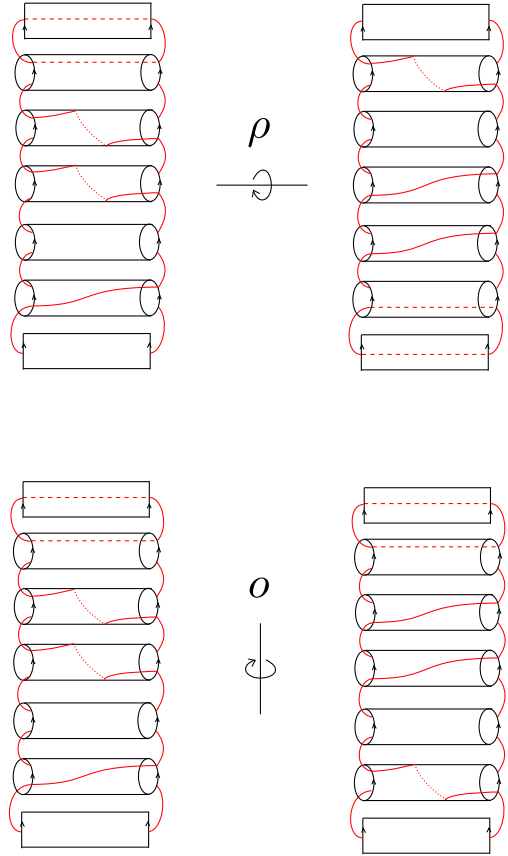}
\caption{Graphical interpretation of the symmetries $\rho$ and $o$ on the strands algebra $\sac nk\Sc$.}
\label{fig:StrandsSymmetries}
\end{figure}

\begin{proposition}
Modify the unrefined Alexander multi-grading on $(\sac nk{\Sc})^{\op}$ by postcomposing the degree function with the involution of $\Z^{2n}$ sending $\tau_i$ to $\beta_i$ and sending $\beta_i$ to $\tau_i$. The map $o$ is a homomorphism of dg algebras from $\sac nk\Sc$ to $(\sac nk\Sc)^{\op}$ and satisfies $o^2 = \id$. Similar statements hold for the refined and single Alexander gradings.
\end{proposition}

\begin{proof}
Like in Proposition~\ref{prop:RhoMultiplicativeStrands}, the proof amounts to checking that Lemma~\ref{lem:concatenable}, Lemma~\ref{lem:general pq differential}, and Definition~\ref{def:gradings} are compatible with the symmetry $o$. The detailed checks will be omitted.
\end{proof}

Note that $\rho \circ o = o \circ \rho$, properly interpreted.

\begin{remark}
The symmetries $\rho$ and $o$ on the strands algebras may be understood visually as follows, in terms of the graphical interpretation of Section~\ref{sec:StrandsAlgDef}: $\rho$ is rotation by 180 degrees around a horizontal line, and $o$ is rotation by 180 degrees around a vertical line (both lines are in the plane of the page as drawn). See Figure~\ref{fig:StrandsSymmetries} for an illustration. Note that the group of orientation-preserving self-diffeomorphisms of $\I \times \Zc(n)$ preserving the matching data, modulo isotopies among such diffeomorphisms, is $\Z/2\Z \times \Z/2\Z$. The rotations corresponding to $\rho$ and $o$ may be taken as generators. Thus, all geometric symmetries of $\I \times \Zc(n)$ are reflected in the algebras $\sac nk\Sc$. We will relate these symmetries with the symmetries $\rho$ and $o$ on $\B(n,k,\Sc)$ in Section~\ref{sec:PhiSymmetries}.
\end{remark}
\section{Homology of the strands algebra}
\label{sec:Homology}

The goal of this section is to compute the homology of $\A(n,k) = \A(n,k,\varnothing)$. The homology of $\A(n,k,\Sc)$ for general $\Sc$ will then follow from Theorem \ref{thm:ReviewOSzHomology} and Theorem~\ref{thm:FinalQIThm}.

By \cite[\lemOrthogonalIdempotents]{MMW1}, the homology $H_*(\A(n,k))$ is still an $\Ib(n,k)$-algebra, and it can be decomposed as
\[
H_*(\A(n,k)) = \bigoplus_{\x, \y \in V(n,k)} \Jb_\x H_*(\A(n,k)) \Jb_\y,
\]
where $\Jb_\x H_*(\A(n,k)) \Jb_\y = H_*(\Jb_\x \A(n,k) \Jb_\y)$. Thus, it suffices to compute the homology of each summand $\Jb_\x \sa nk \Jb_\y$.
Since $\Jb_\x \sa nk \Jb_\y = 0$ if $\x$ and $\y$ are far (see Lemma \ref{lem:FarStatesStrandAlgZero}), we can focus on the case when $\x$ and $\y$ are not far.

Before computing $\Jb_\x H_*(\A(n,k)) \Jb_\y$, we introduce some notation that will be useful later. Recall from Definition \ref{def:differential} that the differential $\de = \de^0$ on $\sa nk$ is a sum over differentials $\de_i$ on each circular backbone $S^1_i$.  We can augment this notation as follows.

\begin{definition}\label{def:subset differential}
Given a subset $S \subset [1,n]$, we define a new differential $\de_S$ on $\sa nk$ by 
\[
\de_S := \sum_{i\in S} \de_i.
\]
\end{definition}

By a simple generalization of the arguments in Section \ref{sec:Differential}, $\de_S$ gives a well-defined differential on $\sa nk$.  The following lemma follows from a comparison of the sets $S,T,S\cup T,$ and $S\cap T$.

\begin{lemma}
For all subsets $S, T \subset [1,n]$, we have $\de_S + \de_T = \de_{S \cup T} + \de_{S \cap T}$.
\end{lemma}

\begin{corollary}
\label{cor:splitdifferential}
If $S_1 \sqcup \dots \sqcup S_a = [1,n]$ is a partition of $[1,n]$, then
\[
\de_{S_1} + \dots + \de_{S_a} = \de = \de_{[1,n]}.
\]
\end{corollary}

\subsection{A splitting theorem}
\label{sec:splittingthm}

In this section we present an important theorem on the structure of any summand $\Jb_\x \sa nk \Jb_\y$ of our strands algebra, whose proof will occupy Sections \ref{sec:JAJtoTensorProduct}, \ref{sec:TensorProducttoJAJ}, and \ref{sec:proof of splitting JAJ}.  The idea is as follows.  The differential $\de$ acts on each backbone $S^1_i$ independently, and $S^1_i$ only admits certain types of strands depending on whether $S^1_i$ corresponds to a crossed line or a member of a generating or edge interval from $\x$ to $\y$.  Thus we expect a tensor product decomposition for $\Jb_\x \sa nk \Jb_\y$ based on the generating interval data, similar to \cite[\corIBItoTensorProduct]{MMW1} for $\Ib_{\x} \B(n,k,\Sc) \Ib_{\y}$.

As in \cite[\secOSzSplittingTheorem]{MMW1}, let $\x, \y \in V(n,k)$ be not far. Based on the structure of the generating intervals and edge intervals for $\x$ and $\y$, we introduce a regrading of the generating and edge algebras $\overline{A}(l)$ and $\overline{A}_{\rho}(l)$ from Definitions~\ref{def:GenAlgebra} and \ref{def:EdgeAlgebra}.  Let $\CL{\x,\y}$ be the set of crossed lines from $\x$ to $\y$ (see Definition \ref{def:CL}), and let $[j_1+1, j_1+l_1], \ldots, [j_b+1, j_b+l_b]$ be the generating intervals for $\x$ and $\y$ (see Definition \ref{def:ReviewGenInts}), of lengths $l_1, \ldots, l_b$ respectively, ordered so that $j_1< \cdots < j_b$.

\begin{definition}\label{def:RegradedGenAlgs}

For a generating interval $G=[j_a+1, j_a + l_a]$ between $\x$ and $\y$, we have a canonical isomorphism of differential algebras
\begin{equation}\label{eq:psiG for non edge}
\psi_G \colon \Jb_{[j_a+1, j_a +l_a-1]} \sa{n}{l_a-1} \Jb_{[j_a+1, j_a +l_a-1]} \to \overline{A}(l_a),
\end{equation}
by a simple re-indexing of the circular backbones, omitting the empty ones. Redefine the Alexander multi-gradings on $\overline{A}(l_a)$ by shifting the indices by $j_a$, so that $\tau_{i}, \beta_{i} \mapsto \tau_{i+j_a}, \beta_{i+j_a}$ and the isomorphism preserves the Maslov grading and all Alexander gradings from Definition~\ref{def:gradings}. 

Similarly, if $G = [n-l_{b+1}+1,n]]$ is a right edge interval for $\x$ and $\y$, there is a canonical isomorphism
\begin{equation}\label{eq:psiG for right edge}
\psi_G \colon \Jb_{[n-l_{b+1}+1, n]} \sa{n}{l_{b+1}} \Jb_{[n-l_{b+1}+1,n]} \to \overline{A}_\rho(l_{b+1}).
\end{equation}
Modify the Alexander gradings on $\overline{A}_\rho(l_{b+1})$ so that $\psi_G$ preserves them as above.

If $G = [[1,l_0]$ is a left edge interval for $\x$ and $\y$, then there is a canonical isomorphism
\begin{equation}\label{eq:psiG for left edge}
\psi_G \colon \Jb_{[0, l_0-1]} \sa{n}{l_0} \Jb_{[0,l_0-1]} \to \overline{A}_\lda(l_0).
\end{equation}
Note that there is no need to redefine the Alexander multi-grading on $\overline{A}_\lda (l_0)$ in this case, because $\psi_G$ already preserves it. If $G = [[1,n]]$ is a two-faced edge interval for $\x$ and $\y$, then $\overline{A}_{\lda\rho}(n) = \Jb_{[0, n]} \sa{n}{n+1} \Jb_{[0,n]}$ by definition.

We will also use the graded polynomial algebra $\F_2[U_i \,|\, i\in\CL{\x,\y}]$ for crossed lines as in the $\Sc=\varnothing$ case of \cite[\defCrossedLinesAlg]{MMW1}.  As described there, $\F_2[U_i \,|\, i\in\CL{\x,\y}]$ has zero differential, and it carries an Alexander multi-grading defined by
\[
w_i(1) = \begin{cases}
\frac12 & \text{if $i\in\CL{\x,\y}$}\\
0 & \text{otherwise}\\
\end{cases}
\]
with multiplication by $U_j$ increasing $w_i$ by $\delta_{i,j}$.  Because we have $\Sc=\varnothing$ here, all of $\F_2[U_i \,|\, i\in\CL{\x,\y}]$ is placed in Maslov degree zero.
\end{definition}

\begin{theorem}
\label{thm:JAJtoTensorProduct}
Let $\x, \y \in V(n,k)$ be not far. With notation as above, there is an isomorphism of chain complexes over $\F_2$
\[
\psi \colon \Jb_\x \A(n,k) \Jb_\y \stackrel{\sim}{\longrightarrow} \F_2[U_i \,|\, i\in\CL{\x,\y}] \otimes \overline{A}_\circ(l_0) \otimes \overline{A}(l_1) \otimes \cdots \otimes \overline{A}(l_{b}) \otimes \overline{A}_\circ(l_{b+1}),
\]
which respects both the Alexander and the Maslov gradings, where the algebras $\overline{A}_\circ(l_0)$ and $\overline{A}_\circ(l_{b+1})$ are defined as follows. 
\begin{itemize}
\item If there is a left edge interval $[[1,l_0]$, then we set $\overline{A}_\circ(l_0) = \overline{A}_\lda(l_0)$; otherwise we set $\overline{A}_\circ(l_0) = \F_2$.
\item If there is a right edge interval $[n-l_{b+1}+1,n]]$, then we set $\overline{A}_\circ(l_{b+1}) = \overline{A}_\rho(l_{b+1})$; otherwise we set $\overline{A}_\circ(l_{b+1}) = \F_2$.
\item If $\x = \y = [0,n]$ (i.e. $[[1,n]]$ is a two-faced edge interval for $\x$ and $\y$), then we set the target of $\psi$ to be $\overline{A}_{\lambda\rho}(n)$.
\end{itemize}
The Alexander and Maslov gradings on the right hand side are specified in Definition~\ref{def:RegradedGenAlgs}.
\end{theorem}

Proving Theorem~\ref{thm:JAJtoTensorProduct} is the goal of Sections \ref{sec:JAJtoTensorProduct}, \ref{sec:TensorProducttoJAJ}, and \ref{sec:proof of splitting JAJ}.

\begin{remark}\label{rmk:splitting A wun}
Just as with \cite[\rmkSplittingBwun]{MMW1} concerning the splitting of $\Ib_\x\B(n,k,\Sc)\Ib_\y$, we could also assign unrefined Alexander gradings to the tensor factors in Theorem~\ref{thm:JAJtoTensorProduct} in such a way that $\psi$ respects these gradings as well.
\end{remark}

\subsection{Definition of \texorpdfstring{$\psi$}{psi}}
\label{sec:JAJtoTensorProduct}
Fix $\x,\y\in V(n,k)$.  If $\x = \y = [0, n]$, then the map $\psi$ of Theorem \ref{thm:JAJtoTensorProduct} is simply the identity map of $\overline{A}_{\lda \rho} (n)$.  Otherwise, there is not a two-faced edge interval and we will build the map $\psi$ by focusing on one tensor factor in the image at a time.

Let $a\in\Jb_\x \sa nk \Jb_\y$ be a standard basis element
\[
a = \vv{p_1}{q_1}{1} \cdots \vv{p_n}{q_n}{n}.
\]
If $i\in\CL{\x,\y}$, then by Lemma \ref{lem:cl} there must be a single strand on the $i$-th backbone, so one of $p_i, q_i$ is zero and the other one is odd (which of $p_i, q_i$ is zero is determined by the sign of $v_i(\x, \y)$).  To each crossed line $i$, then, we can associate the monomial $U_i^{\frac{p_i + q_i -1}{2}} \in \F_2[U_i]$.  Note that the corresponding strand winds on the $i$-th backbone by $p_i + q_i$ half twists.

\begin{lemma}
\label{lem:x|_G}
Given a standard basis element
\[
a = \vv{p_1}{q_1}{1} \ldots \vv{p_n}{q_n}{n} \in \Jb_\x \sa nk \Jb_\y
\]
and a generating interval $G = [j+1, j+l]$ from $\x$ to $\y$, the restriction
\[
a|_G := \vv{p_{j+1}}{q_{j+1}}{j+1} \cdots \vv{p_{j+l}}{q_{j+l}}{j+l}
\]
of $a$ to $G$ is a well defined basis element of $\Jb_{[j+1, j+l-1]} \sa {n}{l-1} \Jb_{[j+1, j+l-1]}$. 

Moreover, the differential $\de_G$ (see Definition \ref{def:subset differential}) satisfies $(\de_G a)|_G = \de(a|_G)$.
\end{lemma}

Geometrically, $a|_G$ is obtained by restricting the support of $a$ to the circular backbones labelled $j+1, \ldots, j+l$.
\begin{proof}
By the discussion in Section~\ref{sec:splittingthm}, we may equivalently view $a|_G$ as an element of $\overline{A}(l)$; we must show that the conditions of Lemma~\ref{lem:BasistdA} are satisfied. Condition~\eqref{it:A1} for $a|_G$ follows from condition~\eqref{it:pq no matched start} of Lemma~\ref{lem:piqi} for $a$, while condition \eqref{it:A3} follows from Lemma \ref{lem:cl} since no line in a generating interval is crossed (see Proposition \ref{prop:QuartumNonDatur}).

For condition \eqref{it:A2} of Lemma~\ref{lem:BasistdA}, we show that $p_{j+1}=0$; the proof that $q_{j+l}=0$ is similar.
First suppose that $j \notin \x$. Condition~\eqref{it:pq start from x} of Lemma \ref{lem:piqi} then implies that $p_{j+1}=0$. Now suppose that $j \in \x$. By the definition of generating interval, the coordinate $j$ is not fully used, so $j \notin \y$.  Since line $j+1$ is not crossed, we must have line $j$ crossed with $v_j(\x,\y)=-1$, and Lemma \ref{lem:cl} then ensures that $p_j=0$ and $q_j$ is odd.  From here, condition \eqref{it:pq no matched start} of Lemma \ref{lem:piqi} implies that $p_{j+1}=0$.

The fact that the restriction commutes with the differential amounts to ensuring that $\de_{j+1} a = \de_{j+l} a = 0$, since equation \eqref{eq:general pq differential bad idemp} of Lemma \ref{lem:general pq differential} implies that $\de_{j+1} a|_G = \de_{j+l} a|_G = 0$ (all other summands of the differential commute trivially).  For $\de_{j+1} a$, if $j\notin\x$ then $\de_{j+1} a =0$ by equation \eqref{eq:general pq differential bad idemp}, while if $j\in\x$ then $q_j$ is odd, so $\de_{j+1} a=0$ by the final case of equation \eqref{eq:general pq differential good idemp}.  The analysis for $\de_{j+l} a$ is similar.
\end{proof}

We have analogous statements when $G$ is an edge interval, with similar proofs.

\begin{lemma}
\label{lem:x|_Glambda}
Given a standard basis element
\[
a = \vv{p_1}{q_1}{1} \ldots \vv{p_n}{q_n}{n} \in \Jb_\x \sa nk \Jb_\y
\]
and a left edge interval $G = [[1, l]$ from $\x$ to $\y$, the restriction
\[
a|_G = \vv{p_{1}}{q_{1}}{1} \cdots \vv{p_l}{q_l}{l}
\]
of $a$ to $G$ is a well defined basis element of $\Jb_{[0, l-1]} \sa{n}{l} \Jb_{[0, l-1]}$.

Moreover, the differential $\de_G$ satisfies $(\de_G a)|_G = \de(a|_G)$.
\end{lemma}

\begin{lemma}
\label{lem:x|_Grho}
Given a standard basis element
\[
a = \vv{p_1}{q_1}{1} \ldots \vv{p_n}{q_n}{n} \in \Jb_\x \sa nk \Jb_\y
\]
and a right edge interval $G = [n-l+1, n]]$ from $\x$ to $\y$, the restriction
\[
a|_G = \vv{p_{n-l+1}}{q_{n-l+1}}{n-l+1} \cdots \vv{p_n}{q_n}n
\]
of $a$ to $G$ is a well defined basis element of $\Jb_{[n-l+1, n]} \sa{n}{l} \Jb_{[n-l+1, n]}$.

Moreover, the differential $\de_G$ satisfies $(\de_G a)|_G = \de(a|_G)$.
\end{lemma}

Suppose that $\x$ and $\y$ are not far. Let $G_1, \ldots, G_b$ denote the generating intervals, and $G_\lda$ and $G_\rho$ denote the edge intervals if they exist. 
\begin{definition}\label{def:definition of psi JAJtoTensorProduct}
For a basis element
\[
a = \vv{p_1}{q_1}{1} \cdots \vv{p_n}{q_n}{n}
\]
of $\Jb_{\x} \sa nk \Jb_{\y}$, we define $\psi(a) = a$ if $[[1,n]]$ is a two-faced edge interval from $\x$ to $\y$. Otherwise, we define
\begin{equation*}
\psi(a) := \left( \prod_{i\in\CL{\x,\y}} U_i^{\frac{p_i+q_i-1}{2}} \right) 
\otimes \psi_{G_\lda}(a|_{G_\lda}) \otimes \psi_{G_1}(a|_{G_1}) \otimes \cdots \otimes \psi_{G_b}(a|_{G_b}) \otimes \psi_{G_\rho}(a|_{G_\rho}),
\end{equation*}
where we set $\psi_{G_\lda}(a|_{G_\lda}) = 1_{\F_2}$ (resp.~$\psi_{G_\rho}(a|_{G_\rho}) = 1_{\F_2}$) if there is no left (resp.~right) edge interval (the various maps $\psi_{G_j}$ are the isomorphisms described in Equations \ref{eq:psiG for non edge}, \ref{eq:psiG for right edge}, and \ref{eq:psiG for left edge}).
\end{definition}

Next, we will show that $\psi$ is a chain map and that it preserves the gradings.

\begin{lemma}\label{lem:differential via generating ints}
The differential $\de$ on $\Jb_\x\sa nk \Jb_\y$ satisfies
\[\de = \de_{G_\lda} + \de_{G_1} + \cdots + \de_{G_b} + \de_{G_\rho} \]
for the various intervals $G_j$ described above (if either $G_\lda$ or $G_\rho$ is empty, we have $\de_\emptyset = 0$).
\end{lemma}
\begin{proof}
By Proposition \ref{prop:QuartumNonDatur}, generating intervals, edge intervals, and crossed lines form a partition of $[1,n]$. Therefore, by Corollary \ref{cor:splitdifferential},
\[
\de = \de_{G_\lda} + \de_{G_1} + \cdots + \de_{G_b} + \de_{G_\rho} + \sum_{i\in\CL{\x,\y}}\de_i
\]
on $\Jb_\x \sa nk \Jb_\y$.
If line $i$ is crossed, then by Lemma \ref{lem:cl} any basis element $a$ of $\Jb_\x \sa nk \Jb_\y$ has only one strand on the $i$-th backbone. Thus, by Definition \ref{def:differential}, $\de_i a = 0$. It follows that, for all $i\in\CL{\x,\y}$, $\de_i$ vanishes on $\Jb_\x \sa nk \Jb_\y$, and the lemma is proved.
\end{proof}

\begin{corollary}
\label{cor:psiChainMap}
The map $\psi$ of Definition \ref{def:definition of psi JAJtoTensorProduct} is a chain map.
\end{corollary}
\begin{proof}
For any basis element $a\in \Jb_\x\sa nk \Jb_\y$, we rename $G_0:=G_\lda$ and $G_{b+1}:= G_\rho$ (again allowing for either to be the empty interval).  We then use Lemmas \ref{lem:differential via generating ints} and \ref{lem:x|_G} to compute
\begin{align*}
\psi( \de a) & = \psi \left( \de_{G_0} a + \de_{G_1}a + \dots + \de_{G_b} a + \de_{G_{b+1}} a\right) = \sum_{d =0}^{b+1} \psi (\de_{G_d} a)\\
& = \sum_{d =0}^{b+1} \left( \prod_{i\in\CL{\x,\y}} U_i^{\frac{p_i+q_i-1}{2}} \right) \otimes \psi_{G_0}(a|_{G_0}) \otimes \cdots \otimes \psi_{G_d}(\de a|_{G_d}) \otimes \cdots \otimes \psi_{G_{b+1}}(a|_{G_{b+1}}) \\
& = \sum_{d =0}^{b+1} \left( \prod_{i\in\CL{\x,\y}} U_i^{\frac{p_i+q_i-1}{2}} \right) \otimes \psi_{G_0}(a|_{G_0}) \otimes \cdots \otimes \de\left(\psi_{G_d}(a|_{G_d})\right) \otimes \cdots \otimes \psi_{G_{b+1}}(a|_{G_{b+1}}) \\
& = \de (\psi(a)).
\end{align*}
To derive the equality on the second line, note that no term in the differential can affect strands on backbones corresponding to crossed lines (see the proof of Lemma \ref{lem:differential via generating ints}), so that $\prod_{i\in\CL{\x,\y}} U_i^{\frac{p_i+q_i-1}{2}}$ is indeed the first factor of each term in the sum.  Similarly, $\de_{G_d}(a|_{G_e}) = a|_{G_e}$ whenever $d\neq e$, while $\de_{G_d}(a|_{G_d})=\de(a|_{G_d})$.
\end{proof}

\begin{lemma}
\label{lem:psiPreservesGradings}
The map $\psi$ of Definition \ref{def:definition of psi JAJtoTensorProduct} preserves the Alexander multi-grading and the Maslov grading.
\end{lemma}
\begin{proof}
We have already seen how the Alexander multi-grading $w_i$ is preserved under each $\psi_{G_j}$ for $i$ in a generating or edge interval $G_j$, and since generating intervals, edge intervals, and crossed lines form a partition of $[1,n]$ (Proposition \ref{prop:QuartumNonDatur}), it only remains to verify the preservation for the crossed lines. This claim follows from the definition of the grading on the crossed-lines algebra (recall that $w_i(1):=\frac{1}{2}$ if $i\in\CL{\x,\y}$, which offsets the $-1$ in the numerator of $U_{i}^{\frac{p_{i}+q_{i}-1}{2}}$).

The Maslov grading is similar. Since the grading on the tensor product will be a sum of gradings on each factor, and the generating interval factors preserve their contributions to this sum, we again only mention the crossed lines.  Note that for a crossed line $i$, one of $p_i,q_i$ is zero so that $|p_i-q_i|=p_i+q_i$, and so the corresponding summand in Definition \ref{def:gradings} is
\[
\frac{|p_i-q_i|}{2} - (p_i+q_i) + \left(0 + \frac{p_i + q_i}{2}\right) = \frac{|p_i-q_i|-(p_i+q_i)}{2} = 0,
\]
agreeing with the Maslov grading on the crossed-lines algebra.
\end{proof}

\subsection{Definition of \texorpdfstring{$\phi$}{phi}}
\label{sec:TensorProducttoJAJ}

In this subsection we define a map
\[
\phi \colon \F_2[U_i \, | \, i\in\CL{\x,\y}] \otimes \overline{A}_\circ(l_0) \otimes \overline{A}(l_1) \otimes \cdots \otimes \overline{A}(l_{b}) \otimes \overline{A}_\circ(l_{b+1}) \to \Jb_\x \A(n,k) \Jb_\y
\]
which will be the inverse to $\psi$.

For a crossed line $i$ from $\x$ to $\y$, we define, for a non-negative integer $r$,
\begin{equation}
\label{eq:phixyir}
\phi_{\x, \y}^i(r) =
\begin{cases}
\vv{1+2r}{0}i & v_i(\x, \y) = +1\\
\vv{0}{1+2r}i & v_i(\x, \y) = -1.
\end{cases}
\end{equation}

Recall that, given a generating interval $[j+1, j+l]$ from $\x$ to $\y$ and a standard basis element $a \in \Jb_{[j+1, j+l-1]} \sa{n}{l-1} \Jb_{[j+1, j+l-1]}$, we denote the array of vectors defining it by
\[
\Ar(a) = \vv{p_{j+1}}{q_{j+1}}{j+1} \cdots \vv{p_{j+l}}{q_{j+l}}{j+l}
\]

(see Lemma \ref{lem:piqi}, noting that all of the other vectors would have zeroes as entries). Note that $\Ar(a)$ itself does not record the ingoing and outgoing idempotents, so we can interpret it as a standard basis element in another idempotent as necessary. Arrays $\Ar(a)$ for edge intervals are defined similarly.

\begin{definition}\label{def:PhiHelper}
Given a monomial $\CLmonomial$ in $\F_2[U_i \, | \, i\in\CL{\x,\y}]$ and standard basis elements $a_c \in \overline{A}(l_c)$, define an array of vectors 
\begin{equation}\label{eq:PhiHelperArrayDef}
\Ar\left(\CLmonomial, a_1, \ldots, a_b\right) := \left( \prod_{i\in\CL{\x,\y}} \phi_{\x,\y}^i(r_i) \right) \cdot \Ar(\psi_{G_1}^{-1}(a_1)) \cdots \Ar(\psi_{G_b}^{-1}(a_b)),
\end{equation}
where we implicitly put the vectors appearing on the right side of the equation in increasing order. By Proposition~\ref{prop:QuartumNonDatur}, no line can be a crossed line while also belonging to a generating interval, so each vector in Equation~\eqref{eq:PhiHelperArrayDef} appears with a different index.

If there is a left edge interval $G_\lda = [[1,l_0]$, and we have $a_0 \in \overline{A}_\lda(l_0)$, then in Equation~\eqref{eq:PhiHelperArrayDef} we should include the term $\Ar(\psi_{G_\lda}^{-1}(a_0))$ as well.
If there is a right edge interval $G_\rho = [n-l_{b+1}+1,n]]$, and we have $a_{b+1} \in \overline{A}_\lda(l_{b+1})$, then in Equation~\eqref{eq:PhiHelperArrayDef} we should include the term $\Ar(\psi_{G_\rho}^{-1}(a_{b+1}))$ as well.
Lastly, if $\x = \y = [0,n]$, i.e.~there is a two-faced interval, then our monomial $\CLmonomial$ is $1$ and we have a single standard basis element $a_1 =: a$ of $\overline{A}_{\lda\rho}(n)$. We define $\Ar(a)$ to be the array associated to this basis element in Lemma \ref{lem:piqi}.
\end{definition}

\begin{lemma}
\label{lem:phiwelldefined}
The array $\Ar(\CLmonomial, a_1, \ldots, a_b)$ from Definition~\ref{def:PhiHelper} represents a basis element of $\Jb_\x \A \Jb_\y$ under the correspondence of Lemma~\ref{lem:piqi}.
\end{lemma}
\begin{proof}
We prove the lemma in the case where there are no edge intervals. The other cases are a straightforward variation of this proof (where one uses Lemmas \ref{lem:BasistdAl} and/or \ref{lem:BasistdAr} in addition to Lemma \ref{lem:BasistdA} below). We check that $\Ar(\CLmonomial, a_1, \ldots, a_n)$ satisfies the conditions of Lemma \ref{lem:piqi}, hence it represents a standard basis element of $\Jb_\x \sa nk$, and that the right idempotent of this basis element is $\Jb_\y$.   Note that by hypothesis we know that $\x$ and $\y$ are not far.  We will use $\dumline$ below for the index $i$ in items \eqref{it:pq no matched start}--\eqref{it:pq dots move other dots} of Lemma~\ref{lem:piqi}.

Condition~\eqref{it:no double C} of Lemma \ref{lem:piqi} is immediate because there are no $C_\dumline$ variables under consideration, so we begin with condition \eqref{it:pq no matched start}. Each of $\dumline$ and $\dumline+1$ is either a crossed line or belongs to a generating interval. If $\dumline$ and $\dumline+1$ belong to the same generating interval $G_c$, condition \eqref{it:pq no matched start} follows immediately from \eqref{it:pq no matched start} for $a_c$. If $\dumline$ belongs to a generating interval $G$, but $\dumline+1$ does not, then $q_\dumline = 0$ (see Lemma \ref{lem:BasistdA}). Analogously, if $\dumline+1$ belongs to a generating interval $G$, but $\dumline$ does not, then $p_{\dumline+1} = 0$.
Lastly, if both $\dumline$ and $\dumline+1$ are crossed lines, suppose that both $q_\dumline$ and $p_{\dumline+1}$ are nonzero. Then, by equation \eqref{eq:phixyir}, $v_\dumline(\x,\y) = -1$ and $v_{\dumline+1}(\x,\y)= 1$, from which we deduce
\[
\y \cap [\dumline, n] - \y \cap [\dumline+1, n] = \x \cap [\dumline, n] - \x \cap [\dumline+1, n] -2 < 0,
\]
contradicting that $[\dumline+1,n] \subset [\dumline,n]$. Thus we must have at least one of $q_\dumline,p_{\dumline+1}$ equal to zero, and in all cases condition \eqref{it:pq no matched start} is satisfied.

For condition \eqref{it:pq no dots pinzer} of Lemma \ref{lem:piqi}, we consider the same cases.  If $\dumline,\dumline+1\in G_c$, then the condition is guaranteed by the same condition for $a_c$.  If only $\dumline$ (respectively $\dumline+1$) is in some $G$, then Lemma \ref{lem:BasistdA} implies $q_\dumline=0$ (respectively $p_{\dumline+1}=0$) which in turn forces $p_\dumline$ even (respectively $q_{\dumline+1}$ even), and both cases satisfy condition \eqref{it:pq no dots pinzer}.  Finally, if both $\dumline,\dumline+1$ are crossed lines, equation \eqref{eq:phixyir} allows for a proof by contradiction as above, so that in all cases condition \eqref{it:pq no dots pinzer} is satisfied.

For condition \eqref{it:pq start from x}, we prove the contrapositive.  Assume that $q_\dumline \neq 0$. If $\dumline$ is a crossed line, then $v_{\dumline}(\x,\y) = -1$ by equation \eqref{eq:phixyir}, so $\dumline\in\x$ because $\x$ and $\y$ are not far.
If $\dumline$ belongs to some generating interval instead, then by Lemma \ref{lem:BasistdA}, item \eqref{it:A2}, we can conclude that line $\dumline +1$ is also in the generating interval, so that coordinate $\dumline$ must be fully used and $j\in\x$. The argument is similar when one assumes that $p_{\dumline+1} \not= 0$.

For condition~\eqref{it:pq disjoint endpoints for non-const} of Lemma~\ref{lem:piqi}, equation \eqref{eq:phixyir} implies that $\dumline$ belongs to a generating interval, so condition \eqref{it:A3} of Lemma \ref{lem:BasistdA} guarantees that $p_\dumline\equiv q_\dumline \pmod 2$ as desired.

Finally, for condition \eqref{it:pq dots move other dots} of Lemma \ref{lem:piqi}, the assumption that one of $p_\dumline,q_\dumline$ is odd while the other is zero implies that $\dumline$ is a crossed line (by condition \eqref{it:A3} of Lemma \ref{lem:BasistdA}, $\dumline$ cannot be contained in a generating interval).  If $p_\dumline$ is odd, then $v_\dumline(\x,\y)=1$ by equation~\eqref{eq:phixyir}. Since $\dumline\in\x$, we must have $v_{\dumline+1}(\x,\y)=1$ as well, so that $\dumline+1$ is a crossed line and equation \eqref{eq:phixyir} gives $p_{\dumline+1}$ odd as desired.  If $q_\dumline$ is odd, a similar argument forces $\dumline-1$ to be a crossed line with $q_{\dumline-1}$ odd.

We now check that the ending I-state of the element $a \in \Jb_\x \A$ defined by the array of vectors in equation~\eqref{eq:PhiHelperArrayDef} is indeed $\y$. Let $\y'$ denote the ending I-state of $a$ as characterized in Lemma \ref{lem:piqi}.  We will show that $\y\subset\y'$, which is sufficient since $|\y|=|\x|=|\y'|$.

Suppose $\dumline\in\y$.  If $\dumline\notin\x$, then we must have either $v_\dumline(\x,\y)=1$ or $v_{\dumline+1}(\x,\y)=-1$.  In the first case $\Ar(\CLmonomial, a_1, \ldots, a_n)$ contains $\vv{2r_\dumline+1}{0}{\dumline}$, and in the second case $\Ar(\CLmonomial, a_1, \ldots, a_n)$ contains $\vv{0}{2r_{\dumline+1}+1}{\dumline+1}$.  Either way, we have $\dumline\in\y'$ by Lemma \ref{lem:piqi}.  On the other hand, if $\dumline\in\x$, we have several cases to consider.  If line $\dumline$ is crossed, first suppose $p_\dumline$ is odd and $q_\dumline$ is zero. Then $\dumline \in \y'$ by Lemma~\ref{lem:piqi}. Alternately, if $p_\dumline$ is zero and $q_\dumline$ is odd, then $v_\dumline(\x,\y) = -1$. Since $\dumline \in \y$, we must have $v_{\dumline+1}(\x,\y) = -1$. It follows that $q_{\dumline+1}$ is odd, so $\dumline \in \y'$. The argument when line $\dumline+1$ is crossed is analogous. Finally, if neither line $\dumline$ nor $\dumline+1$ is crossed, then $\dumline$ and $\dumline+1$ are part of a generating interval and we have $p_\dumline\equiv q_\dumline \pmod{2}$ and $p_{\dumline+1}\equiv q_{\dumline+1} \pmod{2}$ by Lemma \ref{lem:BasistdA}.  Since $\dumline\in\x$, even in the case when some or all of these integers are zero, we must have $\dumline\in\y'$ by Lemma \ref{lem:piqi} and we are done.
\end{proof}

\begin{definition}\label{def:TensorProducttoJAJ}
Let $\CLmonomial$ be a monomial in $\F_2[U_i \, | \, i\in\CL{\x,\y}]$ and let $a_c \in \overline{A}(l_c)$ be standard basis elements. Define 
\[
\phi\left( \CLmonomial \otimes a_1 \otimes \cdots \otimes a_{b} \right)
\]
to be the element of $\Jb_\x \A \Jb_\y$ represented by $\Ar(\CLmonomial, a_1, \ldots, a_b)$.
\end{definition}

By Lemma~\ref{lem:phiwelldefined}, $\phi$ is well-defined.

\subsection{Proof of the splitting theorem}
\label{sec:proof of splitting JAJ}

In this subsection we prove Theorem \ref{thm:JAJtoTensorProduct}.

\begin{lemma}
\label{lem:phi=psi^-1}
The maps $\psi$ and $\phi$ defined in Sections \ref{sec:JAJtoTensorProduct} and \ref{sec:TensorProducttoJAJ} are inverses to each other.
\end{lemma}
\begin{proof}
Let
\[
a = \vv{p_1}{q_1}{1} \cdots \vv{p_n}{q_n}{n} \in \Jb_\x \sa nk \Jb_\y
\]
be a standard basis element (we resume our practice of identifying algebra elements and the arrays of vectors representing them). If line $i$ is crossed, then by Lemma \ref{lem:cl} we have
\begin{equation*}
\vv{p_i}{q_i}{i} =
\begin{cases}
\vv{p_i + q_i}{0}i & v_i(\x, \y) = +1,\\
\vv{0}{p_i + q_i}i & v_i(\x, \y) = -1.
\end{cases}
\end{equation*}
In $\psi(a)$ we get a factor $U_i^{\frac{p_i + q_i - 1}{2}}$, which produces a factor $\phi_{\x,\y}^i(\frac{p_i + q_i - 1}{2})$ in $\phi(\psi(a))$. By equation \eqref{eq:phixyir}, this factor agrees with $\vv{p_i}{q_i}{i}$. If $i$ is in a generating or edge interval, then $\vv{p_i}{q_i}{i}$ also appears as a factor in $\phi \circ \psi (a)$. As we noted after the definition of $\phi$, a factor indexed by some number $i$ does not appear more than once in the formula for $\phi \circ \psi(a)$. Thus, $a = \phi \circ \psi (a)$. The proof that $\psi \circ \phi = \id$ is similar and is left to the reader.
\end{proof}

\begin{proof}[Proof of Theorem \ref{thm:JAJtoTensorProduct}]
The map $\psi$ defined in Section \ref{sec:JAJtoTensorProduct} is a chain map (Corollary \ref{cor:psiChainMap}), it preserves the Alexander multi-grading and the Maslov grading (Lemma \ref{lem:psiPreservesGradings}), and it is bijective, since we exhibited an inverse map $\phi$ (Lemma \ref{lem:phi=psi^-1}). Thus, it is an isomorphism of chain complexes.
\end{proof}

With Theorem \ref{thm:JAJtoTensorProduct} in hand, we set out to compute the homology $\Jb_\x H_*(\sa nk) \Jb_\y$.  By the K{\"u}nneth theorem, it is enough to understand the homology of each generating algebra and edge algebra in the decomposition of Theorem \ref{thm:JAJtoTensorProduct}.  The next several sections are devoted to computing these homology groups.

\subsection{Elements of the generating algebra in Maslov degree zero}
\label{sec:Elements of Gen in Maslov=0}

Note that when $\Sc = \varnothing$, Definition~\ref{def:gradings} implies that the Maslov degree $\m(a)$ is nonpositive for any $a \in \A(n,k) = \A(n,k,\varnothing)$. In this subsection we study some homogeneous elements of the generating algebra in the maximal Maslov degree, namely zero.  Throughout this section, the Alexander grading refers to the refined Alexander grading of Definition \ref{def:gradings}.

\begin{definition}
For a vector $\ru = (r_1, \ldots, r_l) \in \Z_{\geq 0}^l$, define
\[
\overline{A}^\ru (l) = \set{x \in \overline{A}(l) \,\middle|\, w(x) = \ru},
\]
the vector subspace (or subcomplex) of $\overline{A}(l)$ consisting of all Alexander-homogeneous elements of Alexander degree $\ru$.
\end{definition}
Note that for every standard basis element $x \in \overline{A}(l)$, condition \eqref{it:A3} of Lemma~\ref{lem:BasistdA} implies that we have $w(x) \in \Z_{\geq 0}^l$, rather than just $(\frac12\Z_{\geq 0})^l$.

We start with the following observation.

\begin{proposition}
\label{prop:Aru(l)=0}
If $r_i>0$ for all $i=1, \ldots, l$, then $\overline{A}^\ru (l)=0$.
\end{proposition}
\begin{proof}
If a standard basis element
\[
a = \vv{p_1}{q_1}{1} \cdots \vv{p_l}{q_l}{l}
\]
is in $\overline{A}^\ru (l)$, then $q_i p_{i+1} =0$ for all $i \in [1, l-1]$ by Lemma \ref{lem:BasistdA}, and $p_1 = q_l =0$ by the same lemma. Therefore, at least $l+1$ numbers among $p_1, q_1, \ldots, p_l, q_l$ must vanish. By the pigeonhole principle, there must exist $i$ such that $p_i = q_i = 0$, hence $r_i = 0$.  The result follows.
\end{proof}

\begin{definition}
For each $\ru \in \Z_{\geq 0}^l$, we define an element $a^\ru\in\overline{A}^\ru(l)$ by
\begin{equation}
\label{eq:defaru}
a^\ru := \Jb_{[1, l-1]} \prod_{i \,:\, r_i \not=0}\left[\vv{2r_i}{0}i + \vv{0}{2r_i}i\right] \Jb_{[1, l-1]}.
\end{equation}
\end{definition}
Note that after expanding the product defining $a^\ru$, some terms may vanish. For example, if $r_i \neq 0$ for all $i$, then $a^\ru$ must vanish by Proposition \ref{prop:Aru(l)=0}.
The role of the two factors $\Jb_{[1, l-1]}$ is to kill the possible terms containing $\vv{2r_1}{0}1$ or $\vv{0}{2r_l}l$, which are not in the generating algebra.

\begin{lemma}
\label{lem:aruMaslov0}
For each $\ru \in \Z_{\geq 0}^l$, we have $\m(a^\ru) = 0$.
\end{lemma}

\begin{proof}
Each factor $\vv{2r_i}0i + \vv{0}{2r_i}i$ of $a^\ru$ appearing in Equation \eqref{eq:defaru} has vanishing Maslov degree by Definition \ref{def:gradings}. The Maslov degree of their product therefore vanishes too.
\end{proof}

\begin{remark}
\label{rem:aruMaslovmax}
By Definition \ref{def:gradings}, the Maslov degree $\m(a)$ of any homogeneous element $a \in \overline{A}^\ru(l)$ is non-negative.
By expanding Equation \eqref{eq:defaru}, one can check that if $r_i = 0$ for some $i$, then $a^\ru \neq 0$, so the Maslov degree of $a^\ru$ is the maximal Maslov degree in $\overline{A}^\ru(l)$. In fact, $a^\ru$ is the sum of all standard basis elements of $\overline{A}^\ru(l)$ in Maslov degree $0$.
\end{remark}

\begin{lemma}
\label{lem:arucdotaru}
For all $\ru$ and $\ru'$, we have $a^\ru \cdot a^{\ru'} = a^{\ru + \ru'}$ (taking the product in $\overline{A}(l)$).
\end{lemma}
\begin{proof}
First note that if both $r_i$ and $r_i'$ are non-vanishing, then
\[
\left[\vv{2r_i}{0}i + \vv{0}{2r_i}i\right] \cdot \left[\vv{2r_i'}{0}i + \vv{0}{2r_i'}i\right] = \vv{2(r_i+r_i')}{0}i + \vv{0}{2(r_i+r_i')}i,
\]
since the products $\vv{2r_i}{0}i \cdot \vv{0}{2r_i'}i$ and $\vv{0}{2r_i}i \cdot \vv{2r_i'}{0}i$ vanish by condition~\eqref{it:evens no bigons} of Lemma \ref{lem:concatenable}. It follows that $a^\ru \cdot a^{\ru'}$ is the product of the elements
\[
\vv{2(r_i+r_i')}{0}i + \vv{0}{2(r_i+r_i')}i
\]
over $i$ such that at least one of $r_i, r_i'$ is nonzero. The result is $a^{\ru+\ru'}$. 

A separate note should be made for the cases $i=1$ and $i=l$, for which the factor $\vv{2r_i}{0}i + \vv{0}{2r_i}i$ is replaced by either $\vv{0}{2r_i}i$ or $\vv{2r_i}{0}i$. In these cases, we have the equalities
\[
\vv{0}{2r_1}1 \cdot \vv{0}{2r_1'}1 = \vv{0}{2(r_1 + r_1')}1 \qquad \mbox{and} \qquad \vv{2r_l}{0}l \cdot \vv{2r_l'}{0}l = \vv{2(r_l + r_l')}{0}l. \qedhere
\]
\end{proof}

\subsection{Elements of the edge algebras in Maslov degree zero}

In the case of the edge algebras, we have elements $a^\ru_\lda$, $a^\ru_\rho$ and $a^\ru_{\lda\rho}$ analogous to $a^\ru$. In this subsection, $\circ$ denotes either $\lda$, $\rho$, or $\lda\rho$.

\begin{definition}
For a vector $\ru = (r_1, \ldots, r_l) \in \Z_{\geq 0}^l$, define
\[
\overline{A}^\ru_\circ (l) = \set{x \in \overline{A}_\circ(l) \,\middle|\, w(x) = \ru},
\]
the vector subspace (or subcomplex) of $\overline{A}_\circ(l)$ consisting of all Alexander-homogeneous elements of Alexander degree $\ru$.
\end{definition}

\begin{definition}\label{def:MaximalMaslovDegStrandsGens}
For each $\ru \in \Z_{\geq 0}^l$, we define
\begin{align*}
a^\ru_\lda &:= \Jb_{[0, l-1]} \prod_{i \,:\, r_i \not=0}\left[\vv{2r_i}{0}i + \vv{0}{2r_i}i\right] \Jb_{[0, l-1]} \in \overline{A}_\lda(l);\\
a^\ru_\rho &:= \Jb_{[1, l]} \prod_{i \,:\, r_i \not=0}\left[\vv{2r_i}{0}i + \vv{0}{2r_i}i\right] \Jb_{[1, l]} \in \overline{A}_\rho(l);\\
a^\ru_{\lda\rho} &:= \Jb_{[0, l]} \prod_{i \,:\, r_i \not=0}\left[\vv{2r_i}{0}i + \vv{0}{2r_i}i\right] \Jb_{[0, l]} \in \overline{A}_{\lda\rho}(l).
\end{align*}
\end{definition}

The proofs of the next two lemmas are similar to those of Lemmas \ref{lem:aruMaslov0} and \ref{lem:arucdotaru}, and are omitted.

\begin{lemma}
\label{lem:arucircMaslov0}
For each $\ru \in \Z_{\geq 0}^l$, we have $\m(a^\ru_\circ) = 0$.
\end{lemma}

\begin{remark}
\label{rem:arucircMaslovmax}
By expanding the equations in Definition \ref{def:MaximalMaslovDegStrandsGens}, we have that $a^\ru_\circ \neq 0$ for all $\ru \in \Z_{\geq 0}^l$.
Thus, as in Remark \ref{rem:aruMaslovmax}, the Maslov degree of $a^\ru_\circ$ is the maximal Maslov degree in $\overline{A}^\ru_\circ(l)$. Moreover, one could characterize $a^\ru_\circ$ as the sum of standard basis elements of $\overline{A}^\ru_\circ(l)$ in Maslov degree $0$.
\end{remark}

\begin{lemma}
\label{lem:arucdotarucirc}
For all $\ru$ and $\ru'$, we have $a^\ru_\circ \cdot a^{\ru'}_\circ = a^{\ru + \ru'}_\circ$.
\end{lemma}

\subsection{The homology of \texorpdfstring{$\sa nk$}{A(n,k)}}

The main technical result we use in this section is the following lemma, whose proof we postpone.

\begin{lemma}
\label{lem:homology}
For all $l > 0$ and $\ru \in \Z_{\geq 0}^l$,
\begin{enumerate}
\item \label{it:homA} If $r_i=0$ for some $i \in [1,l]$, then $H_*(\overline{A}^\ru(l)) \cong \F_2$; otherwise $H_*(\overline{A}^\ru(l)) =0$
\item \label{it:homB} $H_*(\overline{A}^\ru_\lda(l)) \cong \F_2$
\item \label{it:homC} $H_*(\overline{A}^\ru_\rho(l)) \cong \F_2$
\item \label{it:homD} $H_*(\overline{A}^\ru_{\lda\rho}(l)) \cong \F_2$.
\end{enumerate}
In all nonzero cases, the homology is generated by the cycle $a^\ru_\circ$. In particular, by Lemmas \ref{lem:aruMaslov0} and \ref{lem:arucircMaslov0}, it is concentrated in Maslov degree zero.
\end{lemma}

We will prove that $a^\ru_\circ$ is indeed a cycle while proving Lemma \ref{lem:homology}.

From Lemma \ref{lem:homology} we deduce the following theorem. In the next theorem, $\F_2[U_1, \ldots, U_l]$ is endowed with an Alexander multi-grading by setting 
\[w_i(1)=0, \,\,\, w_i(U_j)=\delta_{i,j} \,\,\, \text{for all $i,j\in\set{1,\dots,l}$}.\]
We define the Maslov grading to be zero on $\F_2[U_1,\ldots,U_l]$.

\begin{theorem}
\label{thm:homology}
For all $l > 0$, we have the following isomorphisms of graded $\F_2$-vector spaces:
\begin{enumerate}
\item \label{it:caseA} $H_*(\overline{A}(l)) \cong \frac{\F_2[U_1, \ldots, U_l]}{U_1 \cdots U_l}$
\item $H_*(\overline{A}_\lda(l)) \cong \F_2[U_1, \ldots, U_l]$
\item $H_*(\overline{A}_\rho(l)) \cong \F_2[U_1, \ldots, U_l]$
\item $H_*(\overline{A}_{\lda\rho}(l)) \cong \F_2[U_1, \ldots, U_l]$.
\end{enumerate}
In all the cases, the isomorphism sends $a^\ru_\circ$ to $U_1^{r_1} \cdots U_l^{r_l}$.
\end{theorem}

\begin{proof}
We prove the case \eqref{it:caseA}; the proof in the other cases requires only slight modification.
Note that we have a splitting
\[
\overline{A}(l) = \bigoplus_{\ru \in \Z_{\geq 0}^l} \overline{A}^\ru(l)
\]
as chain complexes, by Proposition \ref{prop:effectsofde}. Thus we have a natural splitting
\[
H_*(\overline{A}(l)) = \bigoplus_{\ru \in \Z_{\geq 0}^l} H_*(\overline{A}^\ru(l)) = \bigoplus_{\substack{\ru \in \Z_{\geq 0}^l \\ \exists i \,:\, r_i = 0}} \F_2 [a^\ru],
\]
by Lemma \ref{lem:homology}.

We define a linear map
\[
L \colon H_*(\overline{A}(l)) \to \frac{\F_2[U_1, \ldots, U_l]}{U_1 \cdots U_l}
\]
by setting $L(a^\ru) = U_1^{r_1} \cdots U_l^{r_l}$. The check that the map is bijective and that it preserves the gradings is left to the reader.
\end{proof}
Note that $L$ is in fact an isomorphism of $\F_2$-algebras by Lemmas \ref{lem:arucdotaru} and \ref{lem:arucdotarucirc}.

\begin{corollary}\label{cor:general homology of JAJ}
For $\x,\y\in V(n,k)$, there is an isomorphism
\begin{equation}\label{eq:general homology of JAJ}
\psi \colon \Jb_\x H_*(\sa nk) \Jb_\y \stackrel{\sim}{\longrightarrow}
\begin{cases}
0 & \text{if $\x$ and $\y$ are far}\\
\frac{\F_2[U_1,\dots,U_n]}{(p_G\,|\,\text{$G$ generating interval})} & \text{otherwise}
\end{cases}
\end{equation}
The Maslov grading on the right-hand side of equation \eqref{eq:general homology of JAJ} is zero, and the Alexander multi-grading on the right-hand side of equation \eqref{eq:general homology of JAJ} is defined as follows:
\[
w_i(1):=\begin{cases}
\frac{1}{2} & \text{if $i\in\CL{\x,\y}$}\\
0 & \text{otherwise}
\end{cases}
\qquad
w_i(U_j):=\delta_{i,j}. 
\]

\end{corollary}

\begin{proof}
Theorem \ref{thm:JAJtoTensorProduct} gives us a decomposition of $\Jb_\x \sa nk \Jb_\y$ into a tensor product, and the K{\"u}nneth theorem for tensor products over $\F_2$ guarantees that we can compute the overall homology by tensoring together the homologies of the different factors.  The factor $\F_2[U_i \, | \, i\in\CL{\x,\y}]$ corresponding to the crossed lines has no differential, while the homologies of the generating intervals and edge intervals are computed in Theorem \ref{thm:homology}.  When tensored all together, we get a graded vector space that is isomorphic to the right hand side of equation \eqref{eq:general homology of JAJ}.
\end{proof}

\begin{corollary}\label{cor:HomologyBasisAsStrandsGens}
Applying the inverse of the isomorphism from Corollary~\ref{cor:general homology of JAJ} to a monomial $U_1^{r_1}\cdots U_n^{r_n}$, we get the homology class of the element
\begin{equation}\label{eq:H(JAJ) general generator}
\Jb_\x \cdot \prod_{\stackrel{i\in\CL{\x,\y}}{v_i(\x,\y)=1}} \vv{1+2r_i}{0}{i} \cdot \prod_{\stackrel{j\in\CL{\x,\y}}{v_j(\x,\y)=-1}} \vv{0}{1+2r_j}{j} \cdot \prod_{\stackrel{k\notin\CL{\x,\y}}{r_k\neq 0}} \left( \vv{2r_k}{0}{k} + \vv{0}{2r_k}{k} \right) \cdot \Jb_\y
\end{equation}
of $\Jb_\x H_*(\sa nk) \Jb_\y$, where again the idempotents $\Jb_\x$ and $\Jb_\y$ force some of these summands to be zero.
\end{corollary}

\begin{remark}
If we compare Corollary~\ref{cor:general homology of JAJ} to Theorem~\ref{thm:ReviewOSzHomology}, we see that the two algebras $\Jb_\x \sa nk \Jb_\y$ and $\Ib_\x \B(n,k) \Ib_\y$ do indeed have isomorphic homology, at least as graded vector spaces. Using the explicit formulas of \cite[\corExplicitQuiverGensForMonomials]{MMW1} and \ref{cor:HomologyBasisAsStrandsGens}, one can check without too much work that this isomorphism holds on the level of graded algebras. Section \ref{sec:qi} will be devoted to realizing this isomorphism via a genuine map of dg algebras from $\B(n,k)$ to $\sa nk$, and more generally from $\B(n,k,\Sc)$ to $\sac nk\Sc$.
\end{remark}

\subsection{Proof of Lemma \ref{lem:homology}}

As a first step toward proving Lemma \ref{lem:homology}, we study the homology of $\overline{A}^r_{\lda\rho}(1)$, which will constitute the base case for an inductive proof of the aforementioned lemma. Note that in this case $\ru = r$ is just a natural number. To simplify the notation, we will denote the element $\vv{p}{q}{1}$ by $\vv pq{}$.

\begin{lemma}
\label{lem:HAldarho1}
For all $r \in \Z_{\geq 0}$, we have $H_*(\overline{A}^r_{\lda\rho}(1)) \cong \F_2$, generated by the cycle $a^r_{\lda\rho}$ and concentrated in Maslov degree zero.
\end{lemma}

\begin{proof}
Let $C_s \subseteq \overline{A}^r_{\lda\rho}(1)$ be the span of the standard basis elements $\vv pq{}$ with $w \vv pq{} = r$ and $\m \vv pq{} = s$.  By the formulas of Definition~\ref{def:gradings}, basis elements in $C_s$ satisfy $|p-q|=2s+2r$, so $C_s$ is 1-dimensional if $s = -r$, 2-dimensional if $-r<s\leq 0$, and $0$ otherwise (recall that no element of $\overline{A}^r_{\lda\rho}(1)$ has positive Maslov degree). Using the standard basis elements $\vv pq{}$ as bases for each vector space $C_s$, we get isomorphisms to $\F_2$, $\F_2^2$, or $0$. We do not specify how we make each basis an ordered basis, because any order will give the same result.

Since a basis for $\overline{A}^r_{\lda\rho}(1)$ is given by all vectors $\vv pq{}$ with $w \vv pq{} = r$, we obtain a splitting $\overline{A}^r_{\lda\rho}(1) = \bigoplus_{s= -r}^{0} C_s$. Moreover, $\de(C_s) \subseteq C_{s-1}$. Using the formulas of Lemma \ref{lem:general pq differential}, we can compute the matrix of each map $\de\colon C_s \to C_{s-1}$. The resulting chain complex is
\[
\xymatrix{
C_{0} \ar[r]^\de \ar[d]^{\vsimeq} & C_{-1} \ar[r]^\de \ar[d]^{\vsimeq} & \cdots \ar[r]^\de & C_{-r+1} \ar[r]^\de \ar[d]^{\vsimeq} & C_{-r} \ar[d]^{\vsimeq}\\
\F_2^2 \ar[r]_{\scriptscriptstyle{\begin{pmatrix} 1&1\\1&1 \end{pmatrix}}} & 
\F_2^2 \ar[r]_{\scriptscriptstyle{\begin{pmatrix} 1&1\\1&1 \end{pmatrix}}} & 
\cdots \ar[r]_{\scriptscriptstyle{\begin{pmatrix} 1&1\\1&1 \end{pmatrix}}} & 
\F_2^2 \ar[r]_{\scriptscriptstyle{\begin{pmatrix} 1\\1 \end{pmatrix}}} & \F_2
}
\]
The homology of this complex is $1$-dimensional, concentrated in Maslov degree zero. It is generated by the sum of the two basis elements $\vv{2r}0{} + \vv0{2r}{}$ if $r\not=0$ and by $\vv00{}$ if $r=0$. This sum equals $a^r_{\lda\rho}$ by definition.
\end{proof}

\begin{proof}[Proof of Lemma \ref{lem:homology}]
We argue by induction on $l$. First suppose that $l=1$.

\eqref{it:homA}. This claim follows from the fact that the only non-trivial element of $\overline{A}(1)$ is $a^0=\vv001$.

\eqref{it:homB}. For every $r \in \Z_{\geq 0}$, there is a unique non-trivial element in $\overline{A}_\lda^r(1)$, namely $a_\lda^r = \vv{2r}01$ (note that in this algebra $\vv0{2r}1$ is set to $0$, because it is not an element of $\Jb_{\{0\}} \sa 11 \Jb_{\{0\}}$). The claim follows.

\eqref{it:homC}. This claim is analogous to \eqref{it:homB}.

\eqref{it:homD}. This claim is the content of Lemma \ref{lem:HAldarho1}.

For the inductive step, we now suppose that claims \eqref{it:homA}--\eqref{it:homD} are true for all $k<l$, and we prove them for $l$.

\eqref{it:homA}. When $r_i>0$ for all $i\in [1,l]$, by Proposition \ref{prop:Aru(l)=0} the algebra $\overline{A}^\ru(l)$ is trivial and so is its homology. Now suppose that $r_i = 0$ for some fixed $i\in[1,l]$. If $a \in \overline{A}^\ru(l)$ is a standard basis element, then $p_i = q_i = 0$. Let $I_L = [1,i-1]$ and $I_R = [i+1, l-1]$. Then $a|_{I_L}$ and $a|_{I_R}$ completely determine $a$. In fact, there is an isomorphism of complexes
\begin{equation*}
\alpha \colon \overline{A}^\ru(l) \to \overline{A}^{\ru'}_\rho (i-1) \otimes \overline{A}^{\ru''}_\lda (l-i)
\end{equation*}
sending a standard basis element $a$ to $\psi_{I_L}(a|_{I_L}) \otimes \psi_{I_R}(a|_{I_R})$, where $\ru'$ and $\ru''$ are the restrictions of $\ru$ to the first $i-1$ and the last $l-i$ coordinates. It is straightforward to check that the correspondence is bijective, and that it preserves the gradings (after shifting the Alexander multi-grading on $\overline{A}^{\ru''}_\lda (n-i)$ as in Section~\ref{sec:splittingthm}). The fact that $\alpha$ is a chain map follows from Lemmas \ref{lem:x|_Glambda} and \ref{lem:x|_Grho}:
\begin{align*}
\alpha(\de a) &= \alpha \circ (\de_{I_L} + \de_{\{i\}} + \de_{I_R}) (a) = \alpha \circ (\de_{I_L}a + \de_{I_R}a) \\
&= (\psi_{I_L} \circ \de (a|_{I_L})) \otimes \psi_{I_R}(a|_{I_R}) + \psi_{I_L} (a|_{I_L}) \otimes (\psi_{I_R} \circ \de (a|_{I_R})) \\
&= \de \alpha(a).
\end{align*}
Lastly, it follows from the definition of $a^\ru$ that $\alpha(a^\ru) = a^{\ru'}_\rho \otimes a^{\ru''}_\lda$. Thus, by induction, $H_*(\overline{A}^\ru(l)) \cong \F_2$, generated by $a^\ru$.

\eqref{it:homB}. For a standard basis element $a \in \overline{A}_\lda^\ru(l)$, we have $q_l = 0$ and $p_l = 2r_l$ (see Lemma \ref{lem:BasistdAl}). The map
\begin{align*}
\beta \colon \overline{A}^\ru_\lda(l) &\to \overline{A}^{\ru'}_{\lda\rho} (l-1) \left\{{r_l}\right\} \\
a & \mapsto \psi_{[1,l-1]}(a|_{[1,l-1]})
\end{align*}
is an isomorphism of chain complexes, where $\{r_l\}$ denotes an upward translation in the final component of the Alexander multi-grading by $r_l$ and $\ru'$ is the restriction of $\ru$ to the first $l-1$ coordinates. The result then follows from case \eqref{it:homD} for $l-1$.

\eqref{it:homC}. This claim is analogous to \eqref{it:homB}.

\eqref{it:homD}. For convenience, write $\ru = (r, \ru')$. For all $p,q \in \Z_{\geq 0}$ such that $p+q=2r$, define $C_{p,q}$ to be the submodule of $\overline{A}_{\lda\rho}^\ru(l)$ generated by the standard basis elements $a$ such that $a|_{\{1\}} = \vv{p}{q}{1}$. For $m \in \Z$, define
\[
C_m := \bigoplus_{\substack{p,q \in \Z_{\geq 0} \\ p+q = 2r \\ m=\frac12|p-q|-r}} C_{p,q}.
\]
Note that if $m=-r$, then $C_{-r} = C_{r,r}$. If $-r<m \leq 0$, then $C_m = C_{2r+m, -m} \oplus C_{-m, 2r+m}$. For all other values of $m$, we have $C_m = 0$. Moreover,
\[
\overline{A}_{\lda\rho}^\ru(l) = \bigoplus_{m=-r}^0 C_m = \bigoplus_{\substack{p,q \in \Z_{\geq 0} \\ p+q = 2r}} C_{p,q}.
\]
The number $m$ is in fact the first summand of the Maslov grading, as one can check from Definition \ref{def:gradings}.

By Corollary \ref{cor:splitdifferential}, we have that $\de = \de_{\{1\}} + \de_{[2,l]}$. To simplify the notation, denote $\de_1 := \de_{\{1\}}$ and $\de_0 := \de_{[2,l]}$. For every $p,q$ such that $p+q=2r$, we have $\de_0(C_{p,q}) \subset C_{p,q}$. For all $m$, we have $\de_0(C_m) \subset C_m$ and $\de_1(C_m) \subset C_{m-1}$. Thus, we can define a filtration on $\overline{A}_{\lda\rho}^\ru(l)$ by setting
\[
\mc F_s = \bigoplus_{m \leq s} C_m.
\]

Every filtered chain complex induces a spectral sequence. We refer the reader to \cite{mccleary}, and in particular to Section 2.2, for a discussion about spectral sequences arising from filtered chain complexes. In the proof below, in $E_s^t$, the index $s$ denotes the filtration level (usually denoted by $p$), and $t$ denotes the page of the spectral sequence. We skip the homological grading (usually denoted by $q$) to simplify the notation. Note that in \cite{mccleary} subscripts and superscripts are swapped, since McCleary deals with cochain complexes rather than chain complexes.

The zeroth page $(E^0_s, d_0)$ of the associated spectral sequence is the associated graded module, with differential induced by $\de = \de_1 + \de_0$. Therefore, $E^0_s = \mc F_s / \mc F_{s-1} \cong C_s$ under the projection map sending each other summand of $\mc F_s$ to $0$, and the differential $d_0$ is identified with $\de_0$. Thus we have
\[
(E^0_s, d_0) \cong (C_s, \de_0) \cong \bigoplus_{\substack{p,q \in \Z_{\geq 0} \\ p+q = 2r \\ s=\frac12|p-q|-r}} (C_{p,q}, \de_0),
\]
since, as we observed, $C_{p,q}$ is a $\de_0$-subcomplex of $C_s$.

For each $p,q \in \Z_{\geq 0}$ with $p+q = 2r$ and $|p-q|=2s + 2r$, we have an isomorphism of complexes
\[
(C_{p,q}, \de_0) \stackrel{\sim}{\To}
\begin{cases}
\left(\overline{A}_{\lda\rho}^{\ru'}(l-1), \de\right)\{(r,\underline0)\}[s] & \text{if }q=0\\
\left(\overline{A}_{\rho}^{\ru'}(l-1), \de\right)\{(r,\underline0)\}[s] & \text{if }q\not=0\\
\end{cases}
\]
induced by the map $a \mapsto \psi_{[2,l]}(a|_{[2,l]})$ (here the brackets $\{\cdot\}$ denote an upward shift in the Alexander multi-grading and the brackets $[\cdot]$ denote an upward shift in the Maslov grading). Note that the element $b^\ru_{p,q} := \vv pq1 a^{(0,\ru')}_{\lda\rho}$ is sent under the above isomorphism of complexes to $a^{\ru'}_{\circ}$.

Thus, it follows from the inductive hypothesis that
\[
E^1_s = \bigoplus_{\substack{p,q \in \Z_{\geq 0} \\ p+q = 2r \\ s=\frac12|p-q|-r}} \mb F_2 \left\langle \left[ b^\ru_{p,q} \right] \right\rangle.
\]
The differential $d_1$ on $E^1$ is induced by $\de = \de_1 + \de_0$. Since all elements of $E^1$ are represented by $\de_0$-cycles, $d_1 = [\de_1]$.
Therefore, the restriction map $a \mapsto a|_{\{1\}}$ induces an isomorphism of complexes 
\[
(E^1, d_1) \cong \left(\overline{A}_{\lda\rho}^r(1), \de\right) \{(0,\ru')\}
\]
which sends $[b^\ru_{p,q}] \mapsto \vv pq{}$. By Lemma \ref{lem:HAldarho1}, the second page $E^2$ of the spectral sequence is $1$-dimensional, spanned by the homology class of the element
\[
\begin{cases}
\left(\vv{2r}01 + \vv0{2r}1 \right) \cdot a^{(0,\ru')}_{\lda\rho} & \text{if }r>0 \\
\vv001 \cdot a^{(0,\ru')}_{\lda\rho} & \text{if }r=0
\end{cases}
\]
which is the element $a^\ru_{\lda\rho}$ by definition.

Therefore, $E^2 \cong \mb F_2 \left\langle \left[a^\ru_{\lda\rho}\right] \right\rangle$, and we note that $a^\ru_{\lda\rho}$ must be a non-zero $\de_0$-cycle as well as a $\de_1$-cycle. Thus, $a^\ru_{\lda\rho}$ is a $\de$-cycle too.

Since $E^2$ is 1-dimensional, it follows that the spectral sequence collapses at the second page, so $E^2 \cong E^\infty$ is the associated graded module of the homology $H_*\left(\overline{A}^\ru_{\lda\rho}(l)\right)$. Thus, $H_*\left(\overline{A}^\ru_{\lda\rho}(l)\right)$ must be 1-dimensional as well. We noted above that $a^\ru_{\lda\rho}$ is a $\de$-cycle, and it cannot be a $\de$-boundary because it has Maslov degree zero which is maximal. Thus, $\set{\left[a^\ru_{\lda\rho}\right]}$ is a basis for $H_*\left(\overline{A}^\ru_{\lda\rho}(l)\right)$..
\end{proof}
\section{The quasi-isomorphism \texorpdfstring{$\Phi$}{Phi}}
\label{sec:qi}

\subsection{Defining \texorpdfstring{$\Phi$}{Phi}}

We are now in a position to define our map $\Phi:\B(n,k,\Sc)\rightarrow\sac nk\Sc$. We will use Theorem~\ref{thm:ReviewQuiverEquivalence}. Recall that both $\B(n,k,\Sc)$ and $\sac nk\Sc$ can be viewed as algebras over $\Ib(n,k)$.

\begin{remark}
Even when viewing $\B(n,k,\Sc)$ and $\sac nk\Sc$ as algebras over the same ring $\Ib(n,k)$, we will continue to denote the basic idempotents of $\B(n,k,\Sc)$ by $\Ib_{\x}$ and the basic idempotents of $\sac nk\Sc$ by $\Jb_{\x}$.
\end{remark}

\begin{remark}
In this section, we will implicitly make the identification 
\[
\B(n,k,\Sc) \cong \Quiv(\Gamma(n,k,\Sc), \td{\mc R}_{\Sc})
\]
as in Section~\ref{sec:MMW1Review}.
\end{remark}

\begin{definition} \label{def:Phi}
Let $\x,\y \in V(n,k)$ and let $\gamma$ be an edge in $\Gamma(n,k,\Sc)$ from $\x$ to $\y$. To $\gamma$, we associate an element $\Phi(\gamma)$ of $\Jb_{\x} \sac nk\Sc \Jb_{\y}$ as follows: 
\begin{enumerate}
\item\label{it:PhiOfRi} If $\gamma$ has label $R_i$, let $\Phi(\gamma) = \vv{1}{0}{i}$.

\item If $\gamma$ has label $L_i$, let $\Phi(\gamma) = \vv{0}{1}{i}$.

\item\label{eq:PhiUVanishing} If $\gamma$ has label $U_i$ and $\x \cap \set{i-1,i} = \varnothing$, then let $\Phi(\gamma) = 0$.

\item\label{eq:PhiDefRLU} If $\gamma$ has label $U_i$ and $\x \cap \set{i-1,i} = \set{i-1}$, then in $\B(n,k,\Sc)$, we can factor $\gamma$ uniquely as $\gamma = \gamma' \cdot \gamma''$ where $\gamma'$ has label $R_i$ and $\gamma''$ has label $L_i$. Let $\Phi(\gamma) := \Phi(\gamma') \Phi(\gamma'') = \vv{2}{0}{i}$.

\item\label{eq:PhiDefLRU} If $e$ has label $U_i$ and  $\x \cap \set{i-1,i} = \set{i}$, then in $\B(n,k,\Sc)$, we can factor $\gamma$ uniquely as $\gamma = \gamma' \cdot \gamma''$ where $\gamma'$ has label $L_i$ and $\gamma''$ has label $R_i$. Let $\Phi(\gamma) := \Phi(\gamma') \Phi(\gamma'') = \vv{0}{2}{i}$.

\item If $\gamma$ has label $U_i$ and $\x \cap \set{i-1,i} = \set{i-1,i}$, then let $\Phi(\gamma) = \vv{2}{0}{i} + \vv{0}{2}{i}$.

\item If $\gamma$ has label $C_i$, then let $\Phi(\gamma) = C_i$.

\end{enumerate}
By \cite[\propQuiverAlgUniversalProp]{MMW1}, the above data defines a homomorphism of $\Ib(n,k)$-algebras $\Phi: \Path(\Gamma(n,k,\Sc)) \to \sac nk\Sc$. In Lemma~\ref{lem:PhiWellDef} below, we will show that $\Phi$ sends the relation ideal $\td{\mc R}_{\Sc}$ defining $\B(n,k,\Sc)$ to zero, so that $\Phi$ induces a homomorphism of $\Ib(n,k)$-algebras from $\B(n,k,\Sc)$ to $\sac nk\Sc $.
\end{definition}

Visually, we imagine the map $\Phi$ as follows.  A multiplicative generator of $\B(n,k,\Sc)$ is an arrow in the quiver algebra, visualized as a motion of a dot across a line or as a formal $U_i$ or $C_i$ generator. In mapping this motion to the strands algebra $\sac nk\Sc$, we imagine the line $i$ of $\B(n,k,\Sc)$ as the core of the corresponding cylinder $\I \times S^1_i$ in $\I \times \Zc(n)$.  A dot between two lines $i$ and $i+1$ in $\B(n,k,\Sc)$ corresponds to a choice of matching $(z_i^+,z_{i+1}^-)$ for $\sa nk $.  A motion of a dot across a line $i$ in $\B(n,k)$ corresponds to the shortest oriented path around the cylinder $\I \times S^1_i$ from one matched basepoint to another in $\sac nk\Sc$. Stationary dots are interpreted as pairs of dashed strands for the corresponding matchings in $\sac nk\Sc$. A $U_i$ loop that factors as $R_iL_i$ or $L_iR_i$ gets mapped to a path that loops once around $S^1_i$, starting and ending at $z_i^-$ or $z_i^+$ respectively. A nonzero $U_i$ loop that does not factor must be based at a vertex $\x$ with $\{i-1,i\} \subset \x$. This type of $U_i$ loop corresponds to a sum of two terms, each with a strand starting at $z_i^-$ or $z_i^+$ looping once around $S^1_i$ and an adjacent pair of dashed strands.  Finally, a $C_i$ generator in $\B(n,k,\Sc)$ is sent to a closed loop on the cylinder $\I \times S^1_i$. See Figure~\ref{fig:RLUCexamples} and Figure~\ref{fig:UexamplePart2} for an illustration of Definition~\ref{def:Phi}.

\begin{figure}
\includegraphics[scale=0.5]{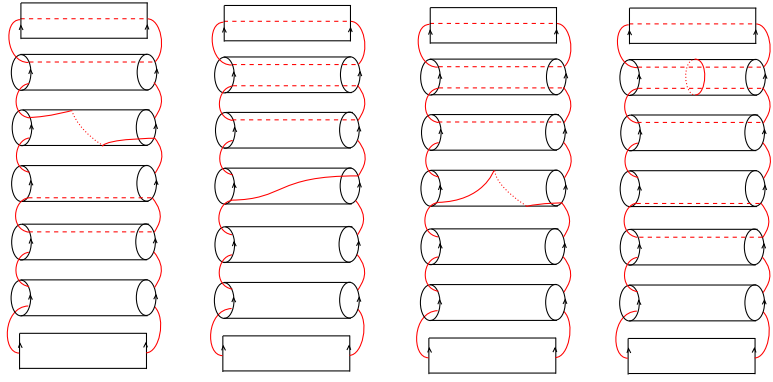}
\caption{From left to right: images under $\Phi$ of generators of $\B(5,3)$ starting at $\x = \{0,1,3\}$ and labeled $R_2$, $L_3$, $U_3$, and $C_1$ respectively.}
\label{fig:RLUCexamples}
\end{figure}

\begin{figure}
\includegraphics[scale=0.5]{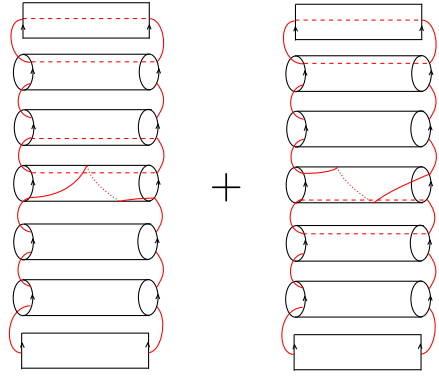}
\caption{Image under $\Phi$ of the generator of $\B(5,3)$ starting at $\x = \{0,2,3\}$ and labeled $U_3$.}
\label{fig:UexamplePart2}
\end{figure}

\begin{lemma}\label{lem:PhiWellDef}
$\Phi$ is well-defined and respects multiplication.
\end{lemma}

\begin{figure}
\includegraphics[scale=0.5]{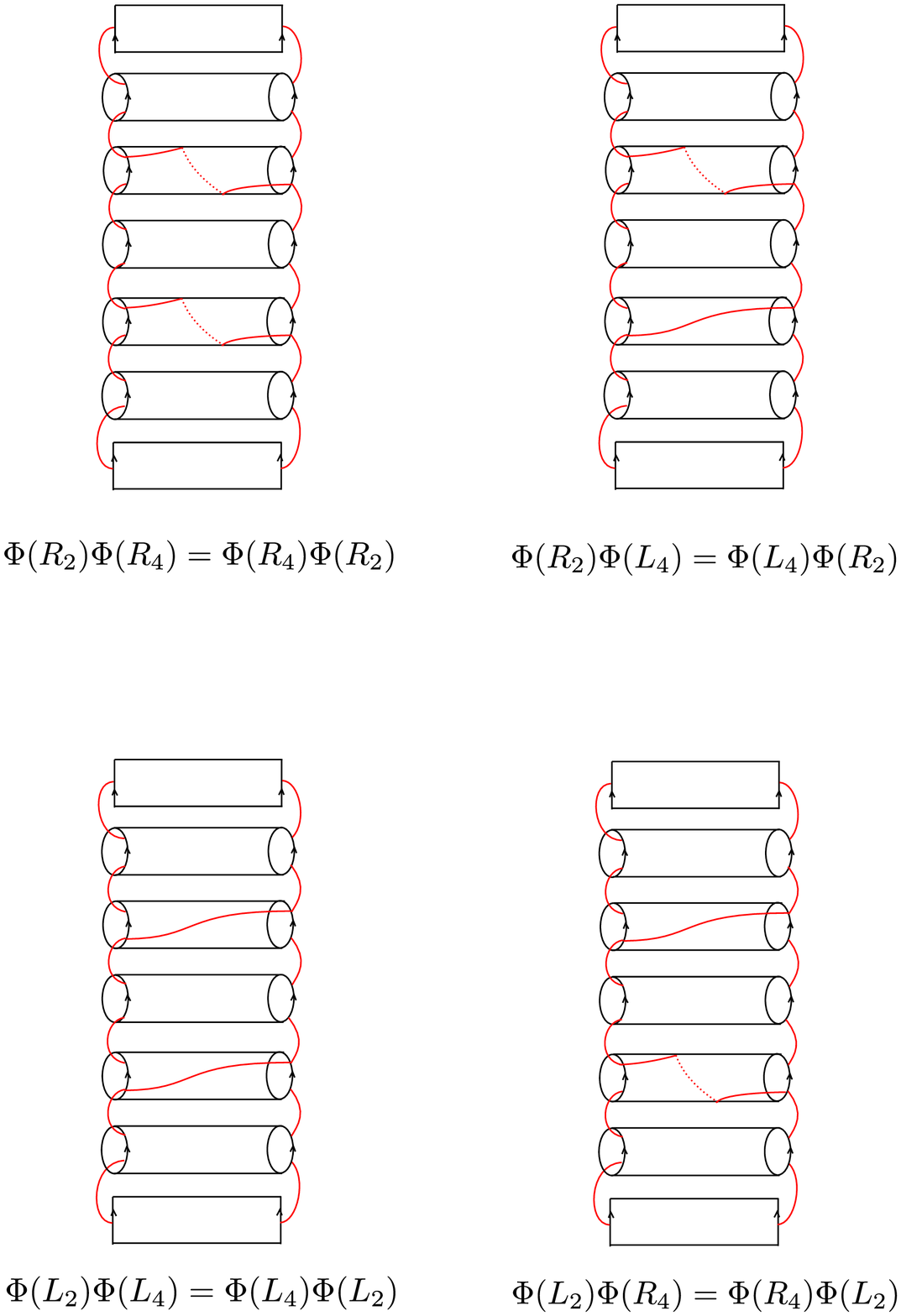}
\caption{Distant commutation relations in $\B(5,2)$.}
\label{fig:PhiCommutationRels}
\end{figure}

\begin{proof}
We must show that the relations given in Definition~\ref{def:ReviewGammaQuiverRels} are satisfied in $\sac nk\Sc$ after applying $\Phi$. The description below Definition \ref{def:Phi}, along with the examples, should make this lemma very plausible.  We carry through the algebraic checks one case at a time.  Recall that in the visualization of $\B(n,k,\Sc)$, line $i$ sits between regions $i-1$ and $i$, and similarly in $\sac nk\Sc$ the cylinder $\I \times S^1_i$ sits between matchings $(z^+_{i-1},z^-_i)$ and $(z^+_i,z^-_{i+1})$ (these can be viewed as matchings $i-1$ and $i$ respectively).  We will be using Lemma \ref{lem:concatenable} together with the convention of equation \eqref{eq:pq single column} throughout the proof.
\begin{itemize}
\item The ``$U$ vanishing relations'' $U_i = 0$ if $\gamma$ is a loop at a vertex $\x \in V(n,k)$ with $\x\cap\{i-1,i\}=\emptyset$:\\
These follow from item~\eqref{eq:PhiUVanishing} in Definition~\ref{def:Phi}.
\item The ``loop relations'' $R_i L_i = U_i$, $L_i R_i = U_i$:\\
In such a relation, let $\gamma$ be the edge labeled $U_i$. By items \eqref{eq:PhiDefRLU} or \eqref{eq:PhiDefLRU} of Definition~\ref{def:Phi}, $\Phi$ maps the relation to zero.
\item The ``distant commutation relations'' $R_i R_j = R_j R_i, L_i L_j = L_j L_i, R_i L_j = L_j R_i$ for $|i-j|>1$:\\
For a relation of the form $R_i R_j = R_j R_i$, both $\Phi(R_i) \Phi(R_j)$ and $\Phi(R_j) \Phi(R_i)$ are the basis element $\vv{1}{0}{i} \vv{1}{0}{j}$ by the formulas of Lemma \ref{lem:concatenable}. The other cases are similar; see Figure~\ref{fig:PhiCommutationRels}.
\item The ``two line pass'' relations $R_i R_{i+1} = 0, L_i L_{i-1} = 0$:\\
For $R_i R_{i+1} = 0$, we have 
\[
\Phi(R_i) \Phi(R_{i+1}) = \vv{1}{0}{i} \cdot \vv{1}{0}{i+1} = 0
\]
by condition~\eqref{it:p odd concat} of Lemma \ref{lem:concatenable}. The relation $L_i L_{i-1} = 0$ is similar, using condition~\eqref{it:q odd concat}.
\item The ``$U$ central relations'' part 1, $R_i U_j = U_j R_i, L_i U_j = U_j L_i$:\\
First consider a relation of the form $R_iU_j=U_jR_i$. We consider several cases; first assume $|i-j| > 1$. In this case we have $\x \cap \{j-1,j\} = (\x \setminus \{i-1\} \cup \{i\}) \cap \{j-1,j\}$.
\begin{itemize}
\item If $\x \cap \{j-1,j\} = 0$, then $\Phi(U_j)$ is zero on both sides of the relation.
\item If $\x \cap \{j-1,j\} = \{j-1\}$, then we have
\[
\Phi(R_i) \Phi(U_j) = \Phi(R_i) \Phi(R_j) \Phi(L_j) = \Phi(R_j) \Phi(L_j) \Phi(R_i) = \Phi(U_j) \Phi(R_i),
\]
using the ``distant commutation relations'' twice in the middle equality. The case where $\x \cap \{j-1,j\} = \{j\}$ is similar.
\item If $\x \cap \{j-1,j\} = \{j-1,j\}$, then by Lemma \ref{lem:concatenable} we have 
\[
\Phi(R_i) \Phi(U_j) = \vv{1}{0}{i} \vv{2}{0}{j} + \vv{1}{0}{i} \vv{0}{2}{j},
\]
and $\Phi(U_j) \Phi(R_i)$ gives the same result.
\end{itemize}
For the cases when $|i-j|\leq 1$, we note that the presence of $R_i$ as the only non-loop edge implies that $\x\cap\{i-1,i\}=\set{i-1}$.  First suppose $j = i-1$.  We consider two subcases:
\begin{itemize}
\item If $\x \cap \{i-2,i-1,i\} = \{i-1\}$, we have $\Phi(U_j) \Phi(R_i) = \Phi(L_{i-1}) \Phi(R_{i-1}) \Phi(R_i) = 0$ by the ``two-line pass'' relations. We also have $\Phi(R_i) \Phi(U_j) = 0$ because the second factor is zero.
\item If $\x \cap \{i-2,i-1,i\} = \{i-2,i-1\}$, we have 
\[
\Phi(U_j) \Phi(R_i) = \bigg[ \vv{2}{0}{i-1} + \vv{0}{2}{i-1} \bigg] \cdot \vv{1}{0}{i} = \vv{2}{0}{i-1} \vv{1}{0}{i} .
\]
by Lemma \ref{lem:concatenable} (the second summand is zero by condition~\eqref{it:q even concat} of that lemma). Meanwhile, we have
\[
\Phi(R_i) \Phi(U_j) = \vv{1}{0}{i} \cdot \vv{2}{0}{i-1} = \vv{2}{0}{i-1} \vv{1}{0}{i} ,
\]
so the relation holds.
\end{itemize}

If $j = i+1$, we follow a parallel argument:
\begin{itemize}
\item If $\x \cap \{i-1,i,i+1\} = \{i-1\}$, we have $\Phi(R_i) \Phi(U_j) = \Phi(R_i) \Phi(R_{i+1}) \Phi(L_{i+1}) = 0$ by the ``two-line pass'' relations. We also have $\Phi(U_j) \Phi(R_i) = 0$ because the first factor is zero.
\item If $\x \cap \{i-1,i,i+1\} = \{i-1,i+1\}$, we have 
\[
\Phi(R_i) \Phi(U_j) = \vv{1}{0}{i} \cdot \bigg[ \vv{2}{0}{i+1} + \vv{0}{2}{i+1} \bigg] = \vv{1}{0}{i} \vv{0}{2}{i+1}.
\]
by Lemma \ref{lem:concatenable} (the first summand is zero by condition~\eqref{it:p odd concat} of that lemma). Meanwhile, we have
\[
\Phi(U_j) \Phi(R_i) = \vv{0}{2}{i+1} \cdot \vv{1}{0}{i} = \vv{1}{0}{i} \vv{0}{2}{i+1},
\]
so the relation holds.
\end{itemize}
Finally, if $j = i$, then again since $\x \cap \{i-1,i\} = \{i-1\}$, we have $\Phi(R_i) \Phi(U_i) = \Phi(R_i) \Phi(L_i) \Phi(R_i) = \Phi(U_i) \Phi(R_i)$. The relations $L_i U_j = U_j L_i$ are analogous to $R_i U_j = U_j R_i$.
\item The ``$U$ central relations'' part 2, $U_i U_j = U_j U_i$:\\
Since both $U_i$ and $U_j$ are loops, we have $\x=\y$ in this case regardless of $i$ and $j$, so the meaning of $\Phi(U_i)$ and $\Phi(U_j)$ does not change depending on the order in which they are taken. We may also assume $i \neq j$.
\begin{itemize}
\item If $\x \cap \{i-1,i\} = \varnothing$ or $\x \cap \{j-1,j\} = \varnothing$, then either $\Phi(U_i)$ or $\Phi(U_j)$ is zero and the relation holds.
\item If $\x \cap \{i-1,i\} = \{i-1\}$, then 
\begin{align*}
\Phi(U_i) \Phi(U_j) &= \Phi(R_i) \Phi(L_i) \Phi(U_j) \\
&= \Phi(R_i) \Phi(U_j) \Phi(L_i) \\
&= \Phi(U_j) \Phi(R_i) \Phi(L_i) \\
&= \Phi(U_j) \Phi(U_i).
\end{align*}
A similar argument shows that the relation also holds if $\x \cap \{i-1,i\} = \{i\}$; by symmetry, it holds if $\x \cap \{j-1,j\} = \{j-1\}$ or $\x \cap \{j-1,j\} = \{j\}$.
\item If $\x \cap \{i-1,i\} = \{i-1,i\}$ and $\x \cap \{j-1,j\} = \{j-1,j\}$, then we consider subcases.
\begin{itemize}
\item
If $|i-j| > 1$, then $\Phi(U_i) \Phi(U_j)$ and $\Phi(U_j) \Phi(U_i)$ are both equal to
\[
\vv{2}{0}{i} \vv{2}{0}{j} + \vv{2}{0}{i} \vv{0}{2}{j} + \vv{0}{2}{i} \vv{2}{0}{j} + \vv{0}{2}{i} \vv{0}{2}{j}.
\]
\item If $j = i-1$, then
\[
\Phi(U_i) \Phi(U_j) = \bigg[\vv{2}{0}{i} + \vv{0}{2}{i} \bigg] \cdot \bigg[ \vv{2}{0}{i-1} + \vv{0}{2}{i-1} \bigg].
\]
We have 
\[
\vv{2}{0}{i} \cdot \vv{2}{0}{i-1} = \vv{2}{0}{i-1} \vv{2}{0}{i},
\]
\[
\vv{2}{0}{i} \cdot \vv{0}{2}{i-1} = 0,
\]
\[
\vv{0}{2}{i} \cdot \vv{2}{0}{i-1} = \vv{2}{0}{i-1} \vv{0}{2}{i},
\]
and 
\[
\vv{0}{2}{i} \cdot \vv{0}{2}{i-1} = \vv{0}{2}{i-1} \vv{0}{2}{i}.
\]
On the other hand,
\[
\Phi(U_j) \Phi(U_i) = \bigg[\vv{2}{0}{i-1} + \vv{0}{2}{i-1} \bigg] \cdot \bigg[ \vv{2}{0}{i} + \vv{0}{2}{i} \bigg].
\]
We have 
\[
\vv{2}{0}{i-1} \cdot \vv{2}{0}{i} = \vv{2}{0}{i-1} \vv{2}{0}{i},
\]
\[
\vv{2}{0}{i-1} \cdot \vv{0}{2}{i} = \vv{2}{0}{i-1} \vv{0}{2}{i},
\]
\[
\vv{0}{2}{i-1} \cdot \vv{2}{0}{i} = 0,
\]
and 
\[
\vv{0}{2}{i-1} \cdot \vv{0}{2}{i} = \vv{0}{2}{i-1} \vv{0}{2}{i}.
\]
Thus, $\Phi(U_i) \Phi(U_{i-1}) = \Phi(U_{i-1}) \Phi(U_i)$ when $\x \cap \{i-2,i-1,i\} = \{i-2,i-1,i\}$.
\item The case $j = i+1$ follows by symmetry.
\end{itemize}
\end{itemize}

\item The ``$C$ central relations'' $C_i A = A C_i$ for all generators $A$ labeled $R_j$, $L_j$, $U_j$, or $C_j$:\\
These relations hold in $\sac nk\Sc$ because multiplication in $\sac nk\Sc$ is defined using addition of the components $\vec{c}$ of generators $E(s,\vec{c})$, and addition is commutative. Visually, a closed loop may be isotoped to near the beginning or the end of a cylinder without changing the corresponding element of $\sac nk\Sc$.

\item The ``$C^2$ vanishing relations'' $C_i^2 = 0$:\\
These relations hold in $\sac nk\Sc$ because the product of $E(s,\vec{c})$ with $E(s',\pvec{c}')$ was defined to be zero if $(\vec{c}+\pvec{c}')(i) = 2$ for any $i$. Visually, the product results in a degenerate annulus.
\end{itemize}
\end{proof}

\begin{lemma}\label{lem:Phi respects differential}
The map $\Phi$ is a homomorphism of differential graded algebras.
\end{lemma}
\begin{proof} The reader may use the grading formulas of Definitions~\ref{def:ReviewOSzGradings} and \ref{def:gradings} to confirm that $\Phi$ does indeed preserve gradings.

For the differential, we first show $\de(\Phi(\gamma)) = 0$ for an edge $\gamma$ with label $R_i$, $L_i$, or $U_i$ of $\B(n,k,\Sc)$. In all of these cases, the array for $\Phi(\gamma)$ has no monomial of $C_j$ variables out front, so $\de^c(\Phi(\gamma))=0$. The array also has $p_j=q_j=0$ for all $j\neq i$, so that $M_j=m_j=0$ in the language of Lemma \ref{lem:general pq differential}. Thus, $\de^0_j(\Phi(\gamma))=0$ for $j \neq i$.

To compute $\de^0_i(\Phi(\gamma))$, let $\x$ denote the starting vertex of $\gamma$. First note that $\de^0_i(\Phi(\gamma))=0$ if $\{i-1,i\}\not\subset\x$. This observation takes care of all cases of $\gamma$ under consideration except the case when $\gamma$ is labelled by $U_i$ and $\{i-1,i\}\subset\x$, for which we have $\Phi(\gamma)=\vv{2}{0}{i} + \vv{0}{2}{i}$.  By formula~\eqref{eq:general pq differential good idemp} in Lemma~\ref{lem:general pq differential}, we have $\de^0_i\left(\vv{2}{0}{i}\right) = \vv{1}{1}{i} = \de^0_i\left(\vv{0}{2}{i}\right)$. Since we are in characteristic $2$, it follows that $\de^0_i(\Phi(\gamma)) = 0$, and so we may conclude that $\de(\Phi(\gamma)) = 0$ in all of these cases.

Now consider a loop $\gamma$ labelled by $C_i$ in $\B(n,k,\Sc)$ at vertex $\x$.  The array $\Phi(\gamma)=C_i$ has all $p_j=q_j=0$, so Lemma \ref{lem:general pq differential} again ensures that $\de^0(\Phi(\gamma))=0$.  Meanwhile, Lemma \ref{lem:general pq C-differential} allows us to compute $\de^c(\Phi(\gamma))$ case-by-case by analyzing $\delta_{i-1}$ and $\epsilon_i$ in equation \eqref{eq:general pq C-differential zeroes}:
\begin{itemize}
\item If $\x \cap \{i-1,i\} = \varnothing$, then $\de^c(\Phi(C_i)) = 0$.
\item If $\x \cap \{i-1,i\} = \{i-1\}$, then $\de^c(\Phi(C_i)) = \vv{2}{0}{i}$.
\item If $\x \cap \{i-1,i\} = \{i\}$, then $\de^c(\Phi(C_i)) = \vv{0}{2}{i}$.
\item If $\x \cap \{i-1,i\} = \{i-1,i\}$, then $\de^c(\Phi(C_i)) = \vv{2}{0}{i} + \vv{0}{2}{i}$.
\end{itemize}
Using Definition \ref{def:Phi} in each case, we have $\de^c(\Phi(\gamma)) = \Phi(\gamma')$ for the loop $\gamma'$ at $\x$ labelled by $U_i$ instead of $C_i$ (note that $\gamma' = 0$ in $\B(n,k,\Sc)$ if $\x \cap \{i-1,i\} = \varnothing$). These four cases are illustrated in Figure~\ref{fig:DifferentialOfCExamples}.

We have shown that $\de(\Phi(\gamma)) = \Phi(\de(\gamma))$ for all edges $\gamma$ in the quiver $\Gamma(n,k,\Sc)$. It follows from the Leibniz rule that $\de(\Phi(\gamma)) = \Phi(\de(\gamma))$ for all $\gamma \in \B(n,k,\Sc)$. Thus, $\Phi$ is a homomorphism of dg algebras.
\end{proof}

\begin{figure}
\includegraphics[scale=0.5]{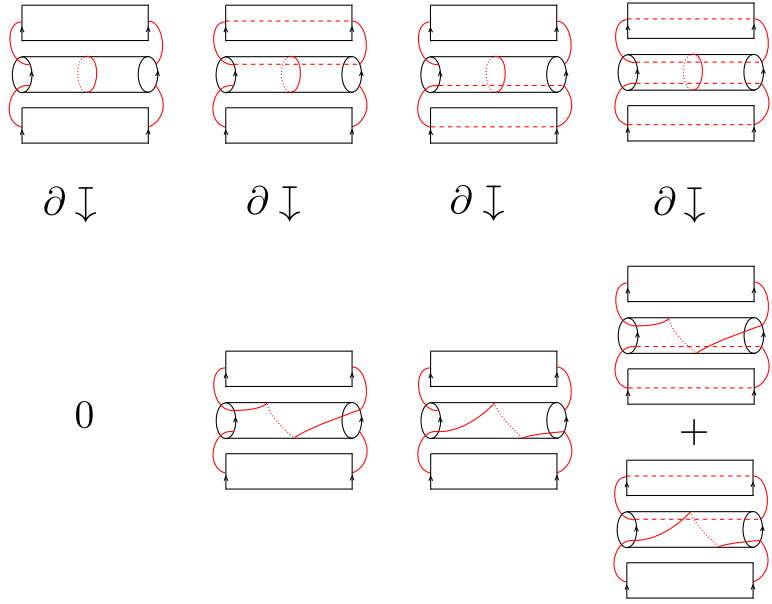}
\caption{Differential of $\Phi(C_i)$ in various cases.}
\label{fig:DifferentialOfCExamples}
\end{figure}

Recall that $\B(n,k,\Sc)$ is a dg algebra over $\F_2[U_1,\ldots,U_n]^{V(n,k)}$, so we have a dg ring homomorphism from $\F_2[U_1,\ldots,U_n]^{V(n,k)}$ to $\B(n,k,\Sc)$.
The natural inclusion $\Ib(n,k) \to \F_2[U_1, \ldots, U_n]^{V(n,k)}$ recovers the usual $\Ib(n,k)$-algebra structure of $\B(n,k,\Sc)$.
We now have another dg ring homomorphism $\Phi: \B(n,k,\Sc) \to \sac nk\Sc$, and we can compose to give $\sac nk\Sc$ the structure of a dg algebra over $\F_2[U_1,\ldots,U_n]^{V(n,k)}$, compatible with the dg algebra structure over $\Ib(n,k)$. Tautologically, $\Phi$ is a homomorphism of dg $\F_2[U_1,\ldots,U_n]^{V(n,k)}$-algebras. One can check that the natural map
\[
\F_2[U_1,\ldots,U_n] \to \F_2[U_1,\ldots,U_n]^{V(n,k)} \to \sac nk\Sc
\]
sends each variable $U_i$ to a central element of $\sac nk\Sc$, where the first map in the composition is the inclusion of constant functions.

\begin{lemma}\label{lem:StrandsHomologyBasisIdentPolyLinear}
For $\x, \y \in V(n,k)$ not far, the isomorphism
\[
\psi \colon \Jb_\x H_*(\sa nk) \Jb_\y \to \F_2[U_1,\ldots,U_n]/(p_G)
\]
given in equation~\eqref{eq:general homology of JAJ} of Corollary~\ref{cor:general homology of JAJ} is linear over $\F_2[U_1,\ldots,U_n]$.
\end{lemma}

\begin{proof}
The action of $U_i$ on a basis element of $\Jb_{\x} \sa nk \Jb_{\y}$ is given by multiplication (on either side) by $\Phi(\gamma)$ where $\gamma$ is a loop in $\Gamma(n,k,\Sc)$ labeled $U_i$. For concreteness, let $\gamma$ be based at $\x$, so that the action of $U_i$ is left multiplication by $\Phi(\gamma)$.

If $\x \cap \{i-1,i\} = \varnothing$, then $\gamma = 0$ so $\Phi(\gamma) = 0$. Correspondingly, line $i$ is contained in a generating interval of length $1$, so the action of $U_i$ on $\F_2[U_1,\ldots,U_n]/(p_G)$ is also zero.

If $\x \cap \{i-1,i\} = \{i-1\}$, then $\Phi(\gamma) = \vv{2}{0}{i}$. We cannot have $v_i(\x,\y) = -1$. If $v_i(\x,\y) = 1$, then line $i$ is crossed, and by Lemma~\ref{lem:concatenable}, $\Phi(\gamma)$ acts on a homology basis element of Corollary~\ref{cor:HomologyBasisAsStrandsGens} by increasing the value of $r_i$ by one. If $v_i(\x,\y) = 0$, then line $i$ is not crossed, but the term $\vv{0}{2r_i}{i}$ in this homology basis element is zero because $i \notin \x$ (see Lemma~\ref{lem:piqi}). Thus, $\Phi(\gamma)$ acts by increasing the value of $r_i$ by one in this case too. The argument when $\x \cap \{i-1,i\} = \{i\}$ is similar.

Finally, if $\{i-1,i\} \subset \x$, we have $\Phi(\gamma) = \vv{2}{0}{i} + \vv{0}{2}{i}$. If $v_i(\x,\y) = -1$, then $\vv{2}{0}{i}$ becomes zero when multiplied by a basis element of $\Jb_\x H_*(\sa nk) \Jb_\y$. If $v_i(\x,\y) = 1$, then $\vv{0}{2}{i}$ becomes zero when multiplied by a basis element of $\Jb_{\x} H_*(\sa nk) \Jb_{\y}$. If $v_i(\x,\y)=0$ (so $i\notin\CL{\x,\y}$) and $r_i \neq 0$, we see a product of the form
\[
\left( \vv{2}{0}{i} + \vv{0}{2}{i} \right) \cdot \left( \vv{2r_i}{0}{i} + \vv{0}{2r_i}{i} \right).
\]
The cross-terms $\vv{2}{0}{i}\cdot\vv{0}{2r_i}{i}$ and $\vv{0}{2}{i}\cdot\vv{2r_i}{0}{i}$ are both zero due to degenerate bigons (see Lemma \ref{lem:concatenable}), while the remaining terms give $\vv{2(r_i+1)}{0}{i}+\vv{0}{2(r_i+1)}{i}$. The case when $r_i = 0$ is left to the reader. In all three cases, we see that multiplication by $\Phi(\gamma)$ on the left has the effect of increasing the value of $r_i$ by one in a basis element from Corollary~\ref{cor:HomologyBasisAsStrandsGens}.
\end{proof}

\begin{lemma}\label{lem:MainTheoremHelper}
Suppose $\x,\y \in V(n,k)$ are not far. If $\gamma_{\x,\y}$ is the element of $\B(n,k)$ represented by the path from Definition \ref{def:ReviewGammaXY}, then $\Phi(\gamma_{\x,\y})$ is the basis element from Corollary~\ref{cor:HomologyBasisAsStrandsGens} with all $r_i$ equal to zero.
\end{lemma}

\begin{proof}
We use the recursive definition of $\gamma_{\x,\y}$ and induct on $k - |\x \cap \y|$. When this quantity is zero, we have $\x = \y$ and $\gamma_{\x,\y}$ is the empty path. Thus, $\Phi(\x,\y) = \Jb_{\x}$, which is the strands element corresponding to $1 \in \mathbb{F}_2[U_1,\ldots,U_n]/(p_G)$ under Corollary~\ref{cor:HomologyBasisAsStrandsGens}. 

Now assume that $\Phi(\gamma_{\x',\y'})$ is as described whenever $k - |\x' \cap \y'| < k - |\x \cap \y|$. If we have $x_a < y_a$ for some $a \in [1,k]$, let $a$ be the maximal such index. Let $\x' = (\x \setminus x_a) \cup \{y_a\}$. By Definition \ref{def:ReviewGammaXY}, we have $\gamma_{\x,\y} = \gamma \cdot \gamma_{\x',\y}$ where $\gamma$ is the unique edge from $\x$ to $\x'$ labeled $R_{x_a + 1}$.
By item~\eqref{it:PhiOfRi} in Definition~\ref{def:Phi}, we have $\Phi(\gamma) = \vv{1}{0}{x_a + 1}$, and by induction $\Phi(\gamma_{\x',\y})$ is the strands element whose array has vectors $\vv{1}{0}{i}$ for $i$ with $v_i(\x',\y) = 1$, $\vv{0}{1}{i}$ for $i$ with $v_i(\x',\y) = -1$, and no other nontrivial factors.

We can use Lemma~\ref{lem:concatenable} to show that the product of these strands elements has vector $\vv{1}{0}{i}$ for $i$ with $v_i(\x,\y) = 1$, $\vv{0}{1}{i}$ for $i$ with $v_i(\x,\y) = -1$, and no other nontrivial factors. We first show that the product is nonzero. The main thing to check is condition~\eqref{it:p odd concat}; conditions~\eqref{it:p even concat}--\eqref{it:no double loops} are tautological. We only have $p_i$ odd for $i = x_a + 1$, so we want to show that $p'_{x_a + 2} = 0$. Indeed, if $p'_{x_a + 2}$ is nonzero, then we have $x_{a+1} = x_a + 1 = y_a < y_{a+1}$, contradicting that $a$ is the maximal index with $x_a < y_a$. Thus, condition~\eqref{it:p odd concat} holds, so the product under consideration is nonzero.

It follows that the product is given by the formula in Lemma~\ref{lem:concatenable}; we must show this product is the desired strands element. We have $r_{x_a + 1} = p_{x_a + 1} + q'_{x_a + 1}$ and $p_{x_a + 1} = 1$; we claim that $q'_{x_a + 1} = 0$. Indeed, if it is nonzero then $q'_{x_a + 1} = 1$ and $p'_{x_a + 1} = 0$, so by Lemma~\ref{lem:cl}, line $x_a + 1$ is crossed from $\x'$ to $\y$ and we have $v_{x_a + 1}(\x',\y) \neq 0$. But the equality $x'_a = y_a$ implies that $v_{x_a + 1}(\x',\y) = 0$, since 
\[
|\x' \cap [x'_a + 1, n]| = k - a = |\y \cap [y_a + 1, n]|,
\]
so we have a contradiction. It follows that $r_{x_a + 1} = 1$. We also have $s_{x_a + 1} = 0$ because $p_{x_a + 1}$ is odd, so vector $i$ in the product is equal to $\vv{1}{0}{i}$ when $i = x_a + 1$. Note that $v_{x_a + 1}(\x,\y) = 1$.

Now consider $i \neq x_a + 1$. We have $p_i = q_i = 0$, so Lemma~\ref{lem:concatenable} gives us $r_i = p_i + p'_i = p'_i$ and $s_i = q_i + q'_i = q'_i$. Thus, $r_i = 1$ if and only if $v_i(\x',\y) = 1$ (in which case $s_i = 0$) and $s_i = 1$ if and only if $v_i(\x',\y) = -1$ (in which case $r_i = 0$). For $i \neq x_a + 1$, we have $|\x \cap [i,n]| = |\x' \cap [i,n]|$, so $v_i(\x',\y) = v_i(\x,\y)$, proving that $\Phi(\gamma) \Phi(\gamma_{\x',\y})$ is the strands element described above. 

The case when $x_a \geq y_a$ for all $a \in [1,k]$ and $x_a > y_a$ for some minimal index $a$ is analogous. By induction, $\Phi(\gamma_{\x,\y})$ is the strands element whose vector has $\vv{1}{0}{i}$ for $i$ with $v_i(\x,\y) = 1$, $\vv{0}{1}{i}$ for $i$ with $v_i(\x,\y) = -1$, and no other nontrivial factors, proving the lemma.
\end{proof}

\begin{theorem}\label{thm:Phi is a quasi-iso}
The map 
\[
\Phi: \B(n,k) \to \A(n,k)
\]
is a quasi-isomorphism of dg algebras over $\F_2[U_1,\ldots,U_n]^{V(n,k)}$.
\end{theorem}

\begin{proof}
The algebra $\B(n,k) = \B(n,k,\varnothing)$ has no differential, so it is equal to its homology.
Given $\x$ and $\y$ in $V(n,k)$ that are not far, let $U_1^{r_1} \cdots U_n^{r_n} \gamma_{\x,\y}$ be the basis element of $\Ib_{\x} \B(n,k) \Ib_{\y}$ given by Theorem \ref{thm:ReviewOSzHomology}. 

Since $\Phi$ is linear over $\F_2[U_1,\ldots,U_n]$, $\Phi$ sends this element to $U_1^{r_1} \cdots U_n^{r_n} \Phi(\gamma_{\x,\y})$. By Lemmas~\ref{lem:StrandsHomologyBasisIdentPolyLinear} and \ref{lem:MainTheoremHelper}, the basis element for $\Jb_\x \A(n,k) \Jb_\y$ corresponding to $U_1^{r_1} \cdots U_n^{r_n}$ is also $U_1^{r_1} \cdots U_n^{r_n} \Phi(\gamma_{\x,\y})$.

Since $\Phi$ sends a basis for the homology of $\Ib_{\x} \B(n,k) \Ib_{\y}$ to a basis for the homology of $\Jb_{\x} \sa nk \Jb_{\y}$, $\Phi$ induces an isomorphism on homology, so $\Phi$ is a quasi-isomorphism.
\end{proof}

\subsection{Homology of the strands algebra: general case}

Finally, we compute the homology of $\sac nk\Sc$ and show that $\Phi$ is a quasi-isomorphism for general $\Sc$.

\begin{theorem}\label{thm:FinalQIThm}
The map $\Phi: \B(n,k,\Sc) \to \sac nk\Sc$ is a quasi-isomorphism.
\end{theorem}

\begin{proof}
We will induct on $|\Sc|$. The base case $|\Sc| = 0$ follows from Theorem~\ref{thm:Phi is a quasi-iso}.

For the inductive step, write $\Sc = \{i_1,\ldots,i_l\}$, and let $\Sc' = \Sc \setminus \{i_l\}$. As a chain complex, $\A(n,k,\Sc)$ has a two-step filtration by powers of the variable $C_{i_l}$. Since we may write this variable $C_{i_l}$ as $\Phi(C_{i_l})$ and we have
\[
\de(\Phi(C_{i_l})) = \Phi(\de(C_{i_l})) = \Phi(U_{i_l}),
\]
$\A(n,k,\Sc)$ is isomorphic to the mapping cone on the endomorphism $\Phi(U_{i_l})$ of $\A(n,k,\Sc')$.  The mapping cone gives us a long exact sequence on homology from which we can extract a short exact sequence by taking kernels and cokernels (as in the proof of \cite[\lemOSzGenIntHomology]{MMW1}). The short exact sequence must split over $\F_2$, giving
\[H_*(\A(n,k,\Sc)) \cong \frac{H_*(\A(n,k,\Sc'))}{\im \Phi([U_{i_l}])} \oplus \ker \Phi([U_{i_l}]).\]

By induction, $\Phi$ identifies these summands with the cokernel and kernel of $[U_{i_l}]$ acting on $H_*(\B(n,k,\Sc'))$.
The chain complex $\B(n,k,\Sc)$ also has a two-step filtration by powers of the variable $C_{i_l}$. Since the corresponding short exact sequence for $\B(n,k,\Sc)$ splits in the same way to give
\[
H_*(\B(n,k,\Sc)) \cong \frac{H_*(\B(n,k,\Sc'))}{\im([U_{i_l}])} \oplus \ker ([U_{i_l}]),
\]
$\Phi$ is a quasi-isomorphism from $\B(n,k,\Sc)$ to $\A(n,k,\Sc)$.
\end{proof}

By \cite[\defDGCatQI]{MMW1}, we can restrict quasi-isomorphisms to full dg subcategories, so we get the following corollary.
\begin{corollary}\label{cor:TruncatedQIThm}
The map $\Phi$ restricts to quasi-isomorphisms from $\B_r(n,k,\Sc)$ to ${\mc A}_r(n,k,\Sc)$, from $\B_l(n,k,\Sc)$ to ${\mc A}_l(n,k,\Sc)$, and from $\B'(n,k,\Sc)$ to ${\mc A}'(n,k,\Sc)$.
\end{corollary}

Since formality is preserved by quasi-isomorphisms, we deduce the next corollary from the results of \cite{MMW1}.
\begin{corollary}[see Corollary~\ref{cor:IntroFormality}]\label{cor:StrandsFormality}
The dg algebra $\A(n,k,\Sc)$ is formal if and only if:
\begin{itemize}
\item $\Sc = \varnothing$, or
\item $k \in \{0,n,n+1\}$.
\end{itemize}
The dg algebra $\A_r(n,k,\Sc)$ is formal if and only if:
\begin{itemize}
\item $\Sc = \varnothing$ or $\{1\}$, or
\item $k \in \{0,n\}$, or
\item $k = n-1$ and $1 \in \Sc$.
\end{itemize}
The dg algebra $\A_l(n,k,\Sc)$ is formal if and only if:
\begin{itemize}
\item $\Sc = \varnothing$ or $\{n\}$, or
\item $k \in \{0,n\}$, or
\item $k = n-1$ and $n \in \Sc$.
\end{itemize}
The dg algebra $\A'(n,k,\Sc)$ is formal if and only if:
\begin{itemize}
\item $\Sc = \varnothing$, $\{1\}$, $\{n\}$, or $\{1,n\}$, or
\item $k \in \{0,n-1\}$, or
\item $k = n-2$ and $\{1,n\} \subset \Sc$.
\end{itemize}
\end{corollary}

\begin{proof}
These results follow from Theorem~\ref{thm:FinalQIThm} and Corollary~\ref{cor:TruncatedQIThm}, as well as \cite[Theorems \thmUntruncatedFormality, \thmRightTruncationFormality, \thmLeftTruncationFormality{}, and \thmDoubleTruncationFormality]{MMW1}.
\end{proof}

\subsection{Symmetries}\label{sec:PhiSymmetries}
Here we show that the quasi-isomorphism $\Phi$ of Theorem~\ref{thm:FinalQIThm} intertwines the symmetries $\rho$ and $o$ from \cite[\secOSzSymmetries]{MMW1} with the symmetries of the same name from Section~\ref{sec:StrandSymmetries}. We start with the symmetry $\rho$.

\begin{proposition}\label{prop:PhiCompatibleWithRho}
The diagram
\[
\xymatrix{
\B(n,k,\Sc) \ar[r]^{\Phi} \ar[d]_{\rho} & \sac nk\Sc \ar[d]^{\rho} \\
\B(n,k,\rho(\Sc)) \ar[r]_{\Phi} & \sac nk{\rho(\Sc)}
}
\]
of morphisms of dg algebras commutes.
\end{proposition}

\begin{proof}
It suffices to check commutativity on the quiver generators of $\B(n,k,\Sc)$. For an edge $\gamma$ of $\Gamma(n,k,\Sc)$ from $\x$ to $\y$ with label $R_i$, we have $\Phi(\gamma) = \vv{1}{0}{i}$, so $\rho(\Phi(\gamma)) = \vv{0}{1}{n+1-i}$. On the other hand, $\rho(\gamma)$ is the edge of $\Gamma(n,k,\rho(\Sc))$ from $\rho(\x)$ to $\rho(\y)$ (with label $L_{n+1-i}$), so we have $\Phi(\rho(\gamma)) = \vv{0}{1}{n+1-i}$ as well. The argument for edges labeled $L_i$ is similar. 

For an edge $\gamma$ of $\Gamma(n,k,\Sc)$ from $\x$ to $\x$ with label $U_i$, we may assume that $\x \cap \{i-1,i\} \neq \varnothing$, so that $\gamma$ represents a nonzero generator. If $\x \cap \{i-1,i\} = \{i-1\}$, then $\Phi(U_i) = \Phi(R_i) \Phi(L_i)$, so commutativity follows from the above paragraph. If $\x \cap \{i-1,i\} = \{i\}$, the argument is similar. If $\{i-1,i\} \subset \x$, then $\Phi(\gamma) = \vv{2}{0}{i}+\vv{0}{2}{i}$, so $\rho(\Phi(\gamma)) = \vv{0}{2}{n+1-i} + \vv{2}{0}{n+1-i}$. On the other hand, $\rho(\gamma)$ is the edge from $\rho(\x)$ to $\rho(\x)$ with label $U_{n+1-i}$, and we have $\{n-i,n+1-i\} \subset \rho(\x)$. Thus, $\Phi(\rho(\gamma)) = \vv{2}{0}{n+1-i} + \vv{0}{2}{n+1-i}$, so $\Phi(\rho(\gamma)) = \rho(\Phi(\gamma))$.

Finally, for an edge $\gamma$ of $\Gamma(n,k,\Sc)$ from $\x$ to $\x$ with label $C_i$, we have $\rho(\Phi(\gamma)) = C_{n+1-i} = \Phi(\rho(\gamma))$. Thus, the square commutes.
\end{proof}

Next we consider the symmetry $o$.
\begin{proposition}
\[
\xymatrix{
\B(n,k,\Sc) \ar[r]^{\Phi} \ar[d]_{o} & \sac nk\Sc \ar[d]^{o} \\
\B(n,k,\Sc)^{\op} \ar[r]_{\Phi} & {\sac nk\Sc}^{\op}
}
\]
of morphisms of dg algebras commutes.
\end{proposition}

\begin{proof}
As in Proposition~\ref{prop:PhiCompatibleWithRho}, we check compatibility on the quiver generators of $\B(n,k,\Sc)$. For an edge $\gamma$ of $\Gamma(n,k,\Sc)$ from $\x$ to $\y$ with label $R_i$, we have $o(\Phi(\gamma)) = \vv{0}{1}{i}$. On the other hand, $o(\gamma)$ is the edge of $\Gamma(n,k,\Sc)^{\op}$ from $\x$ to $\y$ (with label $L_i$), so we have $\Phi(o(\gamma)) = \vv{0}{1}{i}$ as well. The argument for edges labeled $L_i$ is similar.

For an edge $\gamma$ of $\Gamma(n,k,\Sc)$ from $\x$ to $\x$ with label $U_i$, we may again assume that $\x \cap \{i-1,i\} \neq \varnothing$. If $\x \cap \{i-1,i\} = \{i-1\}$ or $\{i\}$, then commutativity follows from the above paragraph. If $\{i-1,i\} \subset \x$, then $o(\Phi(\gamma)) = \vv{0}{2}{i} + \vv{2}{0}{i}$. On the other hand, $o(\gamma) = \gamma$, so $\Phi(o(\gamma)) = \Phi(\gamma) = \vv{2}{0}{i} + \vv{0}{2}{i}$ and we have $\Phi(o(\gamma)) = o(\Phi(\gamma))$.

Finally, for an edge $\gamma$ of $\Gamma(n,k,\Sc)$ from $\x$ to $\x$ with label $C_i$, we have $o(\Phi(\gamma)) = C_i = \Phi(o(\gamma))$. Thus, the square commutes.
\end{proof}

\bibliographystyle{alpha}
\bibliography{biblio}

\end{document}